\documentclass{amsart}

\usepackage{amsmath,amsthm,amssymb,amscd}
\usepackage{mathrsfs,eucal,bm,color,graphicx}
\usepackage[T1]{fontenc}
\usepackage{multirow}
\usepackage{subdepth}
\usepackage[colorlinks=true, linkcolor= cyan, citecolor=magenta, urlcolor=blue]{hyperref}

\renewcommand{\arraystretch}{1.25}
\theoremstyle{plain}
\newtheorem{thm}{Theorem}[section]
\newtheorem{prop}[thm]{Proposition}
\newtheorem{coro}[thm]{Corollary}
\newtheorem{lem}[thm]{Lemma}

\newtheorem*{thm-LG}{\hyperref[thm:LG]{Theorem~\ref*{thm:LG}}}
\theoremstyle{definition}
\newtheorem{defn}[thm]{Definition}
\theoremstyle{remark}
\newtheorem*{rem}{Remark}
\newtheorem{exa}{Example}

\DeclareMathAlphabet{\mathsfit}{\encodingdefault}{\sfdefault}{m}{sl}
\SetMathAlphabet{\mathsfit}{bold}{\encodingdefault}{\sfdefault}{bx}{sl}

\newcommand{\medpmatrix}[2][0.85]{\!\!\!\!\scalebox{#1}{
  \setlength{\arraycolsep}{2.5pt}
  \renewcommand{\arraystretch}{0.85}$\begin{pmatrix}#2\end{pmatrix}$}}
\newcommand{\medbmatrix}[2][0.85]{\!\!\!\!\scalebox{#1}{
  \setlength{\arraycolsep}{2.5pt}
  \renewcommand{\arraystretch}{0.85}$\begin{bmatrix}#2\end{bmatrix}$}}

\newcommand{\ts}{\hspace{.11111em}}
\newcommand{\tts}{\hspace{.05555em}}
\newcommand{\nts}{\hspace{-.11111em}}
\newcommand{\ntts}{\hspace{-.05555em}}

\newcommand{\bbR}{{\mathbb{R}}}
\newcommand{\bbC}{{\mathbb{C}}}
\newcommand{\bbZ}{{\mathbb{Z}}}
\newcommand{\bbN}{{\mathbb{N}}}
\newcommand{\bbQ}{{\mathbb{Q}}}
\newcommand{\scrV}{{\mathscr{V}}}
\newcommand{\scrF}{{\mathscr{F}}}
\newcommand{\scrP}{{\mathscr{P}}}
\newcommand{\scrA}{{\mathscr{A}}}
\newcommand{\scrB}{{\mathscr{B}}}
\newcommand{\scrS}{{\mathscr{S}}}
\newcommand{\calP}{{\mathcal{P}}}
\newcommand{\calA}{{\mathcal{A}}}
\newcommand{\calD}{{\mathcal{D}}}
\newcommand{\calR}{{\mathcal{R}}}
\newcommand{\calS}{{\mathcal{S}}}
\newcommand{\rmd}{{\mathrm{d}}}
\newcommand{\rmf}{{\mathrm{f}}}
\newcommand{\bsa}{{\bm{a}}}
\newcommand{\bsb}{{\bm{b}}}
\newcommand{\bsc}{{\bm{c}}}
\newcommand{\bsn}{{\bm{n}}}
\newcommand{\bsv}{{\bm{v}}}
\newcommand{\bsw}{{\bm{w}}}
\newcommand{\bsx}{{\bm{x}}}
\newcommand{\bsy}{{\bm{y}}}
\newcommand{\bsz}{{\bm{z}}}
\newcommand{\bszeta}{{\bm{\zeta}}}
\newcommand{\bsxi}{{\bm{\xi}}}
\newcommand{\bseta}{{\bm{\eta}}}

\newcommand{\bsE}{{\bm{E}}}
\newcommand{\bsH}{{\bm{H}}}
\newcommand{\bsI}{{\bm{I}}}
\newcommand{\bsQ}{{\bm{Q}}}
\newcommand{\bsM}{{\bm{M}}}
\newcommand{\bsG}{{\bm{G}}}
\newcommand{\bsV}{{\bm{V}}}
\newcommand{\bsF}{{\bm{F}}}
\newcommand{\bsD}{{\bm{D}}}
\newcommand{\bsR}{{\bm{R}}}
\newcommand{\bsS}{{\bm{S}}}
\newcommand{\bsA}{{\bm{A}}}
\newcommand{\bsP}{{\bm{P}}}
\newcommand{\bsJ}{{\bm{J}}}

\renewcommand{\Re}{\mathfrak{Re}}
\renewcommand{\Im}{\mathfrak{Im}}

\newcommand{\trans}{{\scriptscriptstyle \mathsfit{T}}}
\newcommand{\herm}{{\scriptscriptstyle \mathsfit{H}}}

\newcommand{\Euc}{\bsE}
\newcommand{\Hyp}{\bsH}

\newcommand{\Gaussian}{\bbZ[\ts i \ts]}
\newcommand{\bbZfour}{\mathbb{Z}_{\ntts/4\mathbb{Z}}\tts}
\newcommand{\bbZeight}{\mathbb{Z}_{\ntts/8\mathbb{Z}}\tts}
\newcommand{\bbZp}{\mathbb{Z}_{\ntts/p\mathbb{Z}}\tts}

\newcommand{\Mob}{M\nts\ddot{o}\tts b}
\newcommand{\SL}{S \ntts L}
\newcommand{\PSL}{P \ntts S \ntts L}
\newcommand{\SO}{SO}

\newcommand{\Platonic}{{\calP}}
\newcommand{\Apollonian}{{\calA}}
\newcommand{\dualApollonian}{{\calD}}

\newcommand{\SODelta}{SO_{\nts\!\varDelta}}
\newcommand{\SOF}{SO_{\ntts F}}
\newcommand{\OF}{O_{\ntts F}}

\newcommand{\SOhatF}{SO_{\ntts \hat F}}

\newcommand{\bsQSigma}{\bsQ_{\nts\varSigma}}
\newcommand{\bsQW}{\bsQ_{\ntts W}}
\newcommand{\bsQF}{\bsQ_{\! F}}
\newcommand{\bsGSigmaF}{\bsG_{\nts \varSigma,F}}

\begin{document}

%%%%%
% Front Matters

\title[Integral Orthoplicial Apollonian Sphere Packings and the Local-Global Principle]{The Local-Global Principle for Integral Bends\\in Orthoplicial Apollonian Sphere Packings}
\author[Nakamura]{Kei Nakamura}
\address{Department of Mathematics\\ University of California \\ Davis, CA 95616}
\email{knakamura@math.ucdavis.edu}

\begin{abstract}
We introduce an \emph{orthoplicial} Apollonian sphere packing, which is a sphere packing obtained by successively inverting a configuration of 8 spheres with 4-orthplicial tangency graph. We will show that there are such packings in which the bends of all constituent spheres are integral, and establish the asymptotic local-global principle for the set of bends in these packings.
\end{abstract}

\maketitle
%\tableofcontents

%%%%%
% Main Matters

%%%
\section{Introduction} \label{sec:Introduction}

In this article, we introduce a new family of sphere packings in $\bbR^3$, which we call \emph{orthoplicial Apollonian sphere packings}; these packings are obtained by successively inverting a configuration of eight spheres with 4-orthplicial tangency graph. We show that there are such packings in which all constituent spheres have integral bends (oriented curvature), and establish the asymptotic local-global principle for the set of bends appearing in such packings: sufficiently large integer $n$ is the bend of some sphere in the packing, provided that $n$ avoids the local obstructions.

\subsection{Apollonian Packings} \label{ssec:AP}

The family of sphere packings that we work with in this article is a generalization of a few known families of circle/sphere packings. Let us briefly describe these packings for a perspective.

Let us first recall classical Apollonian circle packings. Take a configuration of four pairwise tangent circles in $\bbR^2$, having the tetrahedral tangency graph. For any sub-configuration of three circles, there exists a unique dual circle orthogonal to them, and inverting the whole configuration along the dual circle yields a new configuration of four pairwise tangent circles. Indefinitely continuing this process, as shown in \hyperref[fig:TACP]{Figure~\ref*{fig:TACP}}, we obtain a \emph{classical/tetrahedral} Apollonian circle packing.%; here, the adjective ``tetrahedral'' refers to the \emph{tetrahedral} tangency graph of the initial configuration of four circles.
\begin{figure}[h]
\vspace{-0.5mm}
\begin{center}
\includegraphics[width=30mm]{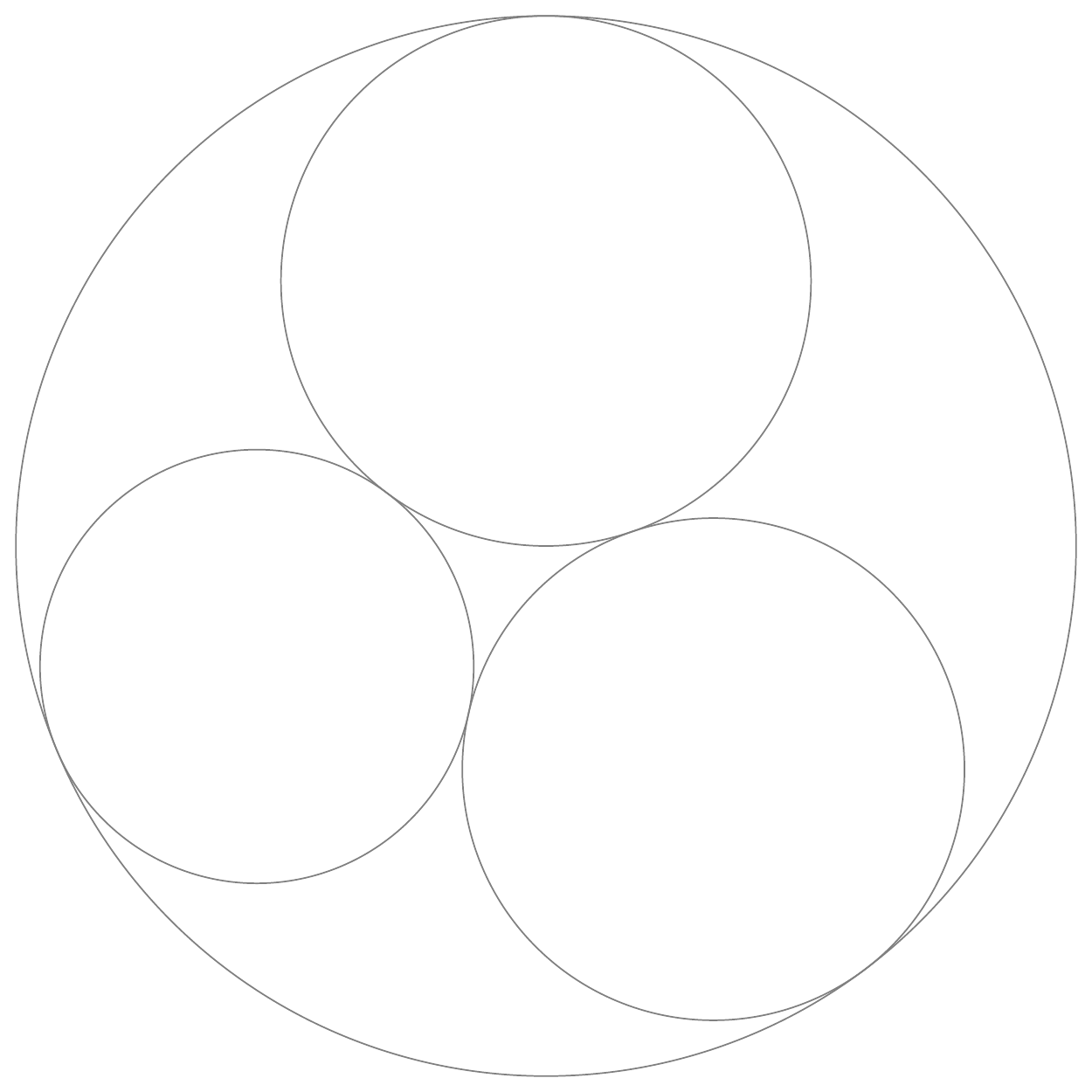}
\hspace{3mm}
\includegraphics[width=30mm]{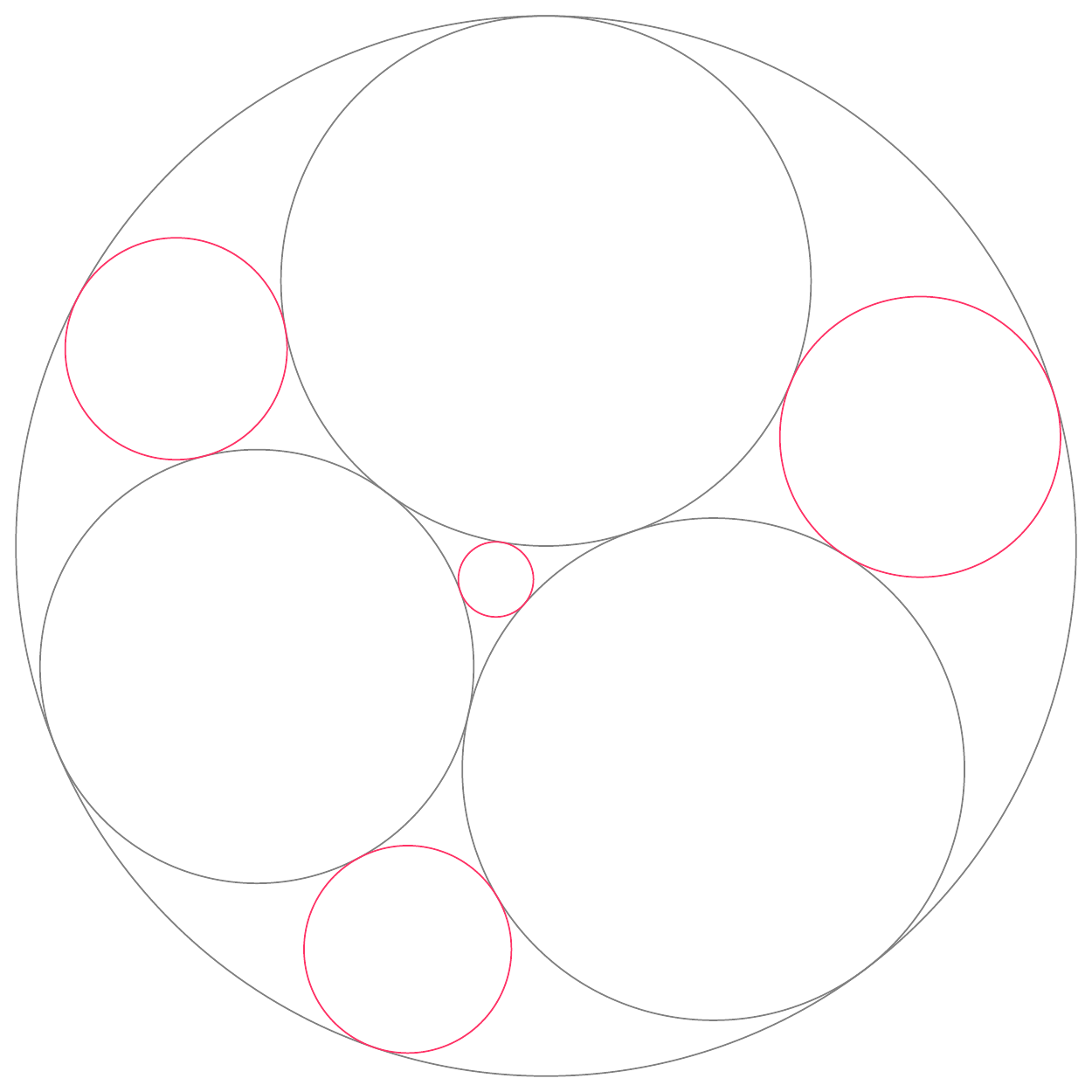}
\hspace{3mm}
\includegraphics[width=30mm]{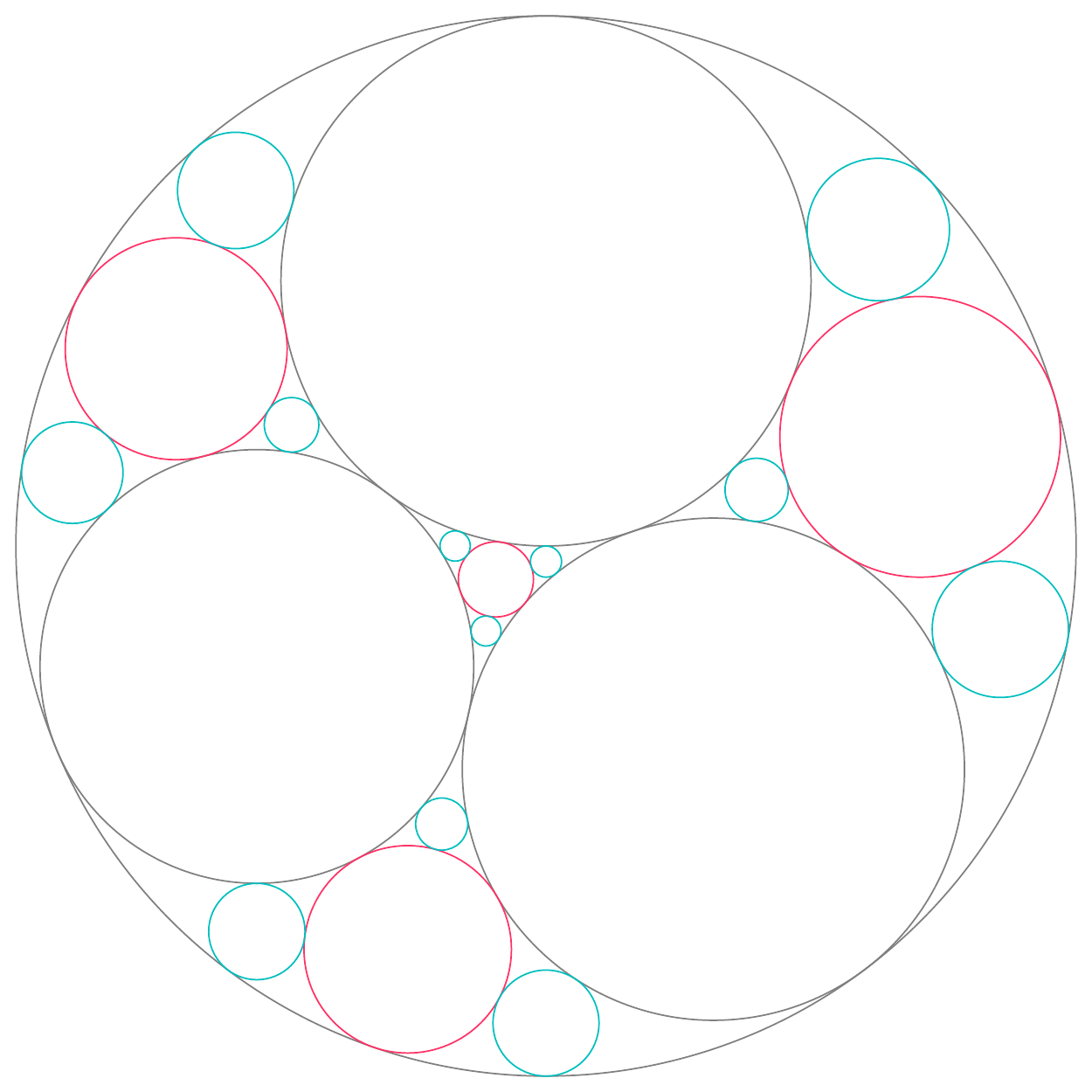}
\vspace{-3mm}
\end{center}
\caption{Construction of a \emph{tetrahedral} Apollonian circle packing}
\label{fig:TACP}
\vspace{-0.5mm}
\end{figure}

Guettler and Mallows generalized this construction by starting with an \emph{octahedral} configuration of six circles in $\bbR^2$, having the octahedral tangency graph \cite{GM}. For any sub-configuration of three pairwise tangent circles, there exists a unique dual circle orthogonal to them, and inverting the whole configuration along the dual circle yields a new octahedral configuration of six circles. Indefinitely continuing this process, as shown in \hyperref[fig:OACP]{Figure~\ref*{fig:OACP}}, we obtain an \emph{octahedral} Apollonian circle packing.
\begin{figure}[h]
\vspace{-0.5mm}
\begin{center}
\includegraphics[width=30mm]{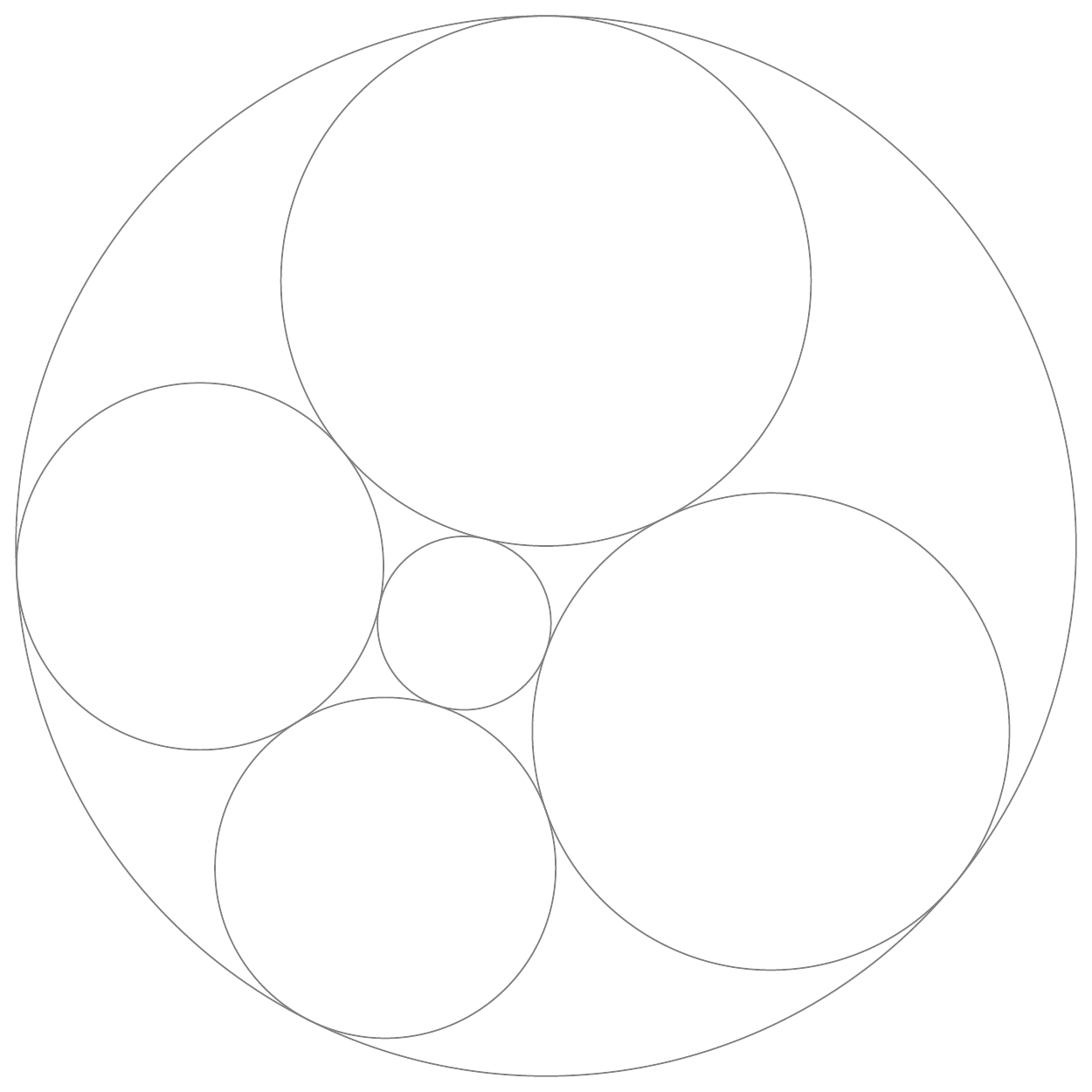}
\hspace{3mm}
\includegraphics[width=30mm]{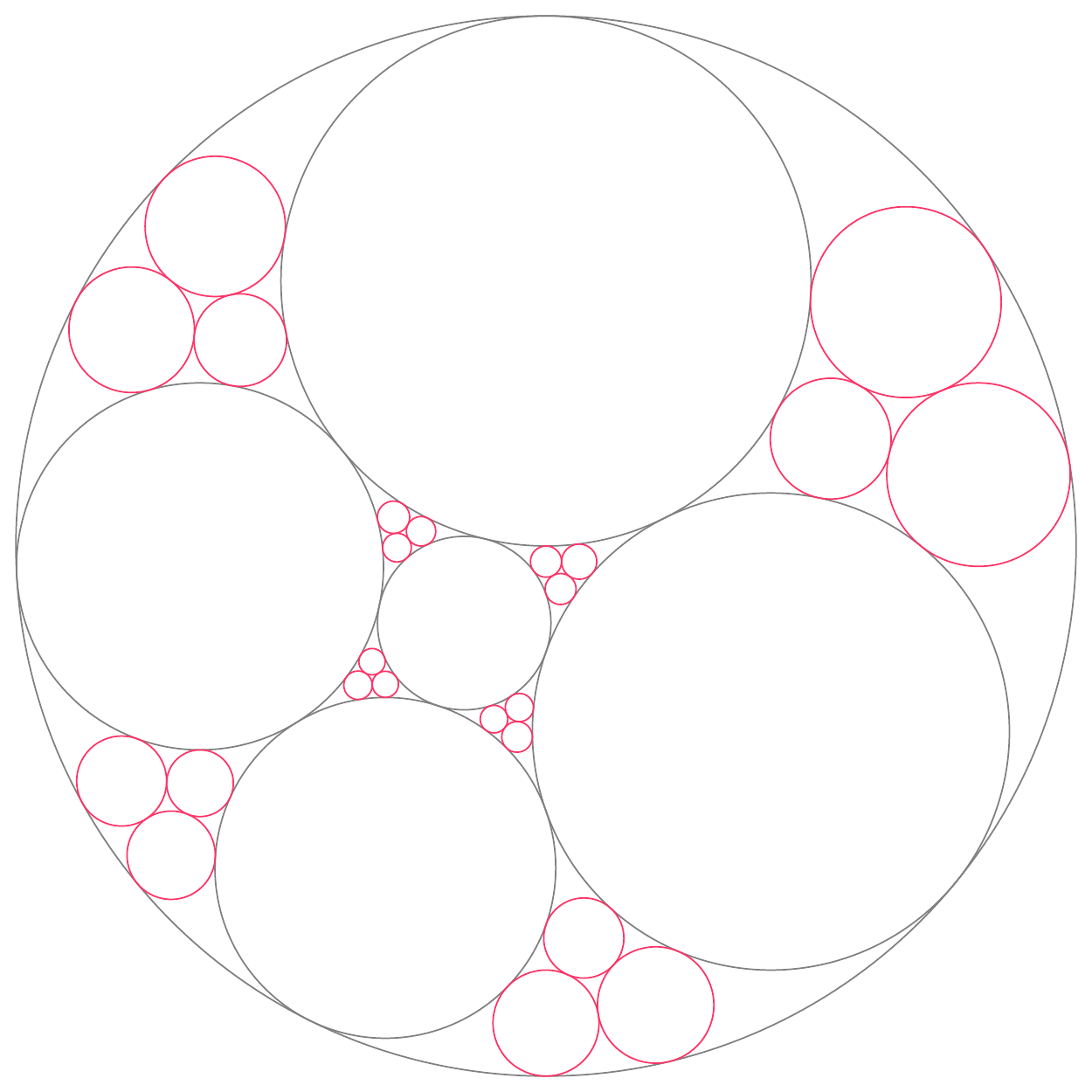}
\hspace{3mm}
\includegraphics[width=30mm]{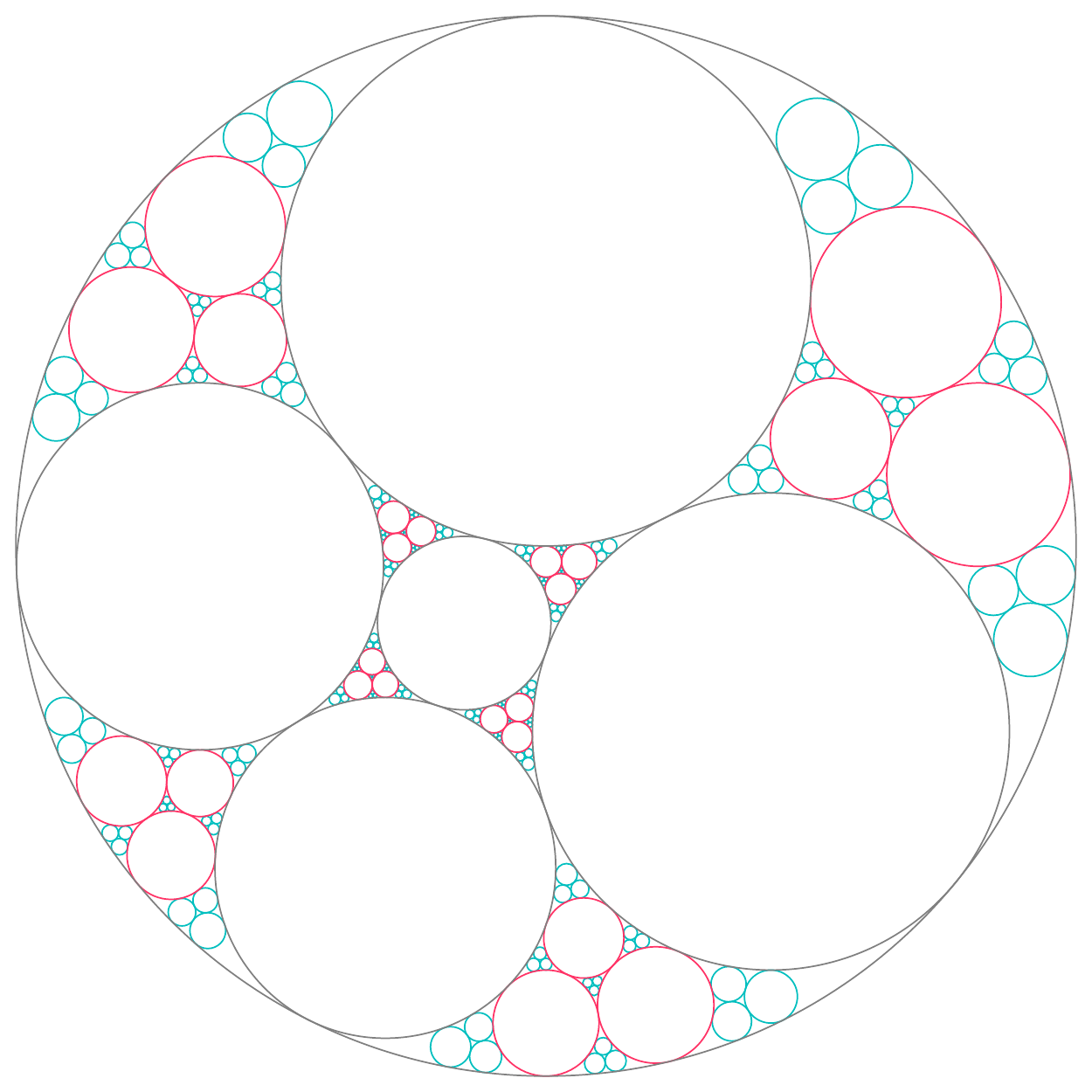}
\vspace{-3mm}
\end{center}
\caption{Construction of an \emph{octahedral} Apollonian circle packing}
\label{fig:OACP}
\vspace{-0.5mm}
\end{figure}

The 3-dimensional analogue of the tetrahedral Apollonian circle packings has been known for some time. Take a configuration of five pairwise tangent spheres in $\bbR^3$, having the \emph{4-simplicial} tangency graph, i.e.\;the 1-skeleton of the 4-simplex. For any sub-configuration of four spheres, there exists a unique dual sphere orthogonal to them, and inverting the whole configuration along the dual sphere yields a new configuration of five pairwise tangent spheres. Indefinitely continuing this process, as shown in \hyperref[fig:SASP]{Figure~\ref*{fig:SASP}}, we obtain a \emph{simplicial} Apollonian sphere packing.
\begin{figure}[h]
\vspace{-0.5mm}
\begin{center}
\includegraphics[trim = 70px 91px 51px 62px, clip, width=30mm]{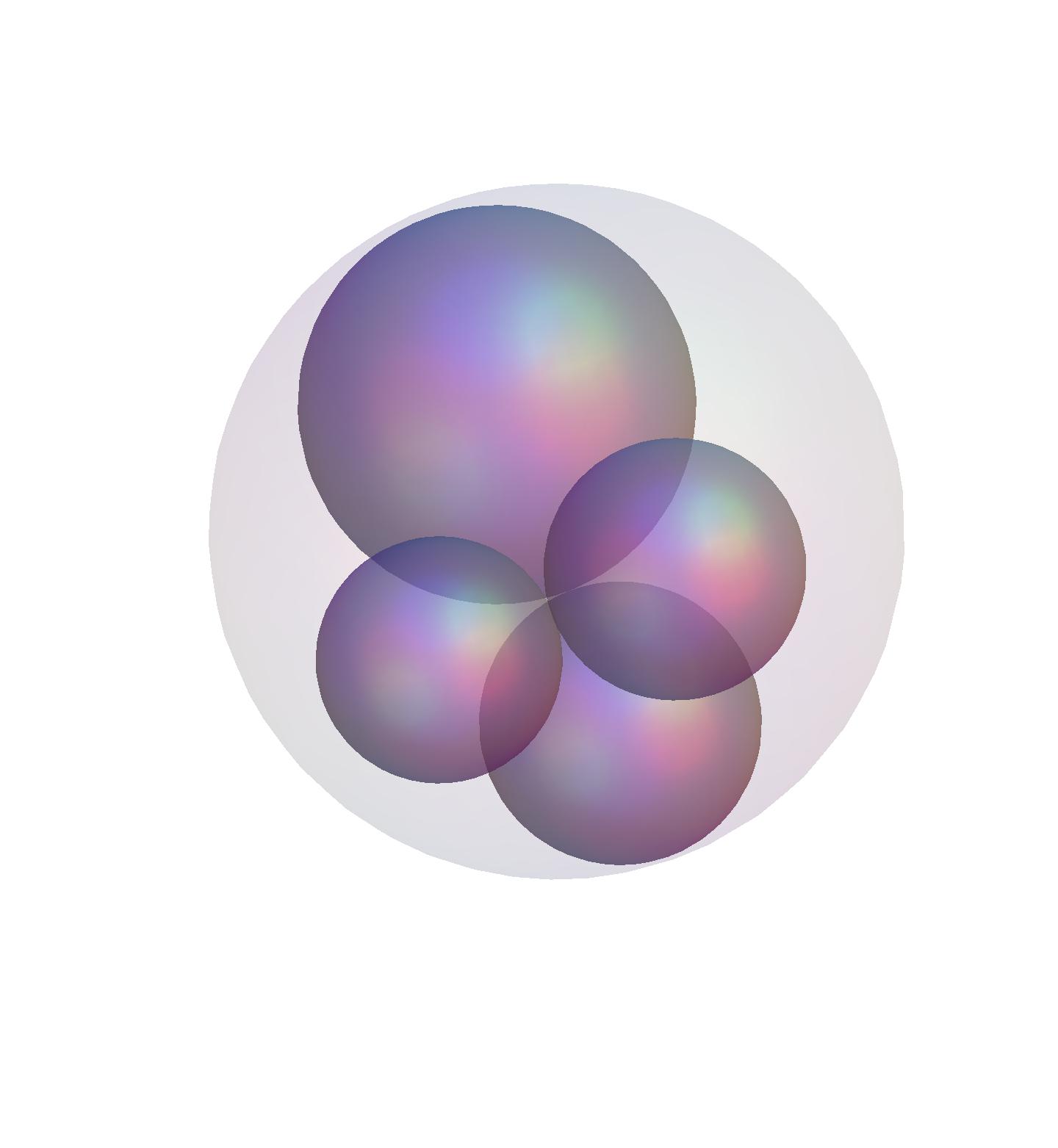}
\hspace{3mm}
\includegraphics[trim = 70px 91px 51px 62px, clip, width=30mm]{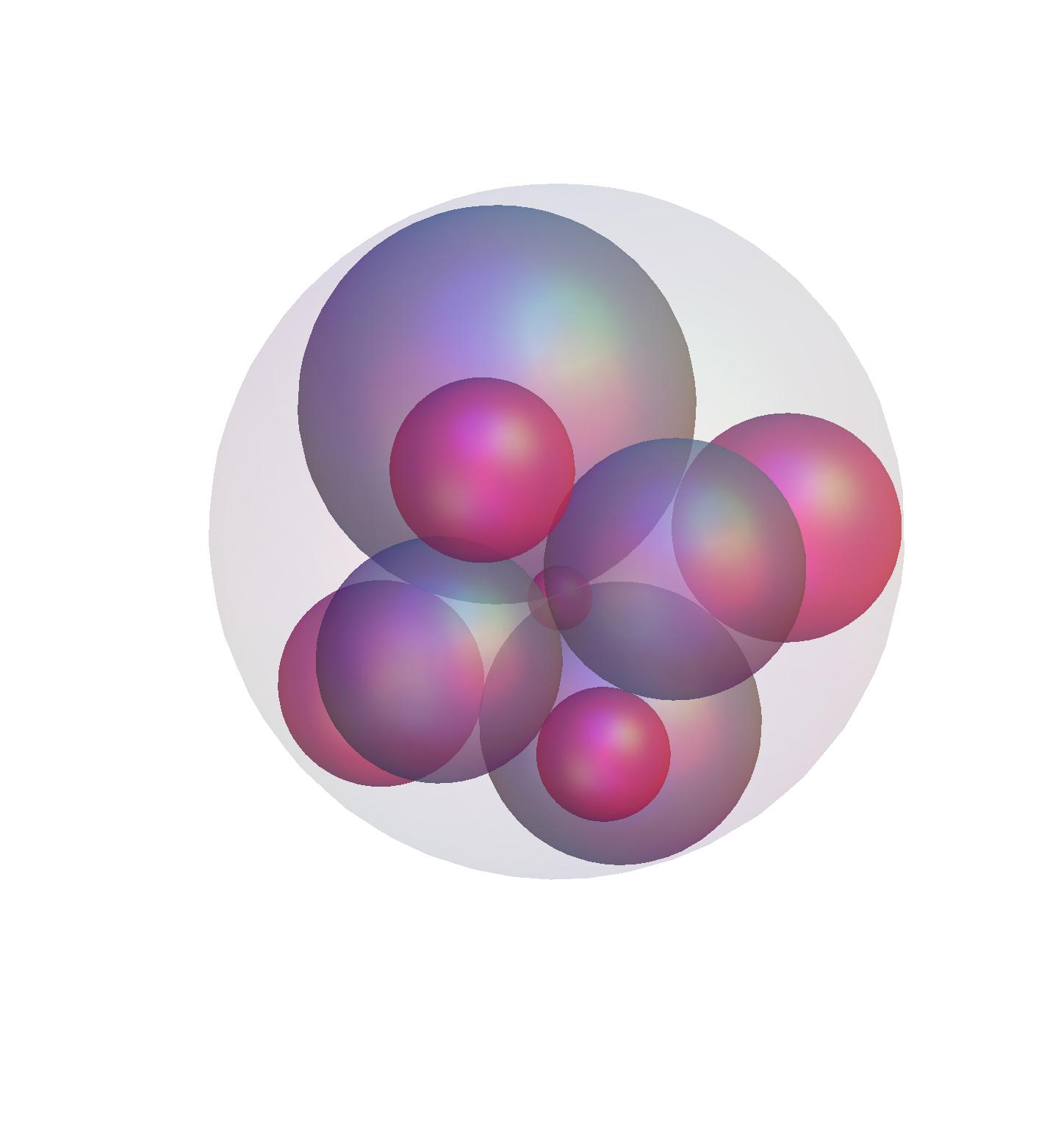}
\hspace{3mm}
\includegraphics[trim = 70px 91px 51px 62px, clip, width=30mm]{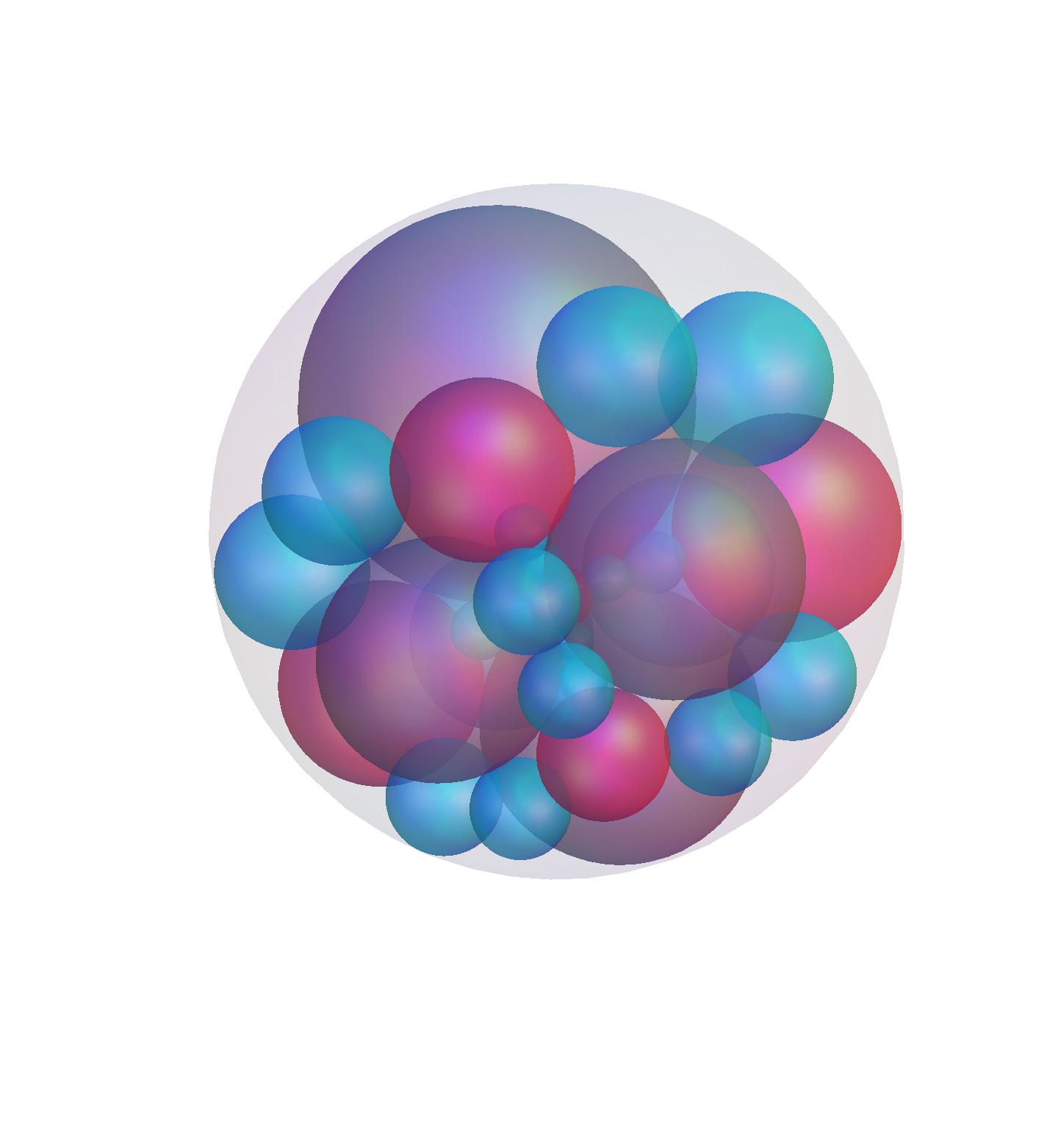}
\vspace{-3mm}
\end{center}
\caption{Construction of an \emph{simplicial} Apollonian sphere packing}
\label{fig:SASP}
\vspace{-0.5mm}
\end{figure}

In this article, we introduce the 3-dimensional analogue of octahedral Apollonian circle packings. Take an \emph{orthoplicial} configuration of eight spheres in $\bbR^3$, having \emph{4-orthoplicial} tangency graph, i.e.\;the 1-skeleton of 4-orthoplex. For any sub-configuration of four pairwise tangent spheres, there exists a unique dual sphere orthogonal to them, and inverting the whole configuration along the dual sphere yields a new orthplicial configuration of eight spheres. We can indefinitely continue this process, as shown in \hyperref[fig:OASP]{Figure~\ref*{fig:OASP}}. We refer to the union of all spheres in this construction as an \emph{orthoplicial} Apollonian sphere packing.
\begin{figure}[h]
\vspace{-0.5mm}
\begin{center}
\includegraphics[trim = 70px 91px 51px 62px, clip, width=30mm]{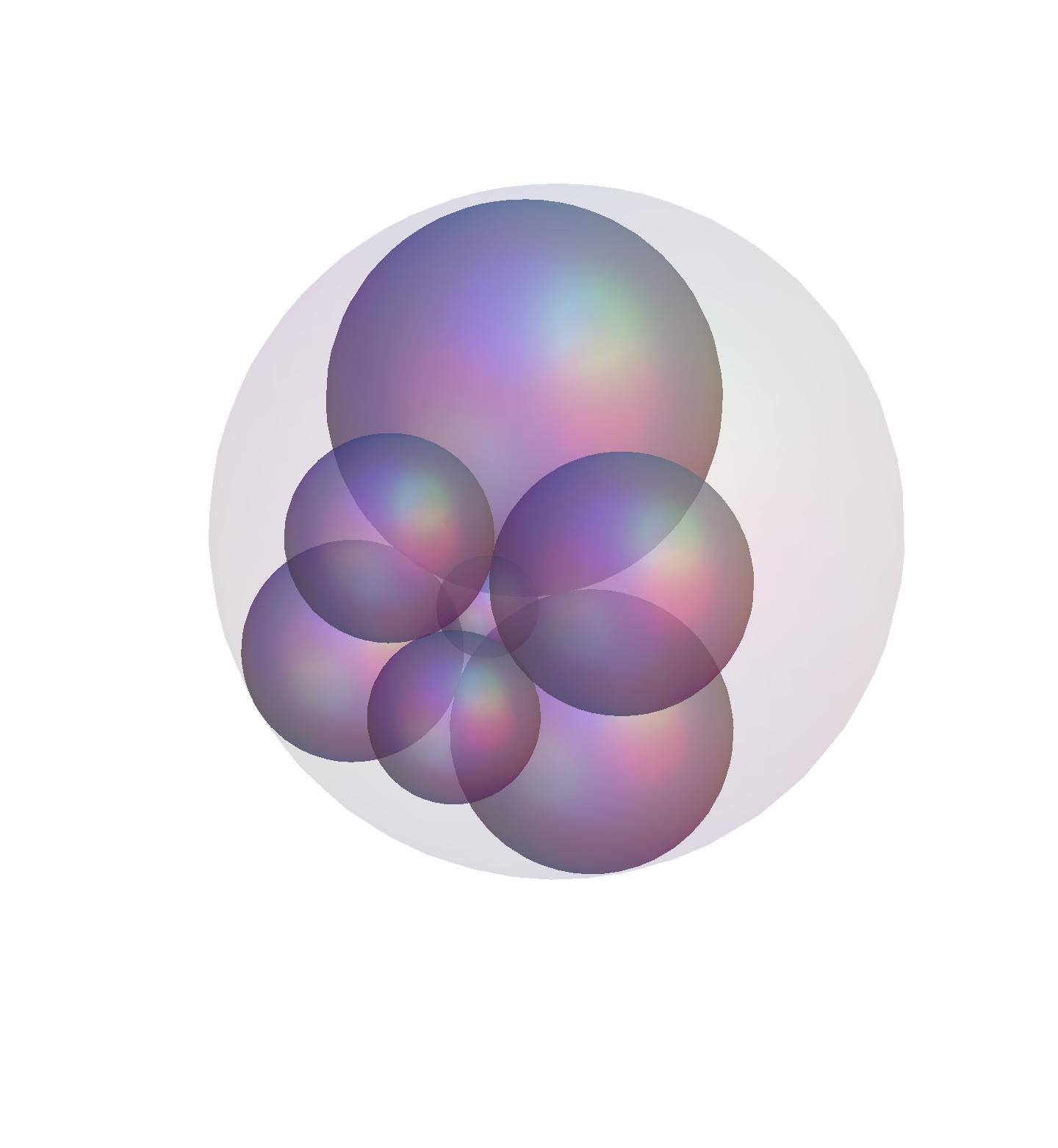}
\hspace{3mm}
\includegraphics[trim = 70px 91px 51px 62px, clip, width=30mm]{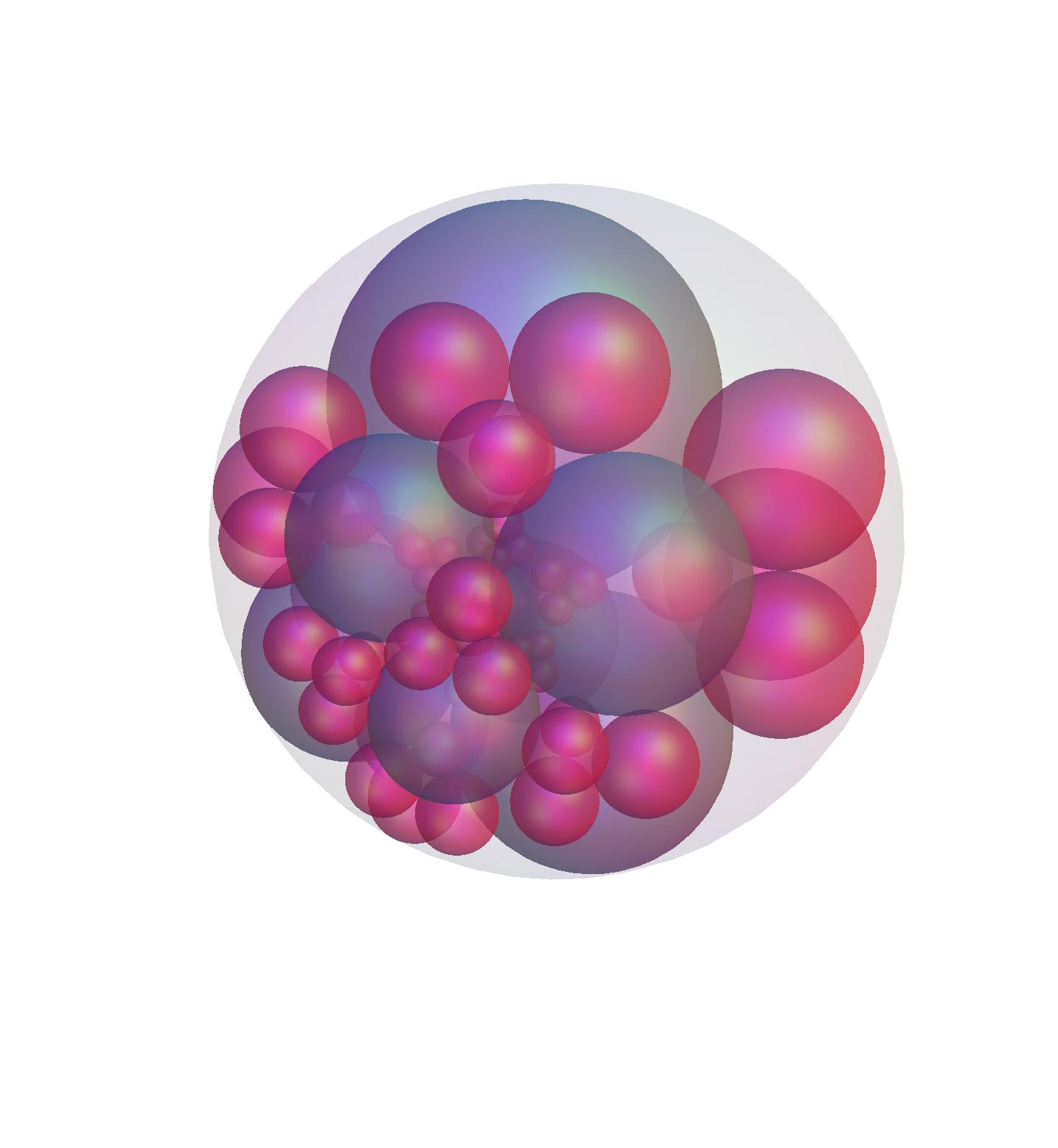}
\hspace{3mm}
\includegraphics[trim = 70px 91px 51px 62px, clip, width=30mm]{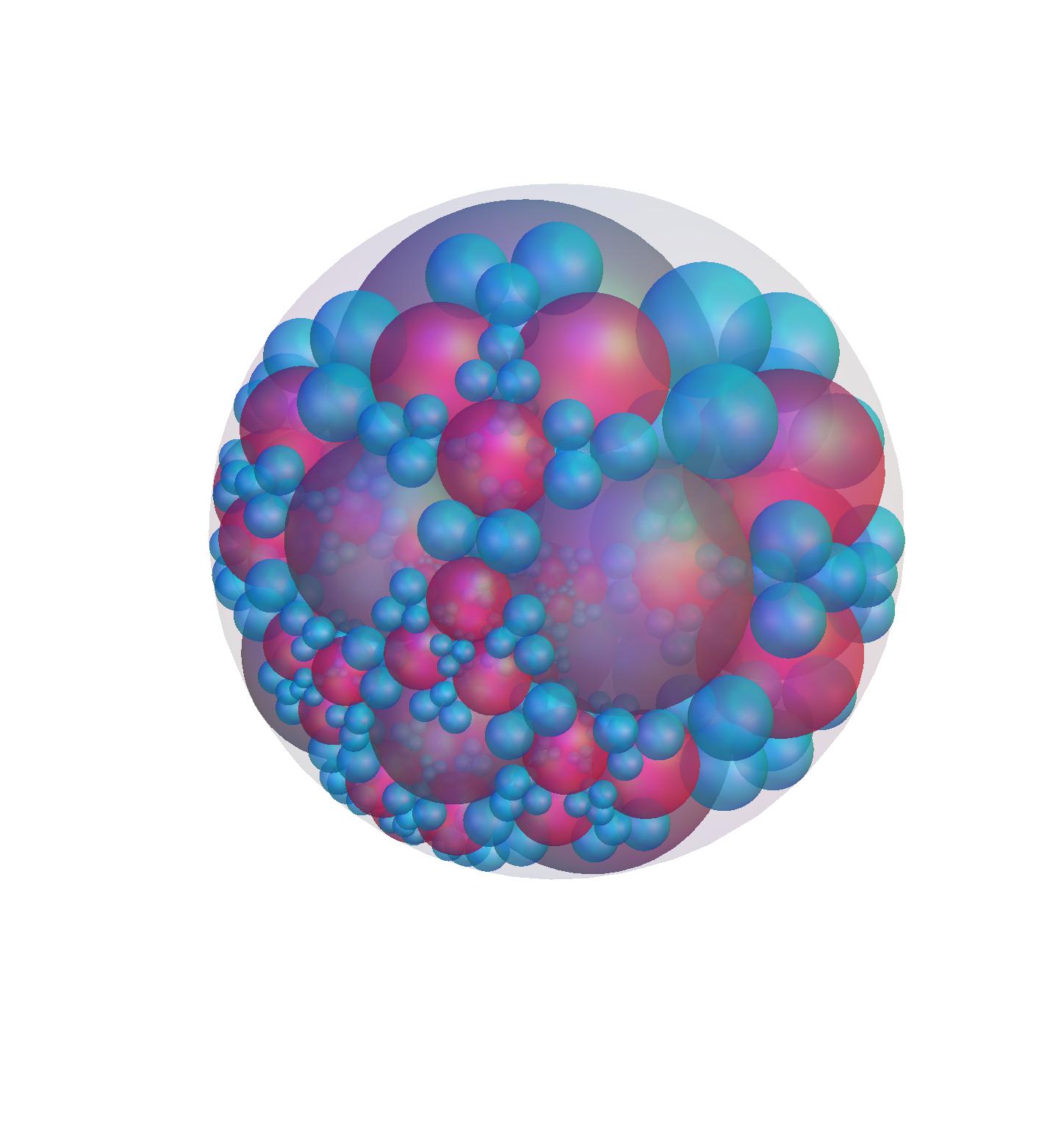}
\vspace{-3mm}
\end{center}
\caption{Construction of an \emph{orthoplicial} Apollonian sphere packing}
\label{fig:OASP}
\vspace{-4mm}
\end{figure}

\subsection{Integral Bends}

Remarkably, there exist tetrahedral/octahedral Apollonian circle packings in which the bends of all circles are integers; we refer to them as \emph{integral} tetrahedral/octahedral Apollonian circle packings. See \hyperref[fig:IntegralTOACP]{Figure~\ref*{fig:IntegralTOACP}} for an example of an integral tetrahedral Apollonian packing (left) and an example of an integral octahedral Apollonian packing (right).
\begin{figure}[h]
%\vspace{-0.5mm}
\begin{center}
\includegraphics[width=50mm]{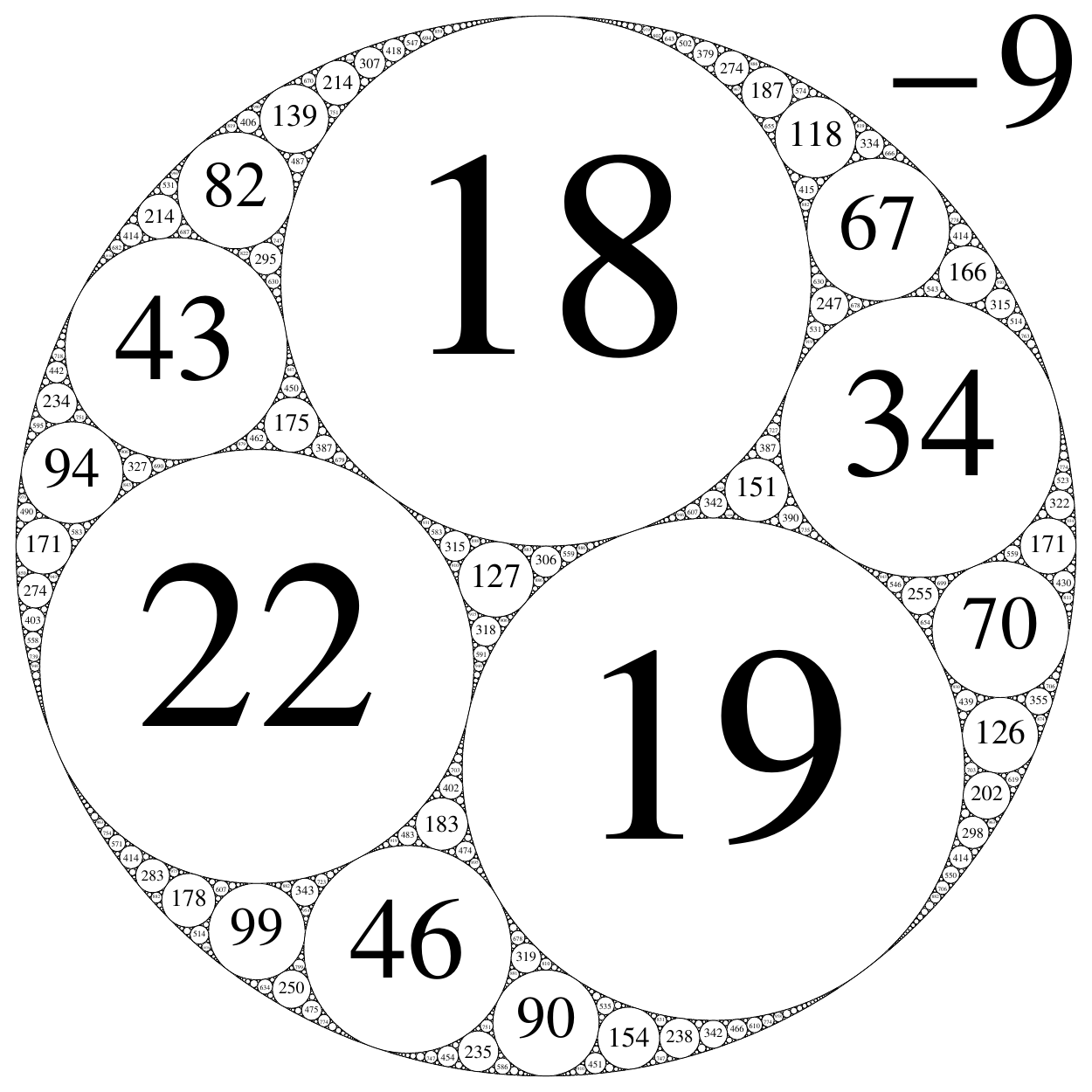}
\hspace{6mm}
\includegraphics[width=50mm]{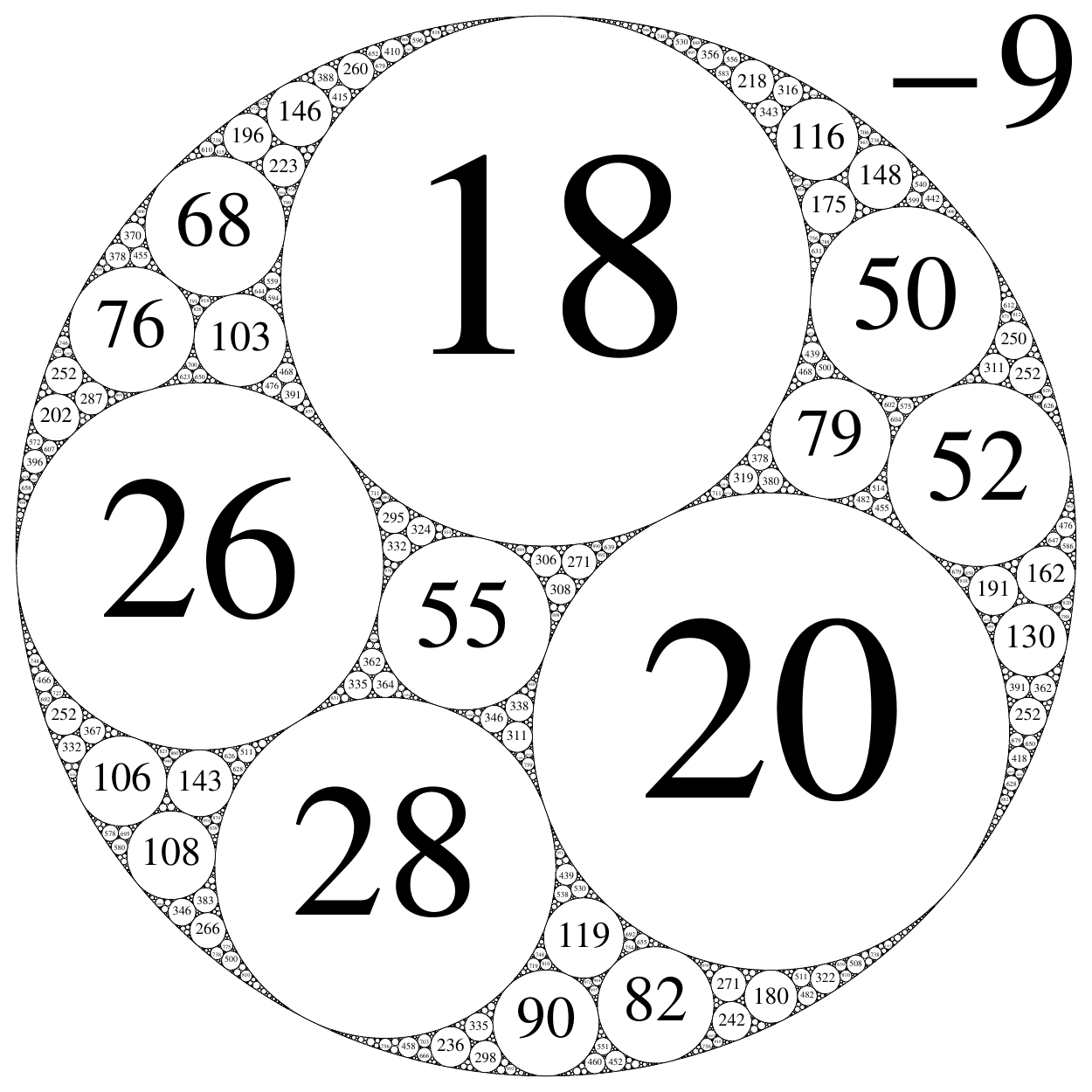}
\end{center}
\caption{Integral Apollonian packings, obtained from a tetrahedral configuration (left) and an octahedral configuration (right)}
\label{fig:IntegralTOACP}
%\vspace{-0.5mm}
\end{figure}

For tetrahedral Apollonian packings, the existence of integral packings is classically known as a consequence of the Descartes' Circle Theorem in his letter to Princess Elizabeth of Behemia \cite[p.\;45-50]{Descartes}, cf.\;\cite{Steiner}, \cite{Beecroft}, which states that the bends $b_1, \cdots, b_4$ of a tetrahedral configuration of four circles must satisfy
\begin{align*}
2(b_1^2+b_2^2+b_3^2+b_4^2)-(b_1+b_2+b_3+b_4)^2=0.
\end{align*}

For octahedral Apollonian packings, the existence of integral packings is observed by Guettler and Mallows \cite{GM} as a consequence of their theorem which states that the bends $b_1, \cdots, b_6$ of an octahedral configuration of six circles, labeled so that $b_k$ and $b_{k+3}$ are the bends of disjoint circles, must satisfy
\begin{align*}
\begin{matrix}
b_1+b_4=b_2+b_5=b_3+b_6=:2b_\mu, \;\; \text{and}\\
b_\mu^2-2(b_1+b_2+b_3)b_\mu+(b_1^2+b_2^2+b_3^2)=0.
\end{matrix}
\end{align*}

It turns out that there also exist simplicial/orthoplicial Apollonian sphere packings in which the bends of all spheres are integers; we refer to them as \emph{integral} simplicial/orthoplicial Apollonian sphere packings. See \hyperref[fig:IntegralSOASP]{Figure~\ref*{fig:IntegralSOASP}} for an example of an integral simplicial Apollonian packing (left) and an example of an integral orthoplicial Apollonian packing (right).

For simplicial Apollonian packings, the existence of integral packings was observed by Soddy \cite{Soddy} as a consequence of the generalization of the Descartes' Circle Theorem, which states that the bends $b_1,\cdots,b_5$ of a simplicial configuration of five spheres must satisfy
\begin{align*}
3(b_1^2+b_2^2+b_3^2+b_4^2+b_5^2)-(b_1+b_2+b_3+b_4+b_5)^2=0.
\end{align*}
This equation has been known to Descartes himself \cite{Aeppli}; the equation has been rediscovered many times, e.g.\;\cite{Lachlan}, \cite{Soddy}.

\begin{figure}[ht]
\vspace{-0.5mm}
\begin{center}
\includegraphics[trim = 70px 91px 51px 62px, clip, width=50mm]{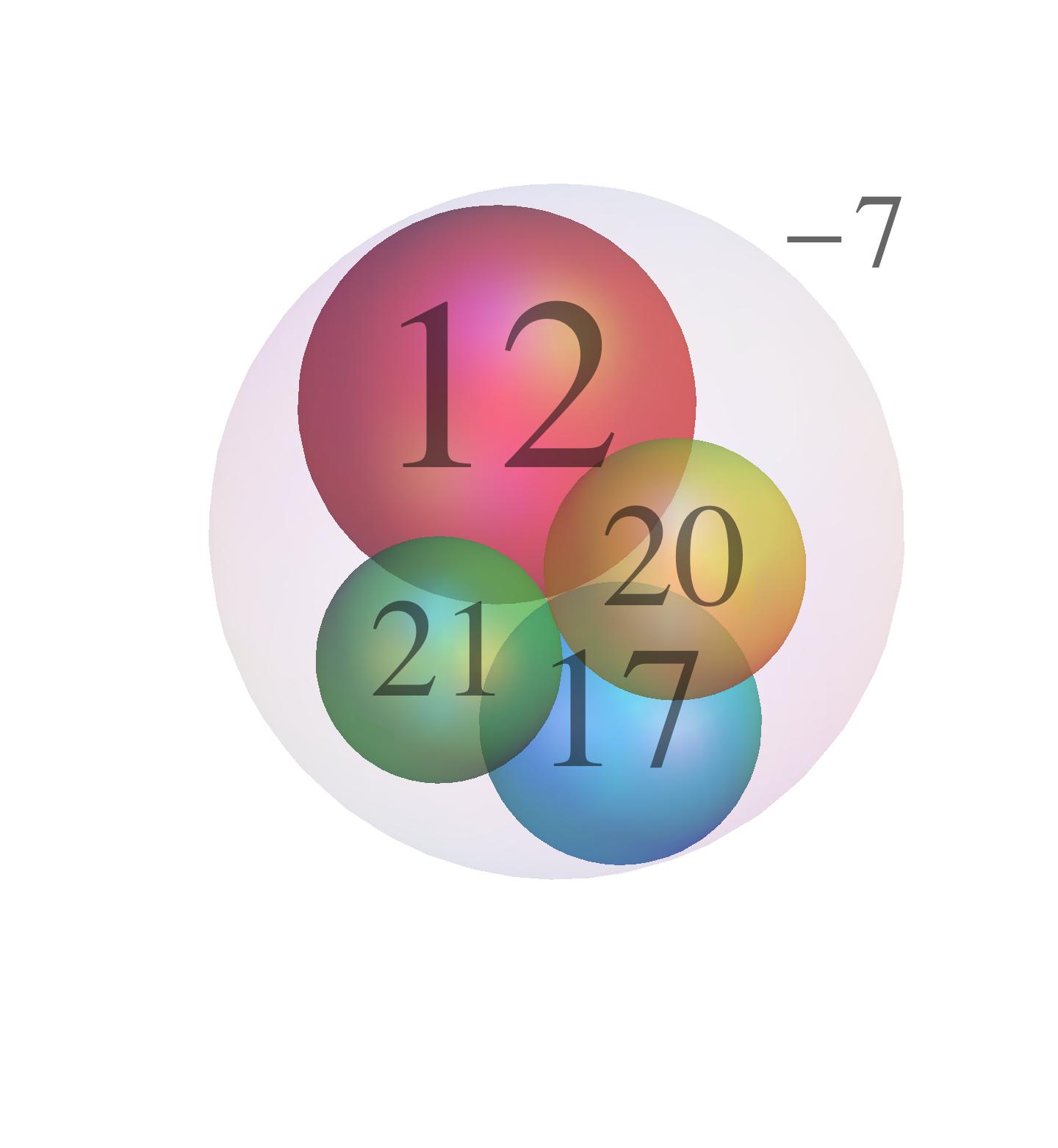}
\hspace{6mm}
\includegraphics[trim = 70px 91px 51px 62px, clip, width=50mm]{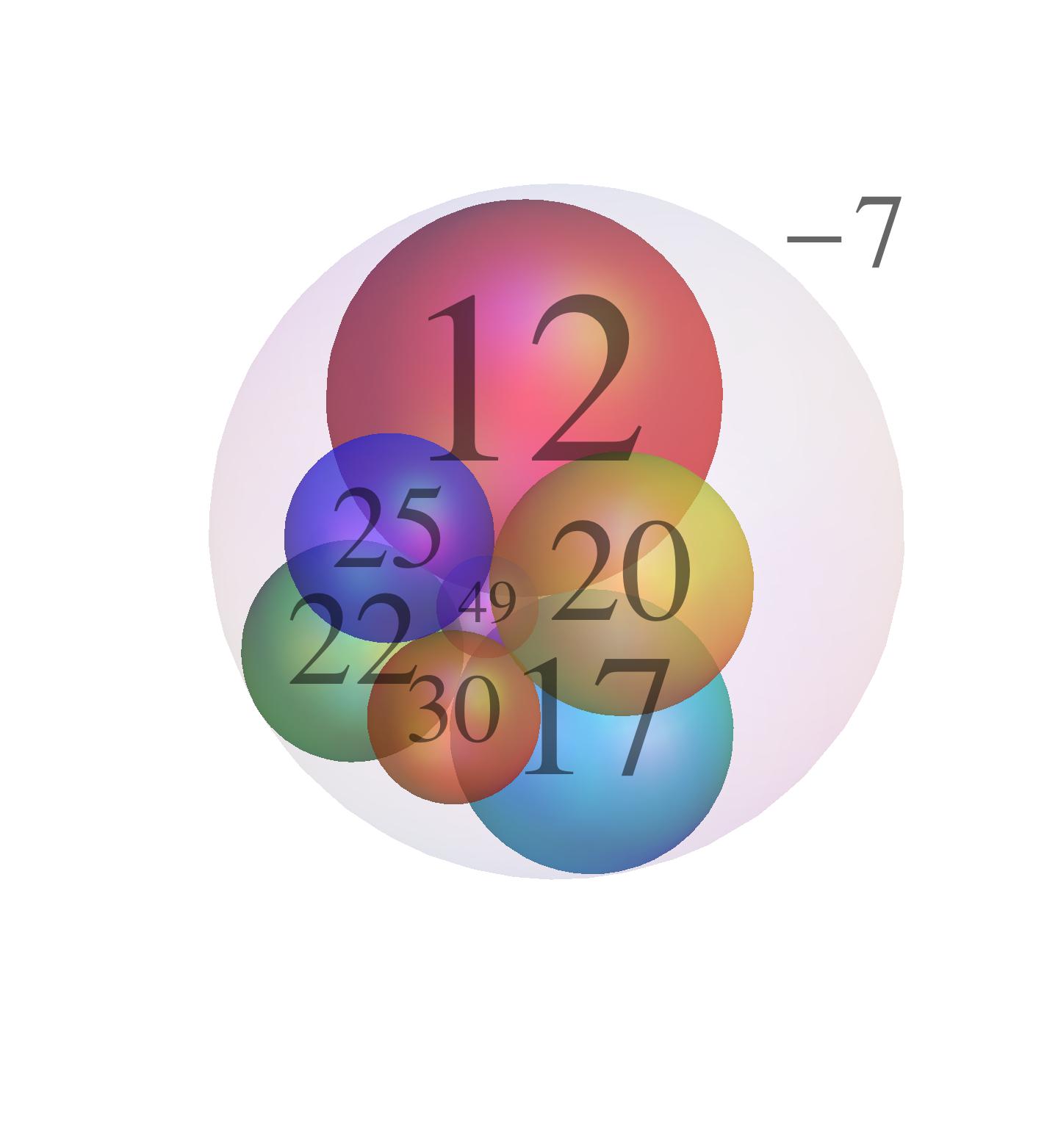}
\\ \vspace{3mm}
\includegraphics[trim = 70px 91px 51px 62px, clip, width=50mm]{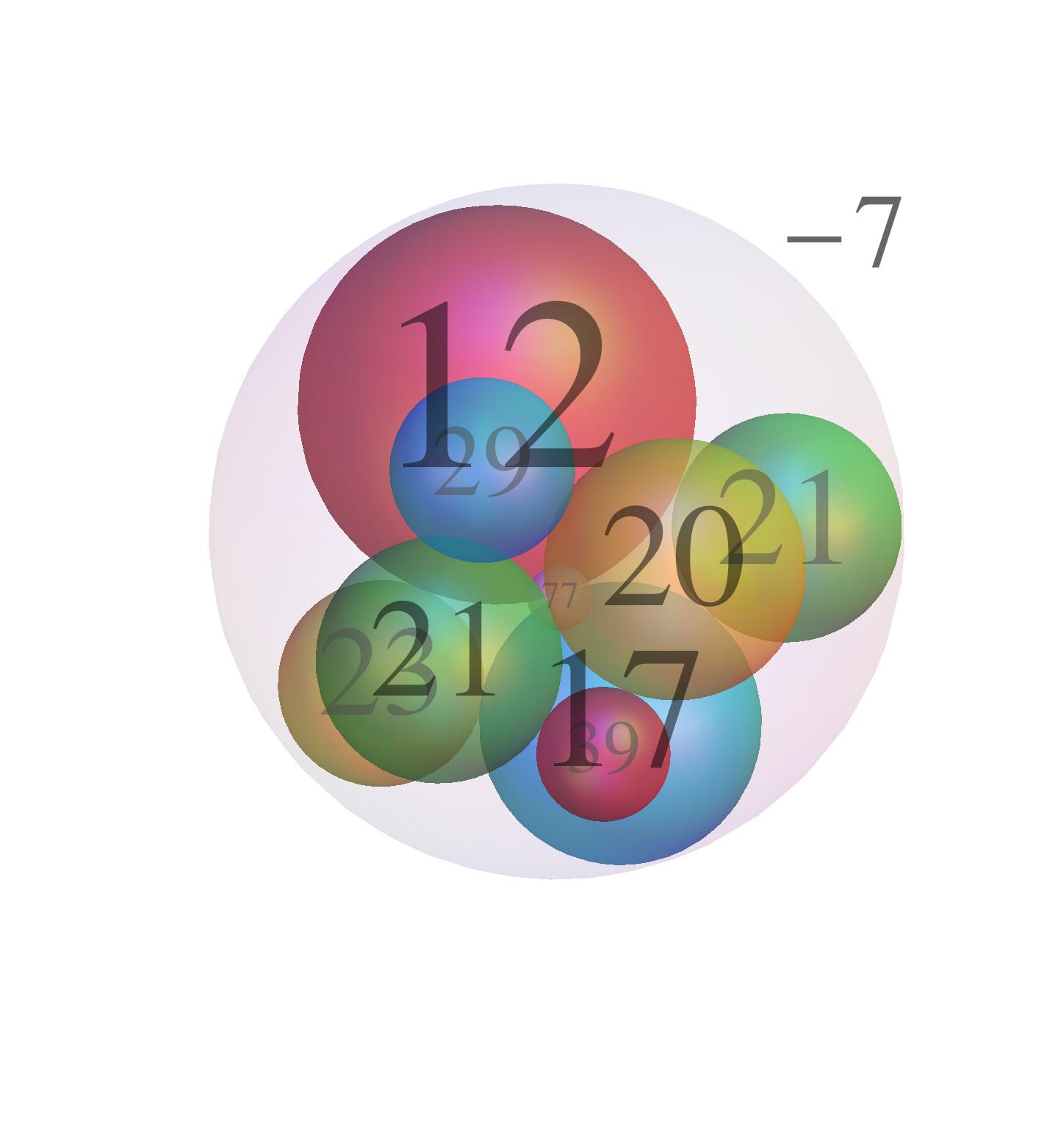}
\hspace{6mm}
\includegraphics[trim = 70px 91px 51px 62px, clip, width=50mm]{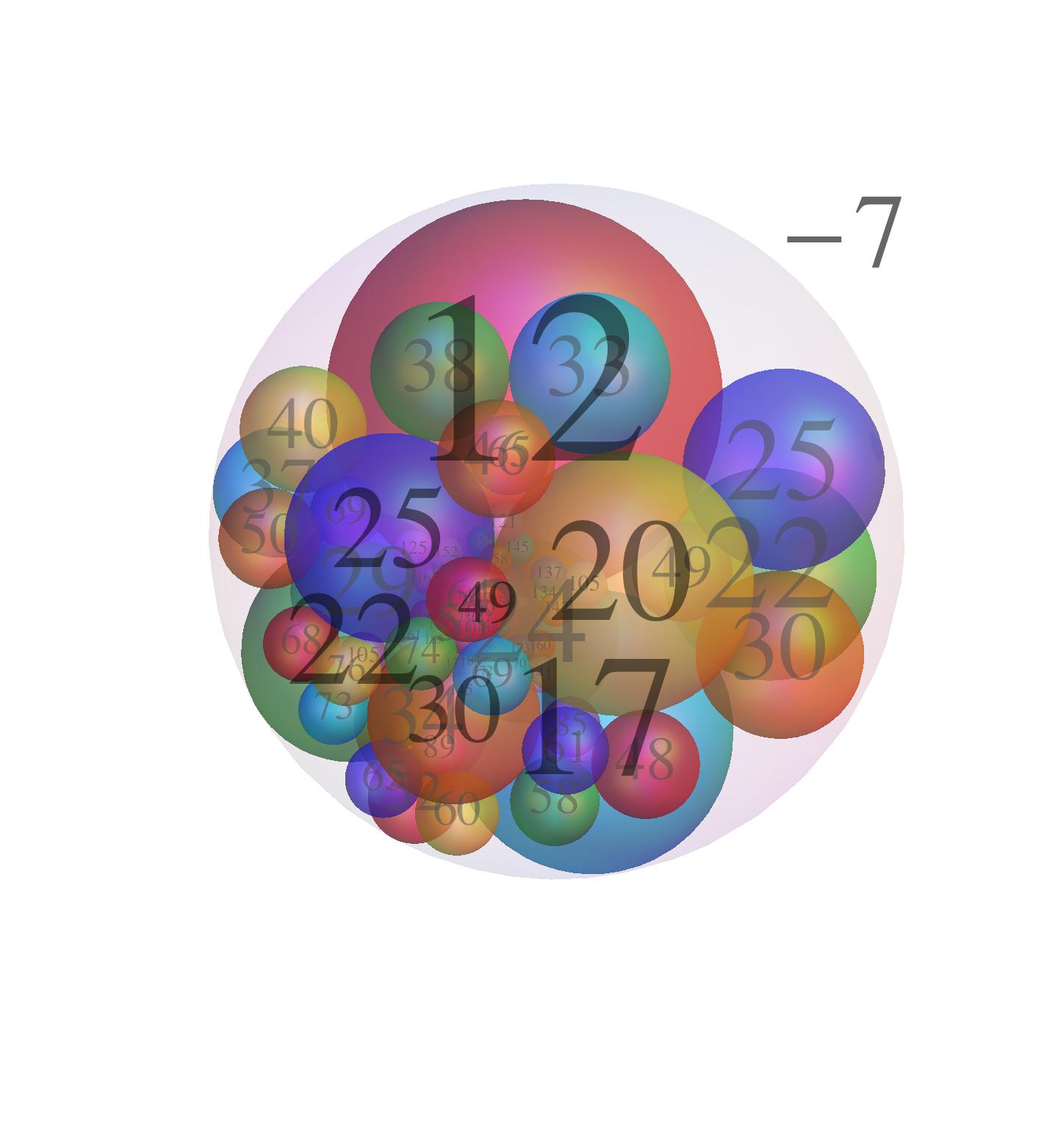}
\\ \vspace{3mm}
\includegraphics[trim = 70px 91px 51px 62px, clip, width=50mm]{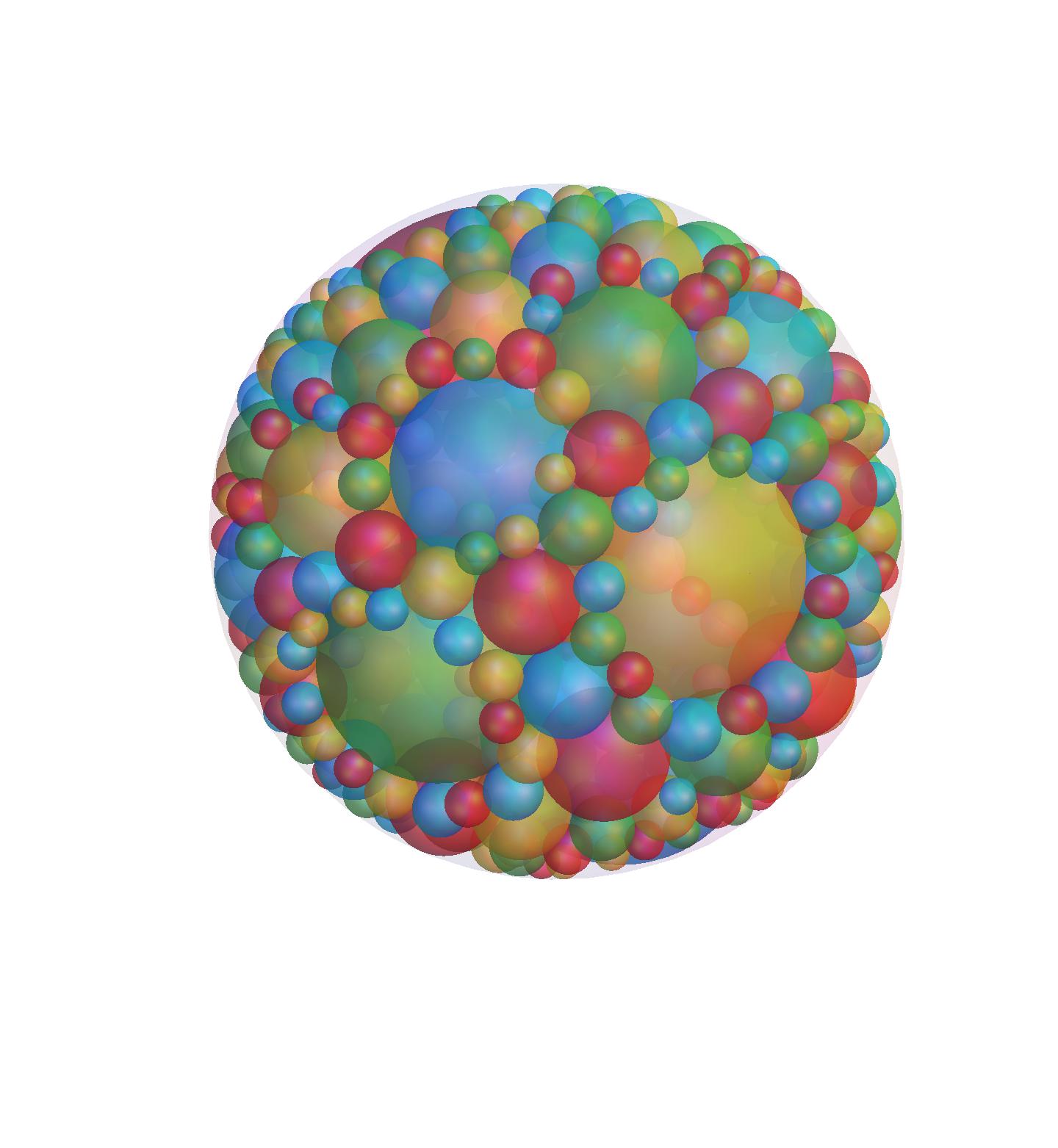}
\hspace{6mm}
\includegraphics[trim = 70px 91px 51px 62px, clip, width=50mm]{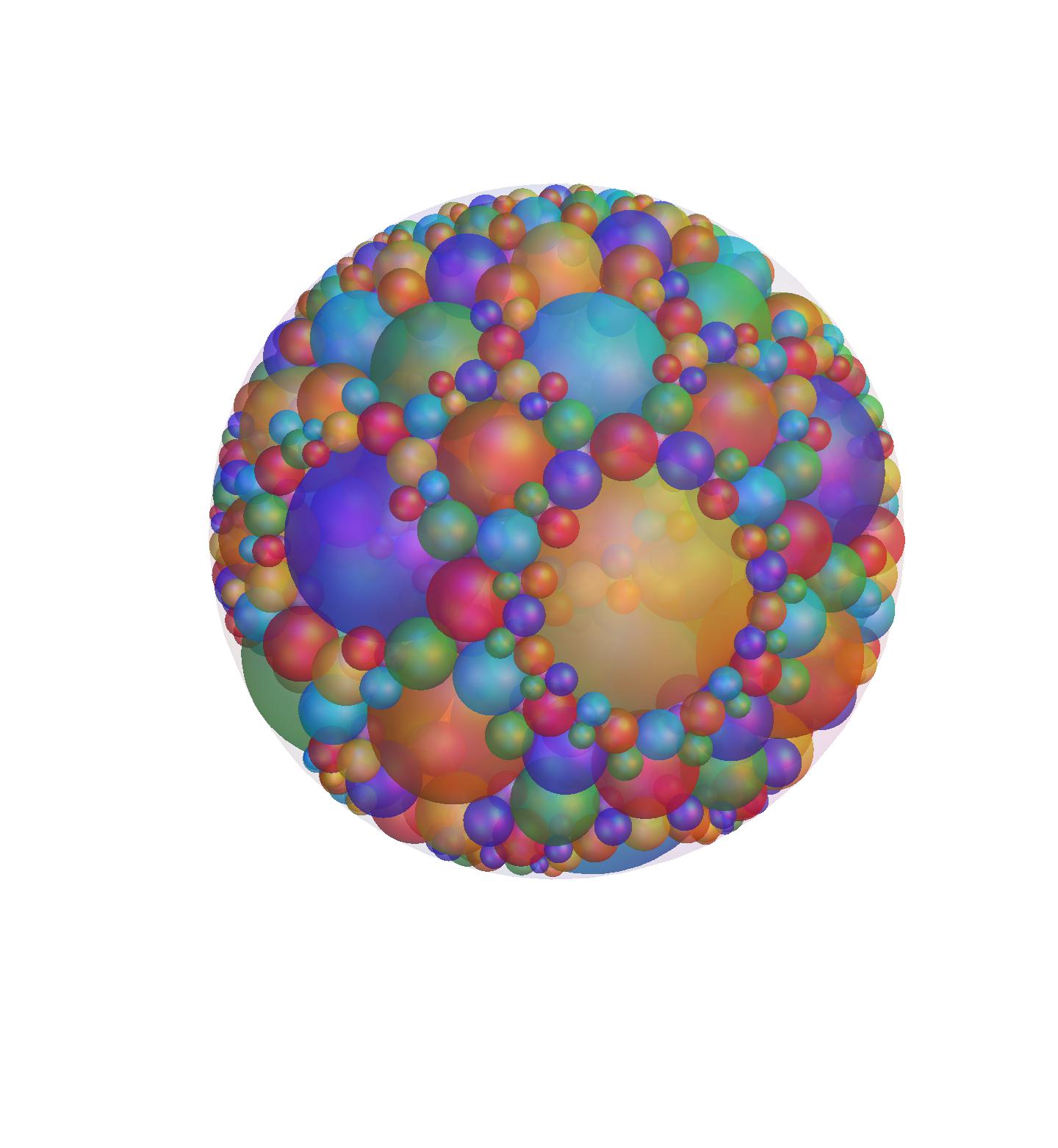}
\end{center}
\caption{Integral Apollonian packings, obtained from a simplicial configuration (left) and an octahedral configuration (right)}
\label{fig:IntegralSOASP}
\vspace{-0.5mm}
\end{figure}

In \hyperref[sec:Platonic]{\S\ref*{sec:Platonic}} and \hyperref[sec:Apollonian]{\S\ref*{sec:Apollonian}}, we establish the existence of integral orthoplicial Apollonian packings, generalizing \cite{GM}. It is a consequence of \hyperref[coro:Antipodal]{Corollary~\ref*{coro:Antipodal}} and \hyperref[thm:DGM]{Theorem~\ref*{thm:DGM}} which imply that the bends $b_1,\cdots,b_8$ of an orthoplicial configuration of eight spheres, labeled so that $b_k$ and $b_{k+4}$ are the bends of disjoint spheres, must satisfy
\begin{align*}
\begin{matrix}
b_1+b_5=b_2+b_6=b_3+b_7=b_4+b_8=:2b_\mu, \;\; \text{and}\\
2b_\mu^2-2(b_1+b_2+b_3+b_4)b_\mu+(b_1^2+b_2^2+b_3^2+b_4^2)=0.
\end{matrix}
\end{align*}

%%, 
\subsection{Local to Global}

Recently, there have been remarkable advances in understanding the diophantine properties of the set  or the multi-set of integers occurring as bends in integral Apollonian circle/sphere packings. In the foundational work \cite{GLMWY:NT1}, Graham, Lagarias, Mallows, Wilks, and Yan posed several fundamental questions on tetrahedral Apollonian circle packings in this direction. Many of them are now resolved, fully or partially; see \cite{S:Letter}, \cite{BF:Positive}, \cite{B:Prime}, \cite{BK:Strong} on the set of integral bends, and \cite{KO:Count}, \cite{LO:Count} on the multi-set of integral bends. Subsequently, analogous questions on the set of integral bends were studied for simplicial Apolonian sphere packings in \cite{K:Soddy} and for octahedral Apollonian circle packings in \cite{Z:Octahedral}. In this article, we address the diophantine properties of the set of integers occurring as bends in integral orthoplicial Apollonian packings. 

Given an integral orthoplicial Apollonian packings, let us write $\scrB(\scrP)$ for the set of integers appearing as bends in $\scrP$. Let us assume that $\scrP$ is primitive, i.e. $\gcd\scrB(\scrP)=1$. To give a heuristic idea on which integers may arise in $\scrB(\scrP)$, let us present an example. For the integral orthoplicial Apollonian packing shown in \hyperref[fig:IntegralSOASP]{Figure~\ref*{fig:IntegralSOASP}}, explicit computations on small integers in $\scrB(\scrP)$ yield
\begin{align*}
\scrB(\scrP)&=
\left\{\begin{matrix}
-7, 12, 17, 20, 22, 24, 25, 29, 30, 33, 34, 37, 38, 40, 41, 44, 46, \\
48, 49, 50, 52, 53, 54, 56, 58, 60, 61, 62, 64, 65, 66, 68, 69, \cdots, \\
200, 201, 202, 204, 205, 206, 208, 209, 210, 212, 213, 214, 216, \\
217, 218, 220, 221, 222, 224, 225, 226, 228, 229, 230, 232, 233, \\
234, 236, 237, 238, 240, 241, 242, 244, 245, 246, 248, 249, \cdots\;\;
\end{matrix}\right\}.
\end{align*}
Looking at these numbers, we immediately observe that $\scrB(\scrP)$ seems to contain all sufficiently large integers $n \equiv 0,1,2 \pmod{4}$ but no integers $n \equiv 3 \pmod{4}$. Generally, for any primitive orthoplicial Apollonian packing $\scrP$, it appears that $\scrB(\scrP)$ contains all sufficiently large integers in $\scrA(\scrP):=\{ n\in \bbZ \mid n \not\equiv -\varepsilon(\scrP) \mod{4} \}$ but no integers outside $\scrA(\scrP)$, where $\varepsilon(\scrP)\in\{\pm1\}$ depends only on $\scrP$. These observations suggests the following statements, phrased in analogy with Hilbert's 11th problem on representations of integers by quadratic forms: 
\begin{itemize}
\item[(a)] there is a local obstruction modulo 4 as above,
\item[(b)] this obstruction modulo 4 is the only local obstruction, and
\item[(c)] sufficiently large locally represented integer is globally represented.
\end{itemize}

In \hyperref[sec:LG]{\S\ref*{sec:LG}}, we establish these statements. Our approach is similar to \cite{K:Soddy}; namely, adapting the ideas of Sarnak \cite{S:Letter}, we use a large arithmetic group to relate the integers represented in $\scrB(\scrP)$ with integers represented by a certain quaternary quadratic form. This quadratic form turns out to be positive-definite and \emph{isotropic at every prime}, allowing us to employ the classical result by Kloosterman on quaternary quadratic forms and prove our main result:

\begin{thm-LG}
Every primitive orthoplicial Apollonian sphere packing $\scrP$ satisfy the asymptotic local-global principle: there is an effectively and explicitly computable bound $N(\scrP)$ so that, if $n>N(\scrP)$ and $n \in \scrA(\scrP)$, then $n \in \scrB(\scrP)$.
\end{thm-LG}

\subsection*{Acknowledgement}

The author would like to thank Elena Fuchs for her enlightening talk on her results on classical Apollonian packings which introduced the author to the subject, and Igor Rivin for insightful discussions on various aspects of circle packings and sphere packings. Some of the results presented here were obtained while the author was at Temple University.

%%%
\section{Preliminaries}

\subsection{Inversive Spheres}

We write $\Euc^3$ for the euclidean 3-space; we choose a frame and a point as the origin to coordinatize $\Euc^3$ as $\bbR^3$. We work with the coordinatized \emph{M\"obius 3-space} $\hat\bbR^3=\bbR^3 \cup\{\infty\}$ and the Euclidean subspace $\bbR^3 \subset \hat\bbR^3$.

An \emph{inversive sphere} $S$ in the Euclidean 3-space $\Euc^3$ is either a sphere or a plane in $\Euc^3$; we say $S$ is \emph{honest} if $S$ is a sphere, and $S$ is \emph{planar} if $S$ is a plane. As usual, planes are regarded as spheres through the point at infinity, and parallel planes are considered to be tangent at infinity. An \emph{orientation} of an inversive sphere is a choice of unit normal field $\bsn$ on it, or equivalently a choice of one region $B \subset \Euc^3$ with $\partial B=S$; by convention, the \emph{orienting normal} $\bsn$ points into the \emph{orienting region} $B$. The orienting region may be a ball or a ball-complement (if $S$ is honest), or a half-space (if $S$ is planar). For every inversive sphere $S$, its \emph{bend} $b=b(S)$ is defined as follows. If $S$ is an honest sphere, we set $b:=1/r$ where $r$ is the \emph{oriented radius}, defined to be a non-zero real number such that (i) $|r|$ is the radius of the sphere, (ii) $r>0$ if the orienting region is a ball, and (iii) $r<0$ if the orienting region is a ball-complement. If $S$ is planar, we set $b:=0$. The bend is often called the \emph{oriented/signed curvature}; we will use the term ``bend'' in order to avoid the double meaning in the phrase ``negative curvature''. For the sake of brevity, in the rest of the article, a \emph{sphere} always means an \emph{oriented inversive sphere} unless stated otherwise. For the most of the article, we work mainly with honest spheres with positive bends, and little confusion should arise.

An oriented inversive sphere $S$ is specified uniquely and unambiguously by its \emph{inversive coordinates} \cite[\S9]{Wilker:Inversive}; see also \cite{LMW}, \cite{GLMWY:GG3}. By convention, we will always regard inversive coordinate vectors as row vectors, and we will usually denote them by $\bsv(S)=(a,b,\hat x, \hat y, \hat z)$. If an oriented inversive sphere $S$ is an honest sphere with the center $\bsc=(c_x,c_y,c_z)$ and the oriented radius $r$, the inversive coordinate vector of $S$ is defined to be the vector
\[
\bsv(S)=(a,b,\hat x, \hat y, \hat z):=(a,b,bc_x,bc_y,bc_z)
\]
where $b=1/r$ is the bend of $S$ and $a$ is the \emph{augmented bend} of $S$, which is defined to be the bend of the sphere obtained by inverting $S$ about the unit sphere centered at the origin. The augmented bend is given explicitly by $a=b|\bsc|-1/b$. If an oriented inversive sphere $S$ is planar, take the linear equation $n_xx+n_yy+n_zz=h$ for the plane, where $\bsn=(n_x,n_y,n_z)$ is the orientation unit normal vector to the plane. Then, the inversive coordinate vector of $S$ is defined to be
\[
\bsv(S)=(a,b,\hat x, \hat y, \hat z):=(2h,0,n_x,n_y,n_z).
\]
This coordinate vector can also be obtained as the limit of the honest sphere case.

\subsection{Inversive Product}

The \emph{inversive product} is one of the most essential notions in inversive geometry. Let us first define two matrices.
\[
\bsQSigma:=\medpmatrix{
0&-\frac{1}{2}&0&0&0\\
-\frac{1}{2}&0&0&0&0\\
0&0&1&0&0\\
0&0&0&1&0\\
0&0&0&0&1
},
\qquad
\bsQW:=2\bsQSigma^{-1}=\medpmatrix{
0&-4&0&0&0\\
-4&0&0&0&0\\
0&0&2&0&0\\
0&0&0&2&0\\
0&0&0&0&2
}.
\]
The matrix $\bsQW$ is the \emph{Wilker matrix} in \cite{LMW}, \cite{GLMWY:GG3}, and it is instrumental in studying sphere packings. The matrix $\bsQSigma$ defines the \emph{inversive product}.

\begin{defn}
The \emph{inversive product} is an indefinite symmetric bilinear form $\varSigma$ on $\bbR^5$, and in particular on inversive coordinate vectors, given by
\[
\varSigma(\bsv_1,\bsv_2):=\bsv_1 \bsQSigma \bsv_2.
\]
\end{defn}

\begin{rem}
In \cite{Wilker:Inversive}, this bilinear form is derived from the standard indefinite inner product on $\bbR^{4,1}$ and denoted as a ``product'' $\bsv_1 * \bsv_2$. For honest spheres $S_1,S_2$, the quantity $\varSigma(\bsv(S_1),\bsv(S_2))$ appeared earlier with the definition directly in terms of radii and the centers of $S_1,S_2$; it is called the \emph{separation} and denoted by $\varDelta(S_1,S_2)$ in \cite{Boyd:Osculatory}, and its negative is called the \emph{inclination} and denoted by $\gamma$ in \cite{Mauldon}. Closely related concepts can be traced back to \cite{Clifford}, \cite{Darboux}, \cite{Lachlan}.
\end{rem}

Identifying the M\"obius 3-space $\hat\bbR^3$ with the boundary $\partial\Hyp^4$ of the upper half-space model of the hyperbolic 4-space $\Hyp^4$, each oriented inversive sphere $S$ can be regarded as the boundary of an oriented 3-dimensional hyperbolic hyperplane $H$ which cuts out a 4-dimensional hyperbolic half-space whose limit at infinity is the orienting region $B$ bounded by $S$. Given two oriented inversive spheres $S_1,S_2$ with the inversive coordinate vectors $\bsv_1,\bsv_2$, the inversive product $\varSigma(\bsv_1,\bsv_2)$ encodes the quantitative data of their relative positions \cite{Wilker:Inversive} as follows.

Inversive spheres $S_1,S_2$ intersects if and only if the corresponding hyperplanes $H_1,H_2$ intersects in $\Hyp^4$. %; an elliptic transformation takes one hyperplane to another, possibly followed by an inversion to give the right orientation.
The relative position of $S_1,S_2$ is captured by the angle $\theta$ between them, measured in the symmetric difference $B_1 \triangle B_2$ of the orienting regions; this angle coincides with the dihedral angle between oriented hyperbolic hyperplanes $H_1,H_2$, measured in the symmetric difference of the corresponding hyperbolic halfspaces. When $S_1,S_2$ intersects, we have
\[
\varSigma(\bsv_1,\bsv_2)=\cos\theta.
\]

Inversive spheres $S_1,S_2$ are tangent if and only if the corresponding hyperplanes $H_1,H_2$ are tangent at $\partial\Hyp^4$. %; a parabolic transformation takes one hyperplane to another, possibly followed by an inversion to give the right orientation.
$S_1,S_2$ are said to be \emph{nested} or \emph{internally tangent} if the orienting regions $B_1,B_2$ are nested, and said to be \emph{not nested} or \emph{externally tangent} if $B_1,B_2$ are not nested. When $S_1,S_2$ are tangent, we have
\[
\varSigma(\bsv_1,\bsv_2)=\begin{cases}
+1&\text{if $S_1,S_2$ are nested,}\\
-1&\text{if $S_1,S_2$ are not nested.}\\
\end{cases}
\]

Inversive spheres $S_1,S_2$ are disjoint if and only if the corresponding hyperplanes $H_1,H_2$ are disjoint in $\Hyp^4\cup\partial\Hyp^4$. %; a hyperbolic transformation takes one hyperplane to another, possibly followed by an inversion to give the right orientation.
The relative position of $S_1,S_2$ is captured by the hyperbolic distance $\delta$ between $H_1,H_2$. $S_1,S_2$ are said to be \emph{nested} or \emph{internally disjoint} if the orienting regions $B_1,B_2$ are nested, and said to be \emph{not nested} or \emph{externally disjoint} if $B_1,B_2$ are not nested. When $S_1,S_2$ are disjoint, we have
\[
\varSigma(\bsv_1,\bsv_2)=\begin{cases}
+\cosh \delta &\text{if $S_1,S_2$ are nested,}\\
-\cosh \delta &\text{if $S_1,S_2$ are not nested.}\\
\end{cases}
\]

\subsection{M\"obius Group Action}

The \emph{M\"obius group} $\Mob_3^\pm=\Mob^\pm(\hat\bbR^3)$ is defined to be the group of conformal/anti-conformal transformation on the M\"obius 3-space $\hat\bbR^3$. We write $\Mob_3^+=\Mob^+(\hat\bbR^3)$ for the subgroup of conformal transformations. By the classical theorem of Liouville, it is well-known that $\Mob_3^\pm$ is generated by inversions along spheres, and $\Mob_3^+$ is generated by rescaling, translations, and rotations. All sphere inversions are conjugates of the \emph{standard inversion} along the standard unit sphere by some element of $\Mob_3^+$.

The M\"obius group acts naturally on inversive coordinate vectors via a representation as a $5\times5$ matrix group, acting on the \emph{right} of inversive coordinate vectors by matrix multiplication. The matrices for the standard inversion, rescaling, translations, and rotations are written down in \cite{GLMWY:GG3}; we recall these matrices below.

\begin{itemize}
\item The standard inversion is represented by
\[
\medpmatrix{
0&1&0&0&0\\
1&0&0&0&0\\
0&0&1&0&0\\
0&0&0&1&0\\
0&0&0&0&1
}.
\]
\item The rescaling by the factor $t$ is represented by
\[
\medpmatrix{
t&0&0&0&0\\
0&1/t&0&0&0\\
0&0&1&0&0\\
0&0&0&1&0\\
0&0&0&0&1
}.
\]
\item The translation by $\bsw=(x,y,z)$ is represented by
\[
\medpmatrix{
1&0&0&0&0\\
x^2+y^2+z^2&1&x&y&z\\
2x&0&1&0&0\\
2y&0&0&1&0\\
2z&0&0&0&1
}.
\]
\item The rotation about the unit vector $(x,y,z)$ by the angle $\theta$ is represented by
\[
\hspace{10mm}\bsR=\medpmatrix{
1&0&0&0&0\\
0&1&0&0&0\\
0&0&x^2(1-\cos\theta)+\cos\theta&xy(1-\cos\theta)+z\cos\theta&xz(1-\cos\theta)-y\cos\theta\\
0&0&yx(1-\cos\theta)-z\cos\theta&y^2(1-\cos\theta)+\cos\theta&yz(1-\cos\theta)+x\cos\theta\\
0&0&zx(1-\cos\theta)+y\cos\theta&zy(1-\cos\theta)-x\cos\theta&z^2(1-\cos\theta)+\cos\theta
}.
\]
i.e.\;rotation matrices $\bsR \in \SO_3(\bbR)$ (acting on the right of the row vectors) are embedded in $5\times5$ matrix as the lower right $3\times3$ minors.
\end{itemize}

Since the standard inversion, rescaling, translations, and rotations generate the M\"obius group $\Mob_3^\pm$, these matrices specify the representation. From now on, we identify $\Mob_3^\pm$ with the $5\times5$ matrix group generated by these matrices, which is precisely the image of $\Mob_3^\pm$ under this representation.

One of the most important features of the inversive product $\varSigma$ is the invariance under the action of the M\"obius group $\Mob_3^\pm$. Identifying the M\"obius 3-space $\hat\bbR^3$ with the boundary of the hyperbolic 4-space $\Hyp^4$, the conformal/anti-conformal action of $\Mob_3^\pm$ on $\hat\bbR^3$ extends to the isometric action of $\Mob_3^\pm$ on $\Hyp^4$. Hence, with the concrete interpretation of the inversive product in terms of the angle and the hyperbolic distances, the invariance of $\varSigma$ under $\Mob_3^\pm$-action is intuitively obvious.

\begin{lem} \label{lem:MobiusInvariance}
If a $5\times5$ matrix $\bsM$ represents a M\"obius transformation via the representation above, we have
\[
\bsM\bsQSigma\bsM^\trans=\bsQSigma, \quad \bsM^\trans\bsQW\bsM=\bsQW. 
\]
In particular, the inversive product is invariant under the M\"obius group action; namely, if $\bsv_1$ and $\bsv_2$  are inversive coordinates of inversive spheres and $\bsM$ represents a M\"obius transformation, then $\varSigma(\bsv_1\bsM,\bsv_2\bsM)=\varSigma(\bsv_1,\bsv_2)$.
\end{lem}

\begin{proof}
By direct calculation, we can check the invariance under the standard inversions, rescaling, translations and rotations using the matrices above; these transformations generate the M\"obius group $\Mob_3^\pm$.
\end{proof}

%%%
\section{Orthoplicial Platonic Configurations and Platonic Group} \label{sec:Platonic}

\subsection{Platonic Configurations} \label{ssec:PlatonicConfig}

We define the \emph{standard orthoplicial Platonic configuration} $\scrV_0$ in the M\"obius 3-space $\hat \bbR^3$ to be an ordred collection of eight spheres in the following table, which lists the inversive coordinates $(a, b, \hat x,\hat y,\hat z)$ of each constituent sphere $S_k$, as well as its oriented radius and its center if $S_k$ is not planar.

\vspace{3mm}
{\small
\begin{center}
\begin{tabular}{|l|rrrrr|rrrr|}
\hline
$k$ & $a$ & $b$ & $\hat x$ & $\hat y$ & $\hat z$ & $r$ & $c_x$ & $c_y$ & $c_z$\\
\hline
$1$&$2$&$0$&$0$&$0$&$1$&-&-&-&-\\
$2$&$2$&$0$&$0$&$0$&$-1$&-&-&-&-\\
$3$&$1$&$1$&$\sqrt{2}$&$0$&$0$&$1$&$\sqrt{2}$&$0$&$0$\\
$4$&$1$&$1$&$0$&$\sqrt{2}$&$0$&$1$&$0$&$\sqrt{2}$&$0$\\
$5$&$0$&$2$&$0$&$0$&$-1$&$1/2$&$0$&$0$&$-1/2$\\
$6$&$0$&$2$&$0$&$0$&$1$&$1/2$&$0$&$0$&$1/2$\\
$7$&$1$&$1$&$-\sqrt{2}$&$0$&$0$&$1$&$0$&$0$&$0$\\
$8$&$1$&$1$&$0$&$-\sqrt{2}$&$0$&$1$&$0$&$0$&$0$\\
\hline
\end{tabular}
\end{center}
}
\begin{figure}[h]
\begin{center}
\includegraphics[height=43mm, width=55mm]{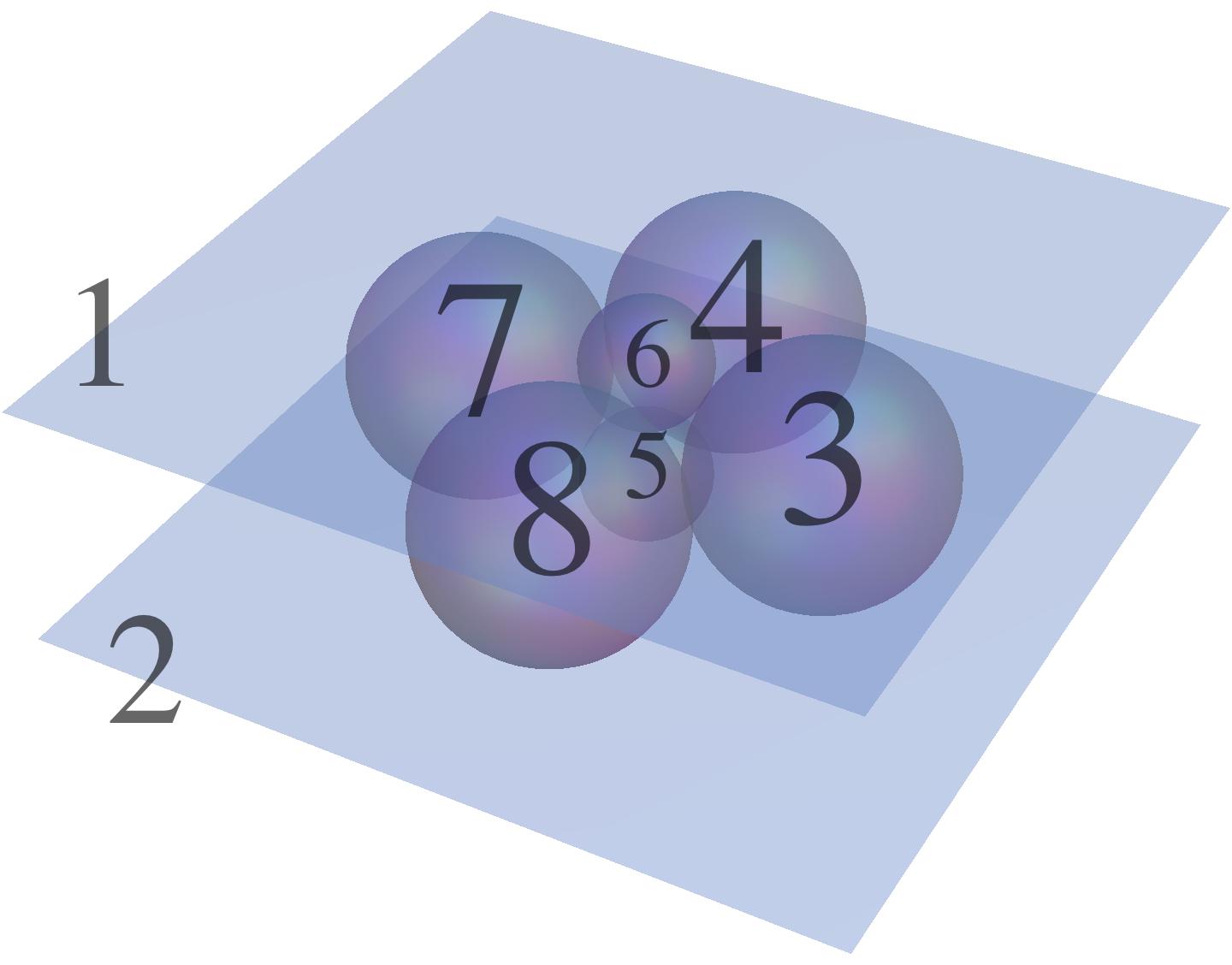}
\hspace{6mm}
\includegraphics[height=43mm, width=55mm]{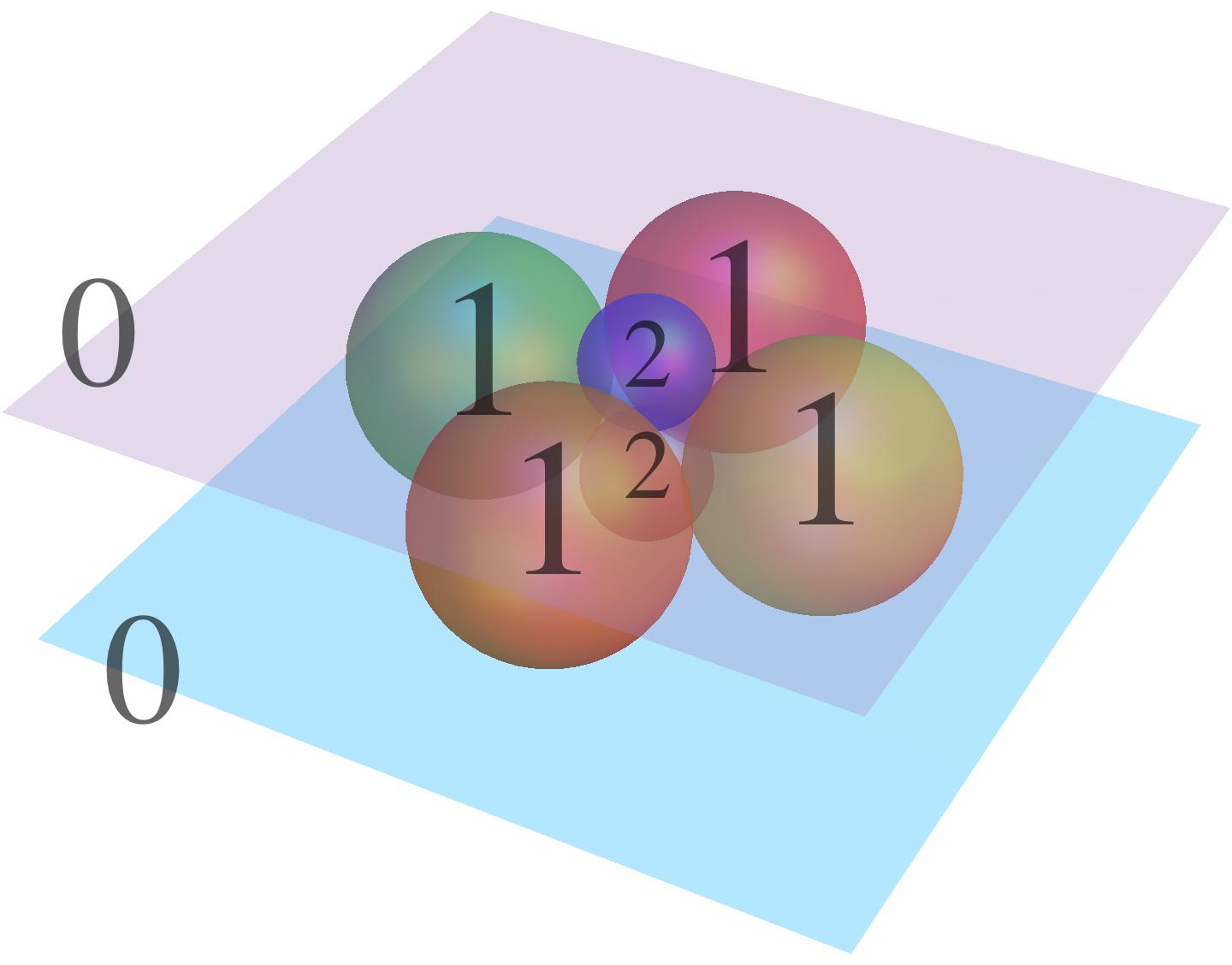} 
\end{center}
\caption{The standard configuration $\scrV_0$}
\label{fig:V0}
\end{figure}

\noindent
The configuration $\scrV_0$ is shown in \hyperref[fig:V0]{Figure~\ref*{fig:V0}}, with each sphere $S_k$ labeled by $k$ (left) and by its bend $b_k=b(S_k)$ (right); we note that $S_1$ and $S_2$ are planes in $\bbR^3$, i.e.\;spheres through $\infty$, while $S_3, \cdots, S_8$ are honest spheres in $\bbR^3$. By inspection, one can verify that distinct spheres $S_i$ and $S_j$ are tangent if $i-j \not\equiv 0 \pmod 4$, and are disjoint if $i-j \equiv 0 \pmod 4$. It follows that the tangency graph for this configuration of spheres is the 4-orthoplicial graph, i.e.\;isomorphic to the 1-skeleton of the 4-orthoplex, also known as the \emph{16-cell} or the \emph{4-dimensional cross-polytope}.

\begin{defn}
An \emph{orthoplicial Platonic configuration} $\scrV$ in the M\"obius 3-space $\hat \bbR^3$ is defined to be a collection of eight spheres, which is conformally or anti-conformally equivalent to the standard configuration $\scrV_0$.
\end{defn}
Any orthoplicial configuration consists of eight spheres, bounding their respective orienting regions with disjoint interiors, such that its tangency graph is the 4-orthoplicial graph. For brevity, we may simply call an orthoplicial Platonic configuration as an orthoplicial configuration or a Platonic configuration, when no confusion should arise in a given context.

We shall label the constituent spheres as $S_1, S_2, \cdots, S_8$, so that distinct spheres $S_i,S_j \in \scrV$ are tangent if $i-j \not\equiv 0 \pmod 4$, and are disjoint if $i-j \equiv 0 \pmod 4$; such an ordering is said to be \emph{admissible}. The standard configuraton $\scrV_0$ is equipped with an admissible ordering. Choosing an admissible ordering is equivalent to choosing an ordered quadruple $\scrF=\{S_1,S_2,S_3,S_4\} \subset \scrV$ of pairwise tangent spheres; the remaining spheres are then unambiguously ordered. There are 16 unordered quadruples corresponding to the 16 facets of the 4-orthoplex, each with 24 ways to order them; hence, every orthoplicial configuration has 384 distinct admissible ordering. For any admissibly ordered configuration $\scrV$ and another configuration $\scrV'$, if we choose a M\"obius transformation that takes $\scrV$ to $\scrV'$, then $\scrV'$ inherits an admissible ordering from $\scrV$; choosing a M\"obius transformation that permutes the constituent spheres of $\scrV$, we obtain a new admissible ordering of $\scrV$.

An admissibly ordered orthoplicial configuration $\scrV$ can be specified directly by an ordered list of the inversive coordinate vectors of the constituent spheres.

\begin{defn}[$V$-matrix]
Given an admissibly ordered orthoplicial configuration $\scrV$, its \emph{$V$-matrix} is an $8 \times 5$ matrix $\bsV=\bsV(\scrV)$ whose $k$-th row is the inversive coordinate vector $\bsv_k=\bsv(S_k)$ of the $k$-th constituent sphere $S_k$ in $\scrV$.
\end{defn}

More efficiently, we can encode such a configuration $\scrV$ by a $5 \times 5$ matrix; the analogous $4 \times 4$ matrix was introduced in \cite{GM} to encode an octahedral configuration of six circles in the M\"obius plane $\hat \bbR^2$.

\begin{defn}[$F$-matrix]
Given an admissibly ordered orthoplicial configuration, its \emph{$F$-matrix} is a $5 \times 5$ matrix $\bsF=\bsF(\scrV)$ whose $k$-th row is $\bsv_k=\bsv(S_k)$ for $k=1, \cdots, 4$, and whose $5$-th row is the \emph{antipodal vector} $\bsv_\mu=\bsv_\mu(\scrV)$, defined by
\[
\bsv_\mu:=\frac{1}{\,2\,}(\bsv_1+\bsv_5).
\]
\end{defn}

Note that, for each unordered orthoplicial configuration $\scrV$, choosing a particular $F$-matrix is equivalent to choosing an ordered quadruple $\scrF \subset \scrV$ of pairwise tangent spheres, and hence equivalent to choosing one of 384 admissible orderings; the spheres in $\scrF$ are precisely the ones whose inversive coordinate vectors $\bsv_k$ appear in the first four rows of the $F$-matrix.

Although the use of the antipodal vector in the definition may seem a bit artificial at first, the $F$-matrices is a natural and useful tool to encode orthoplicial configurations. The following observation is crucial for the utility of $F$-matrices.

\begin{lem} \label{lem:Antipodal}
Let $\scrV$ and $\scrV'$ be admissibly ordered orthoplicial configurations, with the $V$-matrices $\bsV$ and $\bsV'$, the antipodal vectors $\bsv_\mu$ and $\bsv'_\mu$, and the $F$-matrices $\bsF$ and $\bsF'$, respectively. If $\bsM$ represents the M\"obius transformation taking $\scrV$ to $\scrV'$, i.e.\;$\bsV'=\bsV \bsM$, then we have
\[
\bsv'_\mu=\bsv_\mu \bsM
\]
and hence
\[
\bsF'=\bsF \bsM.
\]
\end{lem}

\begin{proof}
Writing $\bsv_k$ and $\bsv'_k=\bsv_k \bsM$ for the $k$-th row vector of $\bsV$ and $\bsV'$ respectively, we have $2\bsv'_\mu=\bsv'_1+\bsv'_5=\bsv_1\bsM+\bsv_5\bsM=(\bsv_1+\bsv_5)\bsM=2\bsv_\mu \bsM$
by linearity, and hence $\bsv'_\mu=\bsv_\mu \bsM$.
\end{proof}

Unlike the first four rows of $F$-matrices, the antipodal vector in the last row of $F$-matrices is independent of the choice of admissible orderings.

\begin{coro} \label{coro:Antipodal}
For any admissibly ordered orthoplicial configuration $\scrV$,
\begin{align} \label{eqn:Antipodal}
\bsv_\mu=\frac{1}{\,2\,}(\bsv_1+\bsv_5)=\frac{1}{\,2\,}(\bsv_2+\bsv_6)=\frac{1}{\,2\,}(\bsv_3+\bsv_7)=\frac{1}{\,2\,}(\bsv_4+\bsv_8).
\end{align}
Hence, the $V$-matrix $\bsV=\bsV(\scrV)$ and the $F$-matrix $\bsF=\bsF(\scrV)$ satisfy $\bsV=\bsD \bsF
$, where the \emph{decompression matrix} $\bsD$ is given by
\[
\bsD:=\medpmatrix{
1&0&0&0&0\\
0&1&0&0&0\\
0&0&1&0&0\\
0&0&0&1&0\\
-1&0&0&0&2\\
0&-1&0&0&2\\
0&0&-1&0&2\\
0&0&0&-1&2
}.
\]
\end{coro}

\begin{proof}
For each of $j=1,2,3$, there is a M\"obius transformation that takes the spheres in $\scrV$ to themselves, taking the disjoint pair $S_1$, $S_5$ to another disjoint pair $S_{1+j}$, $S_{5+j}$; so, the equalities \hyperref[eqn:Antipodal]{(\ref*{eqn:Antipodal})} follow from \hyperref[lem:Antipodal]{Lemma~\ref*{lem:Antipodal}}. The equality $\bsV=\bsD \bsF$ then follows immediately.
\end{proof}

\begin{exa} \label{exa:V0}
The $V$-matrix and the $F$-matrix of the standard configuration $\scrV_0$ are
\[
\bsV_0:=\medpmatrix{
2&0&0&0&1\\
2&0&0&0&-1\\
1&1&\sqrt{2}&0&0\\
1&1&0&\sqrt{2}&0\\
0&2&0&0&-1\\
0&2&0&0&1\\
1&1&-\sqrt{2}&0&0\\
1&1&0&-\sqrt{2}&0
},
\qquad
\bsF_0:=\medpmatrix{
2&0&0&0&1\\
2&0&0&0&-1\\
1&1&\sqrt{2}&0&0\\
1&1&0&\sqrt{2}&0\\
1&1&0&0&0
}.
\]
\end{exa}

The standard configuration $\scrV_0$ is quite special with two of its constituent spheres being planes with bend zero, i.e.\;spheres through $\infty$; indeed, it can be shown that such a configuration is unique up to Euclidean similarity. In this article, we will work mostly with configurations in which one constituent sphere has a negative bend and bounds a ball, as the complement of its orienting region, that contains the remaining seven constituent spheres.

\begin{exa} \label{exa:V1}
Inverting the standard configuration $\scrV_0$ along the 4th sphere $S_4$, we obtain another orthoplicial Platonoic configuration which we denote by $\scrV_1$; inverting along $S_3, S_7, S_8$ yield configurations that are equivalent to $\scrV_1$ up to Euclidean isometry. The configuration $\scrV_1$ is depicted in \hyperref[fig:V1+V7d]{Figure~\ref*{fig:V1+V7d}} (left), with each constituent sphere $S_k$ labeled by its bend $b_k=b(S_k)$. The $V$-matrix and the $F$-matrix of the configuration $\scrV_1$ are
\[
\bsV_1:=\medpmatrix{
4&2&0&2\sqrt{2}&1\\
4&2&0&2\sqrt{2}&-1\\
3&3&\sqrt{2}&2\sqrt{2}&0\\
-1&-1&0&-\sqrt{2}&0\\
2&4&0&2\sqrt{2}&-1\\
2&4&0&2\sqrt{2}&1\\
3&3&-\sqrt{2}&2\sqrt{2}&0\\
7&7&0&5\sqrt{2}&0
},
\qquad
\bsF_1:=\medpmatrix{
4&2&0&2\sqrt{2}&1\\
4&2&0&2\sqrt{2}&-1\\
3&3&\sqrt{2}&2\sqrt{2}&0\\
-1&-1&0&-\sqrt{2}&0\\
3&3&0&2\sqrt{2}&0
}.
\]
\end{exa}

\vspace{0mm}
\begin{exa} \label{exa:V7d}
The orthoplicial Platonic configuration in \hyperref[fig:OASP]{Figure~\ref*{fig:OASP}} is another configuration in which one constituent sphere has a negative bend; let us denote this configuration by $\scrV_{7\rmd}$. The configuration $\scrV_{7\rmd}$ is depicted again in \hyperref[fig:V1+V7d]{Figure~\ref*{fig:V1+V7d}} (right), with each constituent sphere $S_k$ labeled by its bend $b_k=b(S_k)$. The $V$-matrix and the $F$-matrix of the configuration $\scrV_{7\rmd}$ are
\[
\bsV_{7\rmd}:=\medpmatrix{
34&20&18\sqrt{2}&2\sqrt{2}&-5\\
18&12&10\sqrt{2}&2\sqrt{2}&-3\\
29&17&15\sqrt{2}&2\sqrt{2}&-6\\
-11&-7&-6\sqrt{2}&-\sqrt{2}&2\\
32&22&18\sqrt{2}&2\sqrt{2}&-7\\
48&30&26\sqrt{2}&2\sqrt{2}&-9\\
37&25&21\sqrt{2}&2\sqrt{2}&-6\\
77&49&42\sqrt{2}&5\sqrt{2}&-14
},
\qquad
\bsF_{7\rmd}:=\medpmatrix{
34&20&18\sqrt{2}&2\sqrt{2}&-5\\
18&12&10\sqrt{2}&2\sqrt{2}&-3\\
29&17&15\sqrt{2}&2\sqrt{2}&-6\\
-11&-7&-6\sqrt{2}&-\sqrt{2}&2\\
33&21&18\sqrt{2}&2\sqrt{2}&-6
}.
\]
\end{exa}

\begin{figure}[h]
\begin{center}
\includegraphics[trim = 70px 91px 51px 62px, clip, width=55mm]{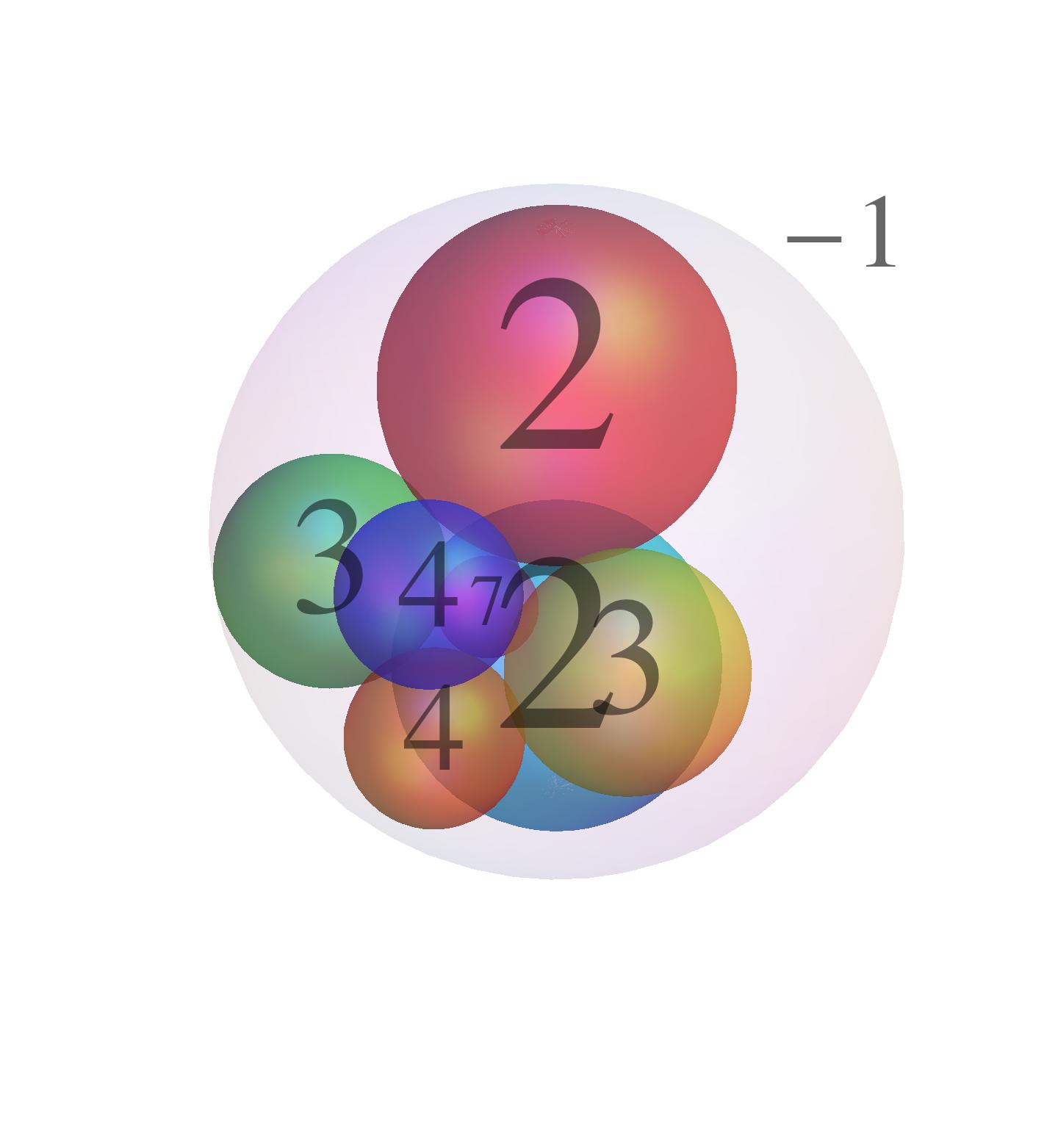}
\hspace{6mm}
\includegraphics[trim = 70px 91px 51px 62px, clip, width=55mm]{B4figF7dAg0ob.jpg}
\end{center}
\caption{The orthoplicial configurations $\scrV_1$ (left) and $\scrV_{7\rmd}$ (right)}
\label{fig:V1+V7d}
\end{figure}

\subsection{Descartes-Guettler-Mallows Theorem}

Guettler and Mallows obtained a certain analogue \cite[Thm.\,1]{GM} of Descartes' Theorem in the context of octahedral configurations of six circles. We shall now discuss an analogous theorem for orthoplicial configurations of eight spheres. Let us first define two matrices:
\[
\bsGSigmaF:=\medpmatrix{
1&-1&-1&-1&-1\\
-1&1&-1&-1&-1\\
-1&-1&1&-1&-1\\
-1&-1&-1&1&-1\\
-1&-1&-1&-1&-1
},
\qquad
\bsQF:=2 \bsGSigmaF^{-1}=\medpmatrix{
1&0&0&0&-1\\
0&1&0&0&-1\\
0&0&1&0&-1\\
0&0&0&1&-1\\
-1&-1&-1&-1&2
}.
\]

\begin{defn}
The \emph{orthoplicial Descartes form} is defined to be the quinternary quadratic form $F$ with indefinite signature $(4,1)$, associated to the symmetric matrix $\bsQF$; namely, the form $F$ on a quintuple $\bszeta=(\zeta_1,\zeta_2,\zeta_3,\zeta_4,\zeta_\mu)^\trans$ is defined by
\begin{align} \label{eqn:Fzeta}
F(\bszeta):=\bszeta^\trans \bsQF \bszeta = 2\zeta_\mu^2-2\zeta_\mu(\zeta_1+\zeta_2+\zeta_3+\zeta_4)+(\zeta_1^2+\zeta_2^2+\zeta_3^2+\zeta_4^2).
\end{align}
We denote the orthogonal and special orthogonal group of $F$ by $\OF$ and $\SOF$.
\end{defn}

The matrix $\bsGSigmaF$ can be regarded as the \emph{Gramian} of $F$-matrices with respect to the inversive product $\varSigma$, which is reflected in the choice of our notation; more precisely, its significance can be stated as follows.

\begin{lem} \label{lem:Gramian}
For any admissibly ordered orthoplicial configuration $\scrV$, its $F$-matrix $\bsF=\bsF(\scrV)$ is non-singular and satisfies
\begin{align} \label{eqn:Gramian}
\bsF \bsQSigma \bsF^\trans=\bsGSigmaF.
\end{align}
\end{lem}

\begin{proof}
The non-singularity of $\bsF$ is implicit in the equation \hyperref[eqn:Gramian]{(\ref*{eqn:Gramian})}, since $\bsG_F$ is invertible. For the $F$-matrix $\bsF_0$ of the standard configuration, the direct calculation yields $\bsF_0 \bsQSigma \bsF_0^\trans=\bsGSigmaF$ as desired. For the general case, let $\bsM$ be a matrix representing a M\"obius transformation that takes the standard configuration $\scrV_0$ to the configuration $\scrV$ so that $\bsF = \bsF_0 \bsM$ by \hyperref[lem:Antipodal]{Lemma~\ref*{lem:Antipodal}}. Then, we have
\[
\bsF \bsQSigma \bsF^\trans=(\bsF_0 \bsM) \bsQSigma (\bsF_0 \bsM)^\trans=\bsF_0 (\bsM \bsQSigma \bsM^\trans) \bsF_0^\trans=\bsF_0 \bsQSigma \bsF_0^\trans
\]
by \hyperref[lem:MobiusInvariance]{Lemma~\ref*{lem:MobiusInvariance}}, i.e.\;the $\Mob_3^\pm$-invariance of the inversive product $\varSigma$.
\end{proof}

The orthoplicial Descartes form $F$ can be regarded as the analogue of the so-called Descartes quadratic form for tetrahedral configuration of four pairwise tangent circles in the M\"obius plane $\hat \bbR^2$; we now establish the analogue of Descartes' Theorem, stated in the matrix form as follows.

\begin{thm}[Orthoplicial Descartes-Guettler-Mallows Theorem] \label{thm:DGM}
For any admissibly ordered orthoplicial configuration $\scrV$, its $F$-matrix $\bsF=\bsF(\scrV)$ satisfies
\begin{align} \label{eqn:DGM}
\bsF^\trans \bsQF \bsF=\bsQW.
\end{align}
In particular, writing $\bsa$, $\bsb$, $\hat\bsx$, $\hat\bsy$, $\hat\bsz$ for the 1st, 2nd, 3rd, 4th, 5th column vectors of the $F$-matrix, we have quadratic equations
\begin{align} \label{eqn:DGM-diag}
F(\bsa)=0, \quad
F(\bsb)=0, \quad
F(\hat\bsx)=2, \quad
F(\hat\bsy)=2, \quad
F(\hat\bsz)=2.
\end{align}
\end{thm}

\begin{proof}
Inverting both sides of the equation $\bsF \bsQSigma \bsF^\trans=\bsGSigmaF$ from \hyperref[lem:Gramian]{Lemma~\ref*{lem:Gramian}} and scaling by the factor 2, we have
\[
\bsQF=2\bsGSigmaF^{-1}=2(\bsF \bsQSigma \bsF^\trans)^{-1}=(\bsF^\trans)^{-1} (2\bsQSigma^{-1}) \bsF^{-1}
=(\bsF^\trans)^{-1} \bsQW \bsF^{-1}.
\]
Multiplying both sides of the equality $\bsQF=(\bsF^\trans)^{-1} \bsQW \bsF^{-1}$ on the left by $\bsF^\trans$ and on the right by $\bsF$, we obtain the matrix equation \hyperref[eqn:DGM]{(\ref*{eqn:DGM})}, whose diagonal entries are precisely the quadratic equations \hyperref[eqn:DGM-diag]{(\ref*{eqn:DGM-diag})}.
\end{proof}

\subsection{Platonic Group}

As we have seen, an orthoplicial configuration $\scrV$ admits 384 admissible ordering; it is easy to see that they correspond bijectively to 384 elements of the full symmetry group of the 4-orthoplex. The action of the orthoplicial symmetries on $V$-matrices is given simply by the permutation representation, i.e.\;via an $8 \times 8$ matrix group acting on the left of $V$-matrices and permuting their row vectors. We shall work with the $F$-matrices instead; the corresponding action of the orthoplicial symmetries on $F$-matrices is given via a $5 \times 5$ matrix group acting on the left of $F$-matrices.

\begin{defn}
The \emph{orthoplicial Platonic group} $\Platonic$ is defined to be the $5 \times 5$ matrix group generated by $\calR:=\{\bsR_1, \bsR_2, \bsR_3, \bsR_4\}$, consisting of the following 4 matrices:
\begin{align*}
\bsR_1:=\medpmatrix{
0&1&0&0&0\\
1&0&0&0&0\\
0&0&1&0&0\\
0&0&0&1&0\\
0&0&0&0&1
},
&\quad
\bsR_2:=\medpmatrix{
1&0&0&0&0\\
0&0&1&0&0\\
0&1&0&0&0\\
0&0&0&1&0\\
0&0&0&0&1
},\\
\bsR_3:=\medpmatrix{
1&0&0&0&0\\
0&1&0&0&0\\
0&0&0&1&0\\
0&0&1&0&0\\
0&0&0&0&1
},
&\quad
\bsR_4:=\medpmatrix{
1&0&0&0&0\\
0&1&0&0&0\\
0&0&1&0&0\\
0&0&0&-1&2\\
0&0&0&0&1
}.
\end{align*}
\end{defn}

Although the $F$-matrices only contains four inversive coordinate vectors explicitly, it is easy to read off the effect of $\bsR_1,\bsR_2,\bsR_3,\bsR_4$ on all eight coordinate vectors. $\bsR_1$ interchanges $\bsv_1$ and $\bsv_2$, and hence $\bsv_5$ and $\bsv_6$, while fixing $\bsv_3,\bsv_4$ and $\bsv_7,\bsv_8$. $\bsR_2$ interchanges $\bsv_2$ and $\bsv_3$, and hence $\bsv_6$ and $\bsv_7$, while fixing $\bsv_1,\bsv_4$ and $\bsv_5,\bsv_8$. $\bsR_3$ interchanges $\bsv_3$ and $\bsv_4$, and hence $\bsv_7$ and $\bsv_8$, while fixing $\bsv_1,\bsv_2$ and $\bsv_5,\bsv_6$. $\bsR_4$ interchanges $\bsv_4$ and $\bsv_8$, while fixing $\bsv_1,\bsv_2,\bsv_3$ and hence $\bsv_5,\bsv_6,\bsv_7$; this follows from $2\bsv_\mu=\bsv_4+\bsv_8$, appeared as \hyperref[eqn:Antipodal]{(\ref*{eqn:Antipodal})} in \hyperref[coro:Antipodal]{Corollary~\ref*{coro:Antipodal}}. 

Since the Platonic group $\Platonic$ is just a faithful representation of the full orthopliclal symmetry group, it admits a presentation as the Coxeter-Weyl group $BC_4$. Our choice of generators are indeed aligned to this presentation: a complete set of relations for the group $\Platonic$ with respect to $\calR$ is given by
\begin{align*}
\begin{matrix}
\bsR_1^2=\bsR_2^2=\bsR_3^2=\bsR_4^2=\bsI,\\
(\bsR_1\bsR_2)^3=(\bsR_2\bsR_3)^3=(\bsR_3\bsR_4)^4
=(\bsR_1\bsR_3)^2=(\bsR_1\bsR_4)^2=(\bsR_2\bsR_4)^2
=\bsI.
\end{matrix}
\end{align*}

\begin{rem}
Once we fix an orthoplicial configuration, orthoplicial symmetries can be realized by M\"obius transformations, acting on the right of $F$-matrices. Such realizations of orthoplicial symmetries depend on configurations; taking different configurations results in conjugation by a M\"obius transformation that takes one configuration to another. On the other hand, the Platonic group $\Platonic$ is independent of a choice of an orthoplicial configuration and acts on the left of $F$-matrices.
\end{rem}

\begin{lem}
The orthoplicial Platonic group $\Platonic$ is a subgroup of $\OF(\bbZ)$; namely, for any matrix $\bsP \in \Platonic$, we have $\det \bsP=\pm1$ and
\[
\bsP^\trans \bsQF \bsP=\bsQF.
\]
\end{lem}

\begin{proof}
For each generator $\bsR=\bsR_i \in \calR$, $\det \bsR_i=-1$ and $\bsR^\trans \bsQF \bsR=\bsQF$ can be checked by direct computation.
\end{proof}

\begin{defn}
The \emph{oriented orthoplicial Platonic group} $\Platonic^+<\Platonic$ is the subgroup consisting of matrices with determinant $+1$, i.e.\;$\Platonic^+:=\Platonic \cap \SOF^+(\bbZ)$.
\end{defn}

The oriented Platonic group $\Platonic^+$ corresponds to the oriented symmetry group of the 4-orthoplex, and it is the index 2 kernel of the determinant on  the Platonic group $\Platonic$. Since every generator $\bsR_i \in \calR$ of $\Platonic$ has determinant $-1$, it follows that $\Platonic^+$ consists of elements that can be, and can only be, written as even-length words in the generators $\calR=\{\bsR_i\}$ of $\Platonic$; hence $\Platonic^+$ is generated by $\{\bsR_i\bsR_j \mid \bsR_i, \bsR_j \in \calR\}$, which can easily be reduced to
\begin{align*}
\calR^+:=\{\bsR_1\bsR_i \mid \bsR_1 \neq \bsR_i \in \calR\}.
\end{align*}
using the relations $\bsR_i^2=\bsI$ for all $\bsR_i \in \calR$.

\subsection{Integral Configurations}

An orthoplicial Platonic configuration $\scrV$ is said to be \emph{integral} if the bends of all constituent spheres are integers. We write $\scrB(\scrV)$ for the set $\{b(S) \mid S \in \scrV\} \subset \bbZ$ of all integers appearing as bends of constituent spheres in $\scrV$, and write $\scrB^+(\scrV):=\scrB(\scrV) \cap \bbN \subset \bbN$. An integral Platonic configuration $\scrV$ is said to be \emph{primitive} if $\gcd \scrB(\scrV)=1$. The standard configuration $\scrV_0$ in \hyperref[exa:V0]{Example~\ref*{exa:V0}}, the configuration $\scrV_1$ in \hyperref[exa:V1]{Example~\ref*{exa:V1}}, and the configuration $\scrV_{7\rmd}$ in \hyperref[exa:V7d]{Example~\ref*{exa:V7d}} are examples of primitive orthoplicial Platonic configurations.

Given an admissibly ordered orthoplicial Platonic configuration $\scrV$, the second column vector $\bsb=\bsb(\scrV):=(b_1,b_2,b_3,b_4,b_\mu)^\trans$ of its $F$-matrix $\bsF$ will be referred to as the \emph{bend vector} of $\scrV$. The following lemma gives a somewhat subtle characterization of integral/primitive configurations in terms of the bend vector.

\begin{prop} \label{prop:IntegralPlatonic}
Let $\scrV$ be an orthoplicial Platonic configuration $\scrV$ with its bend vector $\bsb=\bsb(\scrV)=(b_1, b_2, b_3, b_4, b_\mu)$. Then, $\scrV$ is integral if and only if $\bsb$ is integral; moreover, $\scrV$ is primitive if and only if $\bsb$ is primitive.
\end{prop}

\begin{proof}
Let us first prove the statement on the integrality.
If the bend vector $\bsb$ is integral, the first four bends $b_1,b_2,b_3,b_4$, as well as the remaining complimentary bends $b_5=2b_\mu-b_1,b_6=2b_\mu-b_2,b_7=2b_\mu-b_3,b_8=2b_\mu-b_4$ by \hyperref[coro:Antipodal]{Corollary~\ref*{coro:Antipodal}}, are all integral. Conversely, suppose that $\scrV$ is integral, i.e.\;all bends $b_1, \cdots, b_8$ are integers. It follows immediately that $2b_\mu=b_1+b_5$ is an integer; we need to show that $2b_\mu$ is an even integer so that $b_\mu$ is an integer. Solving the quadratic equation
\[
F(\bsb)=2b_\mu^2-2b_\mu(b_1+b_2+b_3+b_4)+(b_1^2+b_2^2+b_3^2+b_4^2)=0.
\]
from \hyperref[eqn:DGM-diag]{(\ref*{eqn:DGM-diag})} in \hyperref[thm:DGM]{Theorem~\ref*{thm:DGM}} explicitly for $b_\mu$, we find
\begin{align} \label{eqn:Roots}
2b_\mu=b_1+b_2+b_3+b_4 \pm \sqrt{(b_1+b_2+b_3+b_4)^2-2(b_1^2+b_2^2+b_3^2+b_4^2)}.
\end{align}
Since $2b_\mu$ and $b_1+b_2+b_3+b_4$ are integers, it follows that
\[
\sqrt{(b_1+b_2+b_3+b_4)^2-2(b_1^2+b_2^2+b_3^2+b_4^2)}
\]
is an integer. Checking the parity, we have
\begin{align}\begin{split} \label{eqn:Parity}
\hspace{14mm}&\sqrt{(b_1+b_2+b_3+b_4)^2-2(b_1^2+b_2^2+b_3^2+b_4^2)}\\
&\equiv (b_1+b_2+b_3+b_4)^2-2(b_1^2+b_2^2+b_3^2+b_4^2)\\
&\equiv (b_1+b_2+b_3+b_4)^2 \equiv b_1+b_2+b_3+b_4 \pmod{2}.
\end{split}\end{align}
Returning to the equation \hyperref[eqn:Roots]{(\ref*{eqn:Roots})}, we now see that $2b_\mu$ is an even integer. Hence, $b_\mu$ is indeed an integer, and $\bsb=(b_1,b_2,b_3,b_4,b_\mu)^\trans$ is an integral vector as desired.

Let us now assume the integrality and prove the primitivity statement. If the bend vector $\bsb$ is not primitive, i.e.\;$d:=\gcd(b_1,b_2,b_3,b_4,b_\mu)\neq1$, then $d$ divides the first four bends $b_1,b_2,b_3,b_4,$ as well as the remaining complimentary bends $b_5=2b_\mu-b_1,b_6=2b_\mu-b_2,b_7=2b_\mu-b_3,b_8=2b_\mu-b_4$. Conversely, suppose that $\scrV$ is not primitive, i.e.\;$d:=\gcd(b_1,\cdots,b_8)\neq1$. It follows immediately that $d$ divides $2b_\mu=b_1+b_5$. If $d$ is odd, then $d$ must also divide $b_\mu$, and hence $\bsb=(b_1,b_2,b_3,b_4,b_\mu)$ is not primitive. So, let us now assume that $d$ is even. Then, all bends are even, say $b_k=2q_k$, $k=1,\cdots,8$. Together with \hyperref[eqn:Roots]{(\ref*{eqn:Roots})}, we obtain
\begin{align} \label{eqn:RootsAgain}
b_\mu=q_1+q_2+q_3+q_4 \pm \sqrt{(q_1+q_2+q_3+q_4)^2-2(q_1^2+q_2^2+q_3^2+q_4^2)}.
\end{align}
Since $2b_\mu, q_1,q_2,q_3,q_4$ are all integers, we deduce that
\[
\hspace{14mm}\sqrt{(q_1+q_2+q_3+q_4)^2-2(q_1^2+q_2^2+q_3^2+q_4^2)} \hspace{14mm}
\]
is also an integer, which must have, cf.\;\hyperref[eqn:Parity]{(\ref*{eqn:Parity})}, the same parity as $q_1+q_2+q_3+q_4$. Returning to the equation \hyperref[eqn:RootsAgain]{(\ref*{eqn:RootsAgain})}, we now see that $b_\mu$ is an even integer; components of $\bsb=(b_1,b_2,b_3,b_4,b_\mu)$ are all even, and $\bsb$ is not primitive.
\end{proof}

\begin{rem}
In the hindsight, the primitivity statement in \hyperref[prop:IntegralPlatonic]{Proposition~\ref*{prop:IntegralPlatonic}} further justifies our choice of the antipodal vector in the definition of $F$-matrix.
\end{rem}

%%%
\section{Orthoplicial Apollonian Packings and Apollonian Group} \label{sec:Apollonian}

\subsection{Apollonian Packings}

If $\scrF$ is a quadruple of pairwise tangent spheres, and $\scrV$ and $\scrV'$ are two orthoplicial configurations such that $\scrF=\scrV \cap \scrV'$, we say that the configurations $\scrV$ and $\scrV'$ are \emph{adjacent} along $\scrF$.

\begin{exa} \label{exa:Adjacent}
Let $\scrV'_0$ be an orthoplicial configuration given in the following table, which lists the inversive coordinates of each constituent sphere $S_k$, as well as its oriented radius and its oriented center if $S_k$ is not planar. The configuration $\scrV'_0$ shares the first four spheres $\scrF_0=\{S_1,S_2,S_3,S_4\}$ with the standard configuration $\scrV_0$ defined in \hyperref[ssec:PlatonicConfig]{\S\ref*{ssec:PlatonicConfig}}; indeed, $\scrV_0$ and $\scrV'_0$ are adjacent along $\scrF_0$ since $\scrV_0 \cap \scrV'_0=\scrF_0$.

\vspace{3mm}
{\small
\begin{center}
\begin{tabular}{|l|rrrrr|rrrr|}
\hline
$k$ & $a$ & $b$ & $\hat x$ & $\hat y$ & $\hat z$ & $r$ & $c_x$ & $c_y$ & $c_z$\\
\hline
$1$&$2$&$0$&$0$&$0$&$1$&-&-&-&-\\
$2$&$2$&$0$&$0$&$0$&$-1$&-&-&-&-\\
$3$&$1$&$1$&$\sqrt{2}$&$0$&$0$&$1$&$\sqrt{2}$&$0$&$0$\\
$4$&$1$&$1$&$0$&$\sqrt{2}$&$0$&$1$&$0$&$\sqrt{2}$&$0$\\
$5$&$8$&$2$&$2\sqrt{2}$&$2\sqrt{2}$&$-1$&$1/2$&$\sqrt{2}$&$\sqrt{2}$&$-1/2$\\
$6$&$8$&$2$&$2\sqrt{2}$&$2\sqrt{2}$&$1$&$1/2$&$\sqrt{2}$&$\sqrt{2}$&$1/2$\\
$7$&$9$&$1$&$\sqrt{2}$&$2\sqrt{2}$&$0$&$1$&$\sqrt{2}$&$2\sqrt{2}$&$0$\\
$8$&$9$&$1$&$2\sqrt{2}$&$\sqrt{2}$&$0$&$1$&$2\sqrt{2}$&$\sqrt{2}$&$0$\\
\hline
\end{tabular}
\end{center}
}

\begin{figure}[h]
\begin{center}
\includegraphics[height=66mm, width=86.4mm]{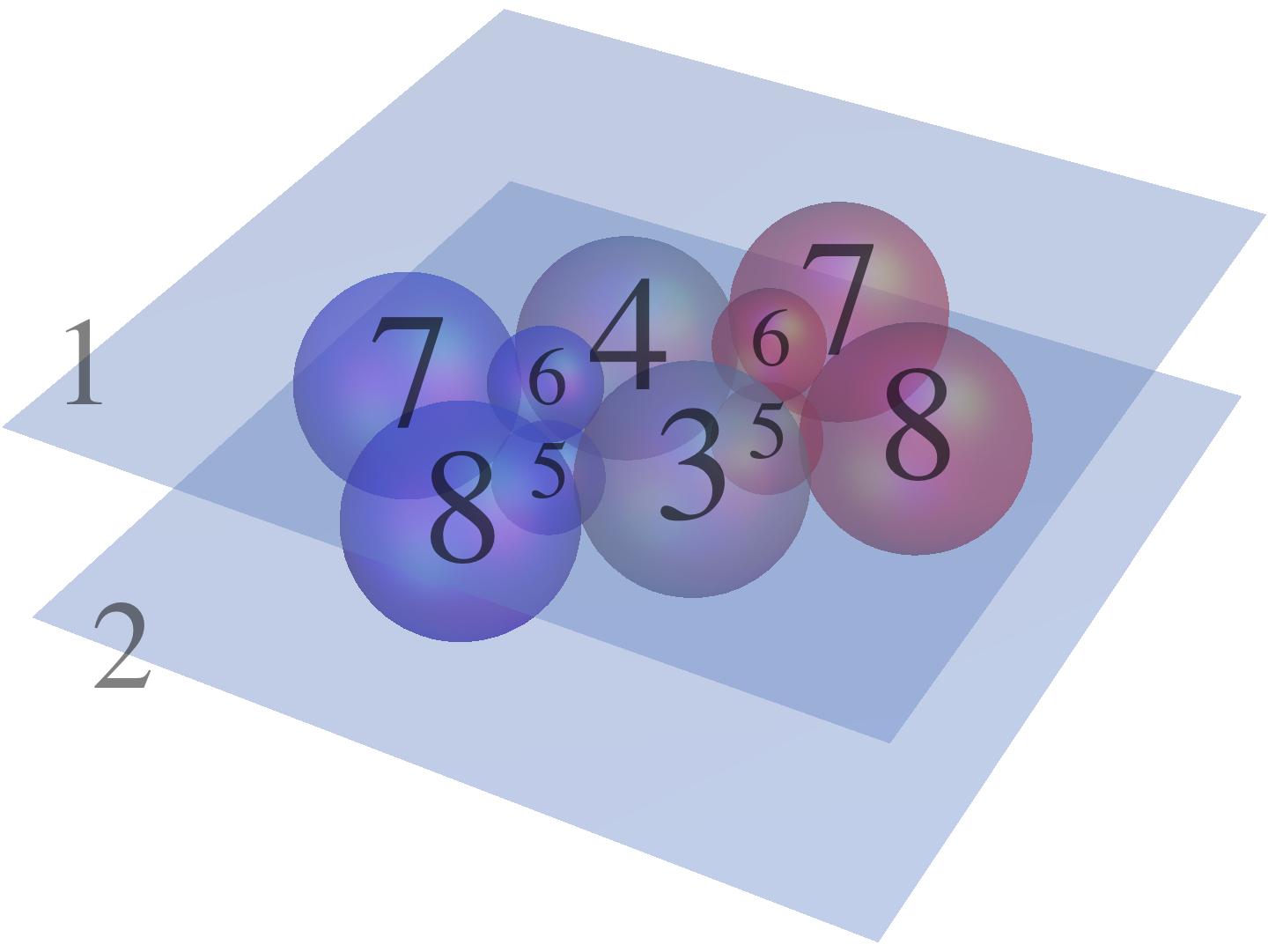}
\end{center}
\caption{The standard configuration $\scrV_0$ (gray and blue) and the configuration $\scrV$ (gray and red), adjacent to $\scrV_0$ along the common quadruple (gray), with each constituent sphere $S_k$ labeled by $k$}
\label{fig:Adjacent}
\end{figure}
\end{exa}

\begin{lem} \label{lem:Adjacent}
For any ordered quadruple $\scrF=\{S_1,S_2,S_3,S_4\}$ of pairwise tangent spheres, there exist exactly two admissibly ordered orthoplicial configurations $\scrV,\scrV'$ containing $\scrF$ as the first four spheres; they are adjacent to each other along $\scrF$, and are mapped from one to the other by the inversion about the dual sphere $S=S(\scrF)$ orthogonal to each sphere in the quadruple $\scrF$. 
\end{lem}

\begin{proof}
We shall first verify the claim directly for the quadruple $\scrF_0$ shared by the standard configuration $\scrV_0$ and the configuration $\scrV'_0$ in \hyperref[exa:Adjacent]{Example~\ref*{exa:Adjacent}}. For any admissibly orthoplicial configuration containing $\scrF_0$ as the first four spheres, its $F$-matrix is given by a matrix of the form
\[
\bsF=\medpmatrix{
2&0&0&0&1\\
2&0&0&0&-1\\
1&1&\sqrt{2}&0&0\\
1&1&0&\sqrt{2}&0\\
a_\mu&b_\mu&\hat x_\mu&\hat y_\mu&\hat z_\mu
}.
\]
By \hyperref[thm:DGM]{Theorem~\ref*{thm:DGM}}, this matrix must satisfy the equation \hyperref[eqn:DGM]{(\ref*{eqn:DGM})}, i.e.\;$\bsF^\trans \bsQF \bsF=\bsQW$; in particular, the equations \hyperref[eqn:DGM-diag]{(\ref*{eqn:DGM-diag})} are quadratic in one variable with solutions
\[
a_\mu=3\pm2, \quad
b_\mu=1, \quad
\hat x_\mu=\frac{1}{\,2\,}(\sqrt{2}\pm\sqrt{2}), \quad
\hat y_\mu=\frac{1}{\,2\,}(\sqrt{2}\pm\sqrt{2}), \quad
\hat z_\mu=0.
\]
By inspecting the possible sign combinations, we find that there are exactly two $F$-matrices of the above form, satisfying the full matrix equation \hyperref[eqn:DGM]{(\ref*{eqn:DGM})}:
\[
\bsF_0=\medpmatrix{
2&0&0&0&1\\
2&0&0&0&-1\\
1&1&\sqrt{2}&0&0\\
1&1&0&\sqrt{2}&0\\
1&1&0&0&0
},
\qquad
\bsF'_0:=\medpmatrix{
2&0&0&0&1\\
2&0&0&0&-1\\
1&1&\sqrt{2}&0&0\\
1&1&0&\sqrt{2}&0\\
5&1&\sqrt{2}&\sqrt{2}&0
}.
\]
The first matrix $\bsF_0$ is the $F$-matrix of the standard configuration $\scrV_0$, and the second matrix $\bsF'_0$ is the $F$-matrix of the configuration $\scrV'_0$ in \hyperref[exa:Adjacent]{Example~\ref*{exa:Adjacent}}. Hence, these configurations are indeed the only admissibly ordered orthoplicial configurations containing $\scrF_0$ as the first four spheres. They are adjacent to each other along $\scrF_0$. One can also check that they are mapped from one to the other by the inversion along the dual sphere $S(\scrF_0)$, given explicitly as a plane $x+y=\sqrt{2}$.

For the general case, let $\scrF$ be an ordered quadruple of pairwise tangent spheres. Choose a M\"obius transformation that takes $\scrF_0$ to $\scrF$ in the order-preserving fashion. The images $\scrV,\scrV'$ of the configurations $\scrV_0,\scrV'_0$ under this transformation are the only configurations containing $\scrF$ as the first four spheres, and they are mapped from one to the other by the reflection about $S(\scrF)$ which is the image of $S(\scrF_0)$.
\end{proof}

Given a quadruple $\scrF \subset \scrV$ of pairwise tangent spheres in an orthoplicial Platonic configuration $\scrV$ of eight spheres, inverting the configuration $\scrV$ along this dual sphere $S=S(\scrF)$ yields a new orthoplicial configuration $\scrV'$ adjacent to $\scrV$ along $\scrF$. Each orthoplicial configuration contains 16 quadruples of pairwise tangent spheres, and the 16 corresponding inversions yield 16 adjacent configurations. Successively applying these inversions, we obtain an infinite family of orthoplicial configurations. We refer to the union of all spheres appearing in this family of orthoplicial configurations as an \emph{orthoplicial Apollonian packing}.

It follows from the definition that \emph{all} orthoplicial Apollonian packings in the M\"obius space is equivalent under the action of M\"obius transformations. We will distinguish orthoplicial Apollonian packings in our coordinatization $\hat \bbR^3$ of the M\"obius space. An orthoplicial Apollonian packing is said to be \emph{bounded} or \emph{ball type} if the bends of all spheres are positive except for a unique exceptional sphere whose bend is strictly negative; the exceptional sphere is the largest sphere in the packing, and it encloses all other spheres in packing. See \hyperref[fig:P1]{Figure~\ref*{fig:P1}} for an example of a bounded orthoplicial Apollonian packing.

\begin{figure}[h]
\begin{center}
\includegraphics[trim = 70px 91px 51px 62px, clip, width=35mm]{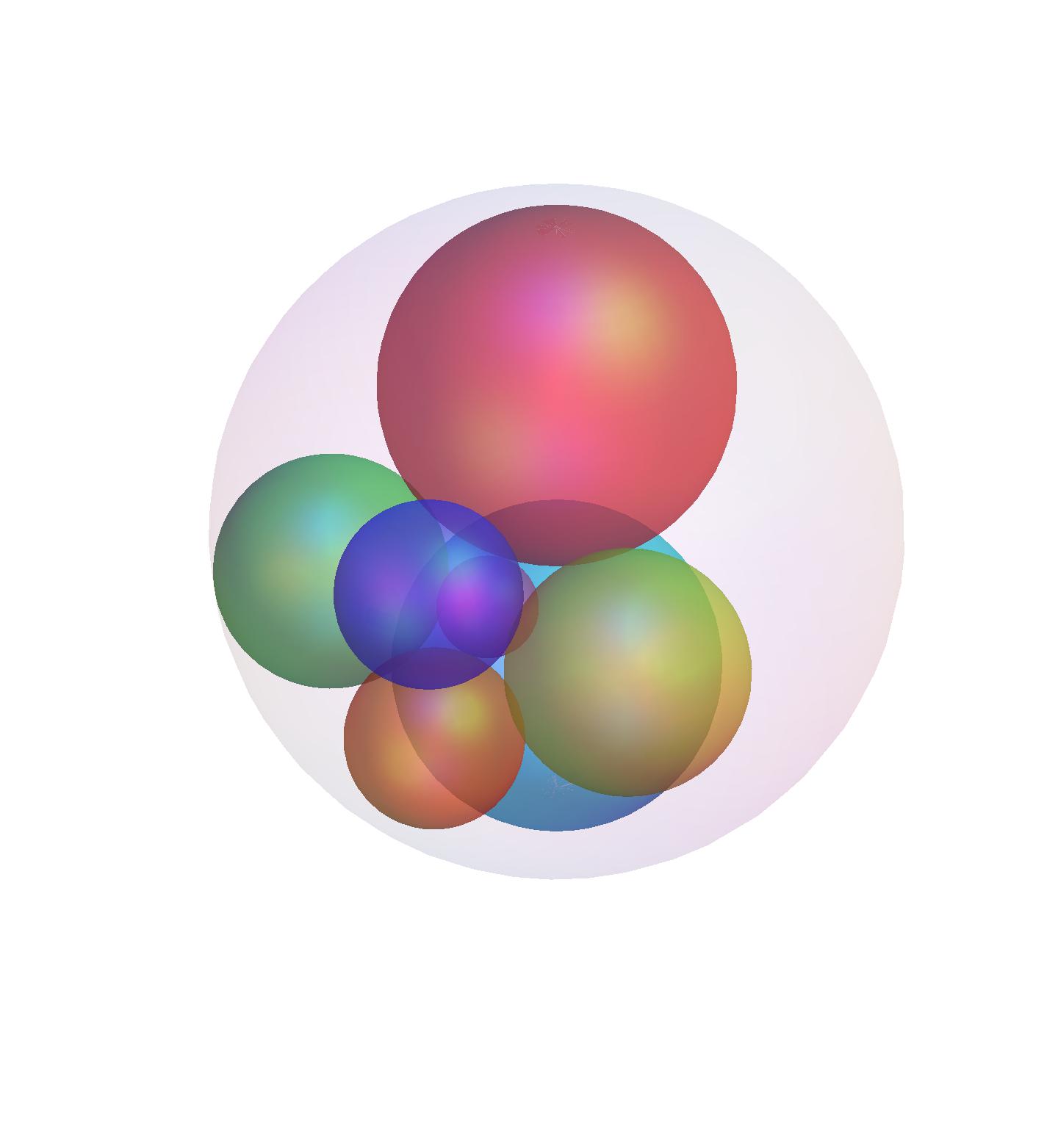}
\hspace{3mm}
\includegraphics[trim = 70px 91px 51px 62px, clip, width=35mm]{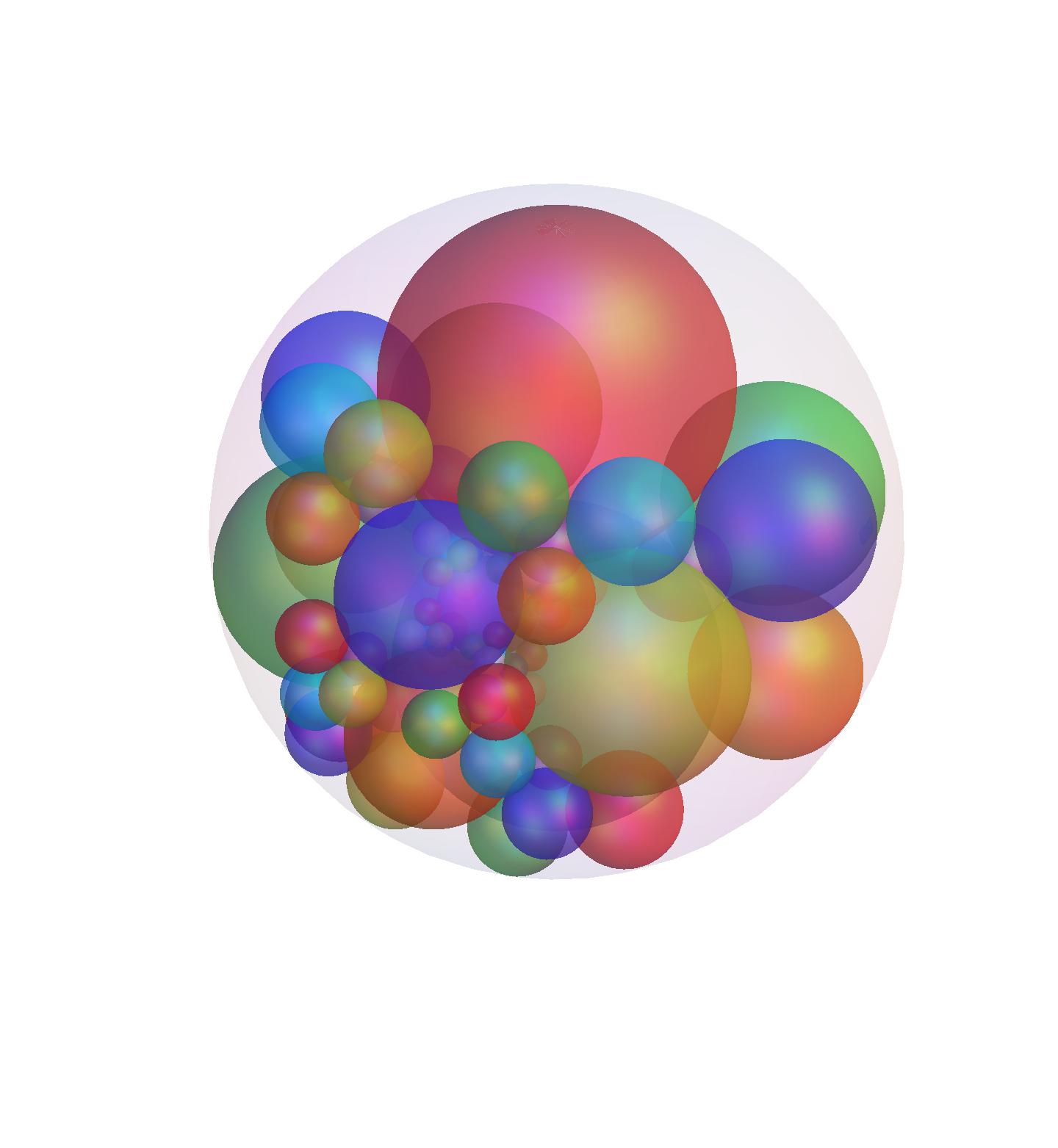}
\hspace{3mm}
\includegraphics[trim = 70px 91px 51px 62px, clip, width=35mm]{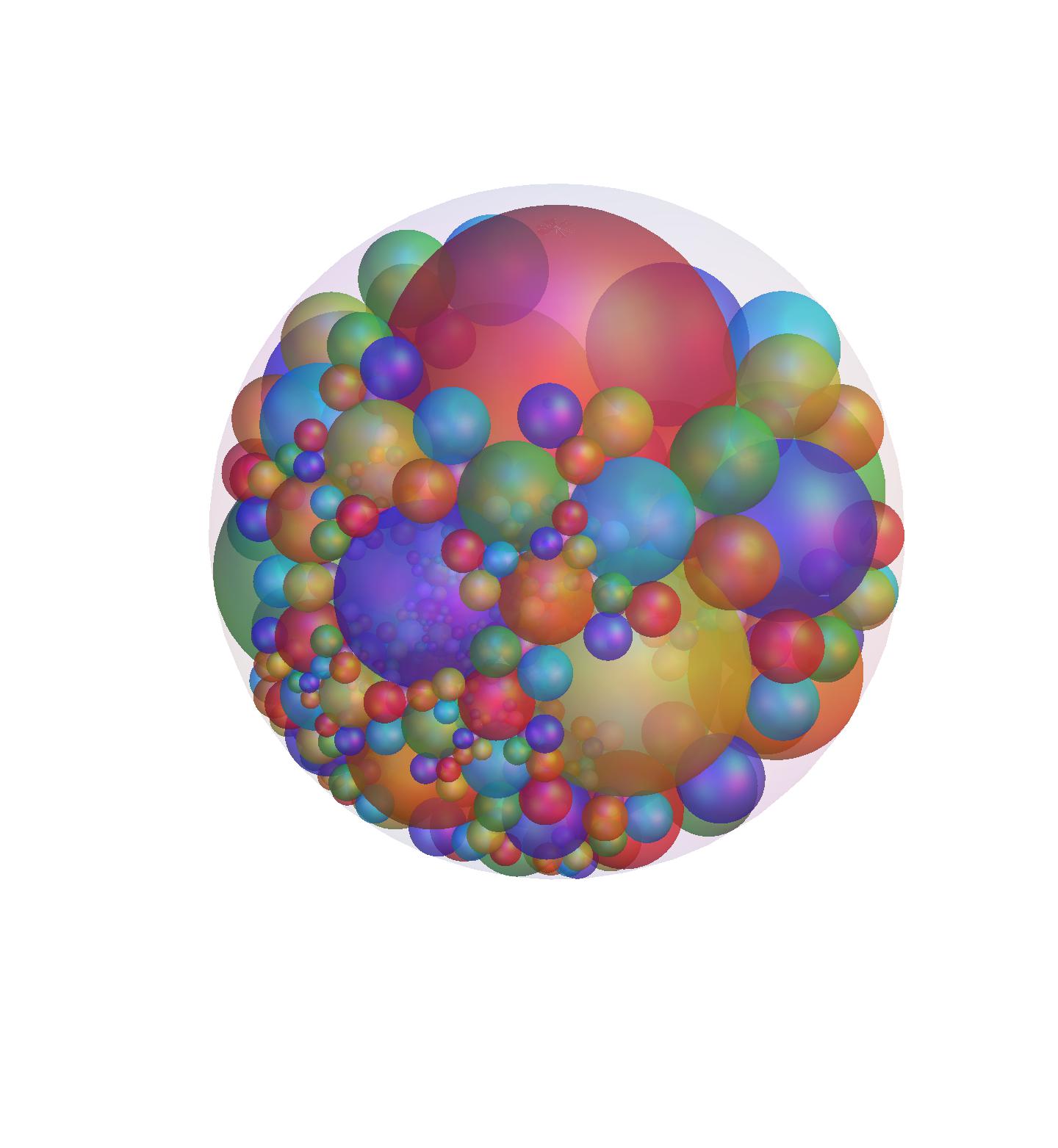}
\par \vspace{6mm}
\includegraphics[trim = 70px 91px 51px 62px, clip, width=78mm]{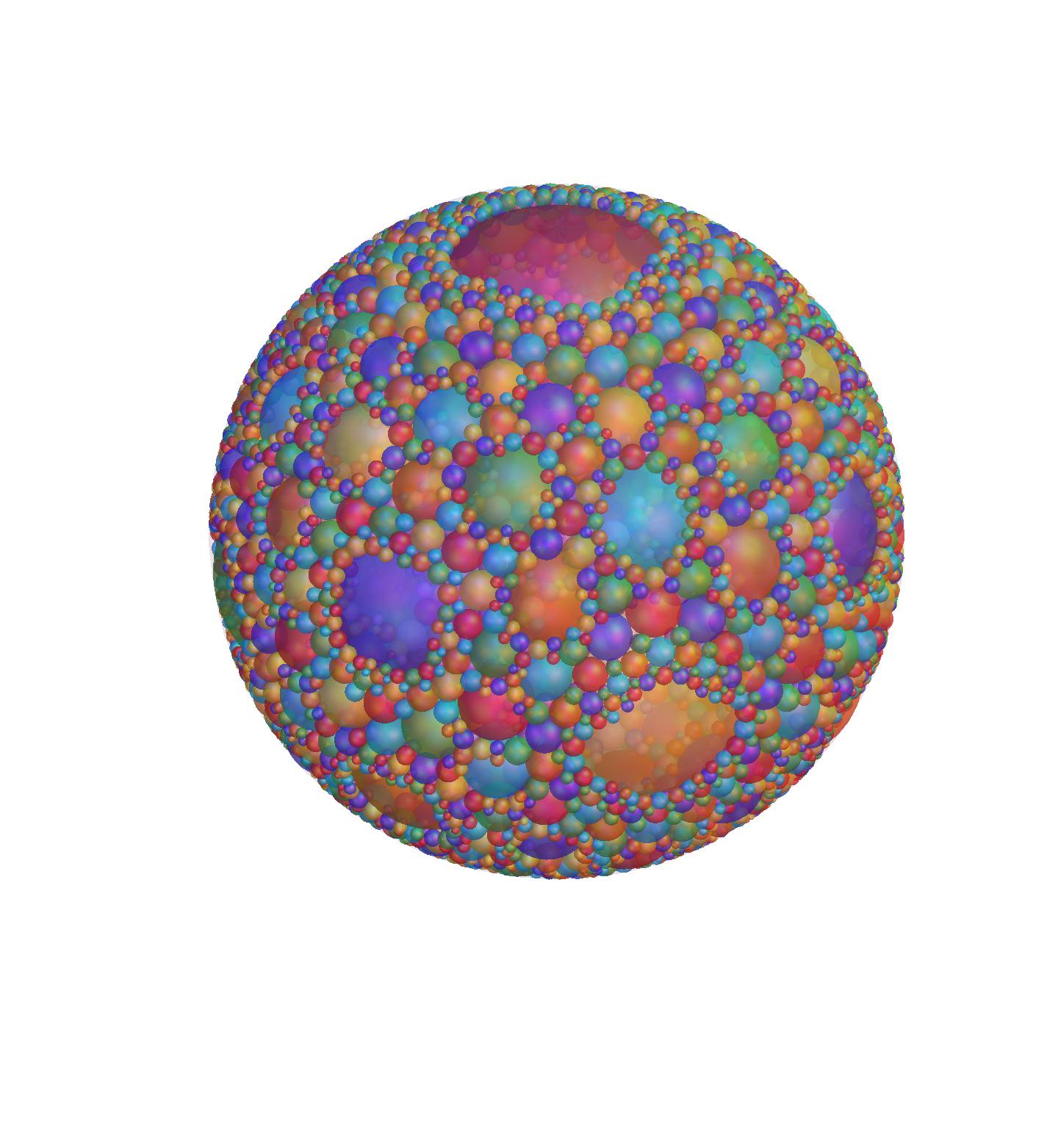}
\end{center}
\caption{The bounded orthoplicial Apollonian packing $\scrP_1$ generated form the orthoplicial Platonic configuration $\scrV_1$}
\label{fig:P1}
\end{figure}

\begin{rem}
Some of the basic properties of these packings will not be discussed in full detail in this article, and they will be treated elsewhere. In particular, we \emph{do not} establish that the orthoplicial packing is indeed a \emph{sphere packing}, in a sense that the union of spheres have disjoint orienting regions, and the spheres can only have pairwise tangency; this is true, but we decided it is better not to include the cumbersome proof of this fact in the present article. Once we establish this fact, the general classification of sphere packings applies, namely sphere packings can be classified into four types based on the number and the sign of bends of exceptional spheres in the sphere packing: (i) \emph{bounded} or \emph{ball type}, in which the bends of all spheres are positive except for a unique exceptional sphere whose bend is negative, (ii) \emph{planar} or \emph{slab type}, in which the bends of all spheres are positive except for two exceptional spheres through $\infty$ with bend zero, (iii) \emph{half-space type}, in which the bends of all spheres are positive except for a unique exceptional sphere through $\infty$ with bend zero, and (iv) \emph{full-space type}, in which the bends of all spheres are positive with no exceptional spheres. These four cases correspond to four different position of $\infty$ relative to the packing.
\end{rem}

The following lemma describes the mechanism of the inversions interchanging adjacent configurations in terms of coordinate vectors and the antipodal vectors, and will be instrumental in order to study orthoplicial Apollonian packings.

\begin{lem} \label{lem:Inversion}
Let $\scrF$ be an ordered quadruple of pairwise tangent spheres with inversive coordinate vectors $\bsv_1, \bsv_2, \bsv_3, \bsv_4$, and let $\scrV,\scrV'$ be admissibly ordered configurations that are adjacent to each other along the quadruple $\scrF$. If $\bsv_\mu, \bsv'_\mu$ are the antipodal vectors of the configurations $\scrV,\scrV'$ respectively, then we have
\begin{align} \label{eqn:Inversion}
\bsv_\mu+\bsv'_\mu=\bsv_1+\bsv_2+\bsv_3+\bsv_4.
\end{align}
\end{lem}

\begin{proof}
The equation \hyperref[eqn:Inversion]{(\ref*{eqn:Inversion})} can be verified directly for the quadruple $\scrF_0$ shared by the standard configuration and the configuration $\scrV'_0$ in \hyperref[exa:Adjacent]{Example~\ref*{exa:Adjacent}}; in this case, from the $F$-matrices $\bsF_0,\bsF'_0$ in the proof of \hyperref[lem:Adjacent]{Lemma~\ref*{lem:Adjacent}}, we have
\[
\bsv_\mu+\bsv'_\mu=(6,2,\sqrt{2},\sqrt{2},0)=\bsv_1+\bsv_2+\bsv_3+\bsv_4.
\]
For the general case, let $\bsF,\bsF'$ be the $F$-matrices of the configurations $\scrV,\scrV'$. There is a M\"obius transformation that takes $\scrF_0$ to $\scrF$, $\scrV_0$ to $\scrV$, and $\scrV'_0$ to $\scrV'$. Let $\bsM$ be a matrix representing this transformation so that $\bsF=\bsF_0 \bsM$ and $\bsF'=\bsF'_0 \bsM$. Then, by \hyperref[lem:Antipodal]{Lemma~\ref*{lem:Antipodal}}, the equation \hyperref[eqn:Inversion]{(\ref*{eqn:Inversion})} for the configurations $\scrV,\scrV'$ follows from the equation \hyperref[eqn:Inversion]{(\ref*{eqn:Inversion})} for the configurations $\scrV_0,\scrV'_0$.
\end{proof}

As an immediate corollary, we observe that the inversion interchanging adjacent configurations can be captured by multiplying by a matrix on the \emph{left} of $F$-matrices.

\begin{coro}
Let $\bsF,\bsF'$ be the $F$-matrices of admissibly ordered orthoplicial configurations $\scrV,\scrV'$ that are adjacent to each other along the quadruple $\scrF$ of pairwise tangent spheres, appearing as the first four spheres in both $\scrV,\scrV'$. Then, left multiplication by the matrix
\[
\bsR_\rmf:=\medpmatrix{
1&0&0&0&0\\
0&1&0&0&0\\
0&0&1&0&0\\
0&0&0&1&0\\
1&1&1&1&-1},
\]
interchanges the $F$-matrices $\bsF$ and $\bsF'$, i.e.\;$\bsR_\rmf \bsF=\bsF'$ and $\bsR_\rmf \bsF'=\bsF$.
\end{coro}

\begin{proof}
Left multiplication by $\bsR_\rmf$ interchanges the antipodal vector $\bsv_\mu$ of $\scrV$ and $\bsv'_\mu$ of $\scrV'$ by \hyperref[lem:Inversion]{Lemma~\ref*{lem:Inversion}}, while fixing the first four common rows $\bsv_1,\bsv_2,\bsv_3,\bsv_4$ representing the spheres in $\scrF$.
\end{proof}

\subsection{Apollonian Group}

We now define the orthoplicial analogue of the $4 \times 4$ matrix group introduced in \cite[\S4]{GM} as the octahedral analogue of the classical \emph{tetrahedral} Apollonian group from \cite{Hirst}. Given an orthplicial Platonic configuration, there are 16 quadruples $\scrF_{ijk\ell}=\{S_i,S_j,S_k,S_\ell\}$ of pairwise tangent spheres, where $i\equiv1$, $j\equiv2$, $k\equiv3$, $\ell\equiv4$ modulo 4. These 16 quadruples define 16 inversions that yield 16 adjacent Platonic configurations. Each inversion about the sphere $S_{ijk\ell}=S(\scrF_{ijk\ell})$ dual to the quadruple $\scrF_{ijk\ell}$ is given by a conjugate of $\bsR_\rmf=\bsS_{1234}$ by a suitable element of the Platonic group $\Platonic$. We consider the group generated by these matrices.

\begin{defn}
The \emph{orthoplicial Apollonian group} $\Apollonian$ is defined to be the $5 \times 5$ matrix group generated by $\calS:=\{\bsS_{ijk\ell}\}$, consisting of the following 16 matrices:
\[
\bsS_{1234}=\medpmatrix{
1&0&0&0&0\\
0&1&0&0&0\\
0&0&1&0&0\\
0&0&0&1&0\\
1&1&1&1&-1},
\]\[
\bsS_{5234}=\medpmatrix{
-1&2&2&2&0\\
0&1&0&0&0\\
0&0&1&0&0\\
0&0&0&1&0\\
-1&1&1&1&1},
\quad
\bsS_{1634}=\medpmatrix{
1&0&0&0&0\\
2&-1&2&2&0\\
0&0&1&0&0\\
0&0&0&1&0\\
1&-1&1&1&1},
\]\[
\bsS_{1274}=\medpmatrix{
1&0&0&0&0\\
0&1&0&0&0\\
2&2&-1&2&0\\
0&0&0&1&0\\
1&1&-1&1&1},
\quad
\bsS_{1238}=\medpmatrix{
1&0&0&0&0\\
0&1&0&0&0\\
0&0&1&0&0\\
2&2&2&-1&0\\
1&1&1&-1&1},
\]\[
\bsS_{5634}=\medpmatrix{
-1&-2&2&2&4\\
-2&-1&2&2&4\\
0&0&1&0&0\\
0&0&0&1&0\\
-1&-1&1&1&3},
\quad
\bsS_{5274}=\medpmatrix{
-1&2&-2&2&4\\
0&1&0&0&0\\
-2&2&-1&2&4\\
0&0&0&1&0\\
-1&1&-1&1&3},
\quad
\bsS_{5238}=\medpmatrix{
-1&2&2&-2&4\\
0&1&0&0&0\\
0&0&1&0&0\\
-2&2&2&-1&4\\
-1&1&1&-1&3},
\]\[
\bsS_{1674}=\medpmatrix{
1&0&0&0&0\\
2&-1&-2&2&4\\
2&-2&-1&2&4\\
0&0&0&1&0\\
1&-1&-1&1&3},
\quad
\bsS_{1638}=\medpmatrix{
1&0&0&0&0\\
2&-1&2&-2&4\\
0&0&1&0&0\\
2&-2&2&-1&4\\
1&-1&1&-1&3},
\quad
\bsS_{1278}=\medpmatrix{
1&0&0&0&0\\
0&1&0&0&0\\
2&2&-1&-2&4\\
2&2&-2&-1&4\\
1&1&-1&-1&3},
\]\[
\bsS_{5674}=\medpmatrix{
-1&-2&-2&2&8\\
-2&-1&-2&2&8\\
-2&-2&-1&2&8\\
0&0&0&1&0\\
-1&-1&-1&1&5},
\quad
\bsS_{5638}=\medpmatrix{
-1&-2&2&-2&8\\
-2&-1&2&-2&8\\
0&0&1&0&0\\
-2&-2&2&-1&8\\
-1&-1&1&-1&5},
\]\[
\bsS_{5278}=\medpmatrix{
-1&2&-2&-2&8\\
0&1&0&0&0\\
-2&2&-1&-2&8\\
-2&2&-2&-1&8\\
-1&1&-1&-1&5},
\quad
\bsS_{1678}=\medpmatrix{
1&0&0&0&0\\
2&-1&-2&-2&8\\
2&-2&-1&-2&8\\
2&-2&-2&-1&8\\
1&-1&-1&-1&5},
\]\[
\bsS_{5678}=\medpmatrix{
-1&-2&-2&-2&12\\
-2&-1&-2&-2&12\\
-2&-2&-1&-2&12\\
-2&-2&-2&-1&12\\
-1&-1&-1&-1&7}.
\]
\end{defn}

We have $\bsS_{ijk\ell}^2=\bsI$ for all of the generators above, since they are conjugates of $\bsR_\rmf=\bsS_{1234}$. We also have relations of the form
\begin{align*}
(\bsS_{ijk\ell}\bsS_{i'\!jk\ell})^2=
(\bsS_{ijk\ell}\bsS_{ij'\!k\ell})^2=
(\bsS_{ijk\ell}\bsS_{ijk'\!\ell})^2=
(\bsS_{ijk\ell}\bsS_{ijk\ell'\!})^2=I.
\end{align*}
Namely, for any pair of generators sharing three out of four labels, their product have order 2; there are 32 such pairs, and we can check these relations directly. The 16 generators and $16+32=48$ relations above appear to give a complete presentation for this group abstractly, but we will not verify this fact here.

\begin{lem}
The orthoplicial Apollonian group $\Apollonian$ is a subgroup of $\OF(\bbZ)$; namely, for any matrix $\bsA \in \Apollonian$, we have $\det \bsA=\pm1$ and
\[
\bsA^\trans \bsQF \bsA=\bsQF.
\]
\end{lem}

\begin{proof}
For each generator $\bsS=\bsS_{ijk\ell} \in \calS$, $\det \bsS=-1$ and $\bsS^\trans \bsQF \bsS=\bsQF$ can be checked by direct computation.
\end{proof}

\begin{defn}
The \emph{oriented orthoplicial Apollonian group} $\Apollonian^+<\Apollonian$ is the subgroup consisting of matrices with determinant $+1$, i.e.\;$\Apollonian^+:=\Apollonian \cap \SOF(\bbZ)$.
\end{defn}

The oriented Apollonian group $\Apollonian^+$ is the index 2 kernel of the determinant on the Apollonian group $\Apollonian$. Since every generator $\bsS_{ijk\ell} \in \calS$ has determinant $-1$, it follows that $\Apollonian^+$ consists of elements that can be, and can only be, written as even-length words in the generators $\calS=\{\bsS_{ijk\ell}\}$ of $\Apollonian$; hence, $\Apollonian^+$ is generated by $\{\bsS_{ijk\ell}\bsS_{i'\!j'\!k'\!\ell'\!} \mid \bsS_{ijk\ell}, \bsS_{i'\!j'\!k'\!\ell'\!} \in \calS\}$, which can easily be reduced to
\begin{align} \label{eqn:S+}
\calS^+:=\{\bsS_{1234}\bsS_{ijk\ell} \mid \bsS_{1234} \neq \bsS_{ijk\ell} \in \calS\}.
\end{align}
using the relations $\bsS_{ijk\ell}^2=\bsI$ for all $\bsS_{ijk\ell} \in \calS$.

For our purposes, the most important features of the Apollonian group $\Apollonian$ and the oriented Apollonian group $\Apollonian^+$ are the actions on $F$-matrices, summarized below.

\begin{lem} \label{lem:AonF}
If $\scrP$ is an orthoplicial Apollonian packing containing orthoplicial Platonic configurations $\scrV, \scrV'$, with their $F$-matrices $\bsF,\bsF'$, then, $\bsF' \in \Apollonian \bsF$.
\end{lem}

\begin{proof}
At each step in the construction of an orthoplicial Apollonian packing $\scrP$, the $F$-matrix of a Platonic configuration is linearly transformed to the $F$-matrix of the adjacent Platonic configuration by the generators $\bsS_{ijk\ell}$ of $\Apollonian$ corresponding to the sphere inversions. Hence, it follows that the orbit $\Apollonian \bsF$ of the $F$-matrix $\bsF$ of the initial Platonic configuration consists of $F$-matrices $\bsF'$ of \emph{all} Platonic configurations in $\scrP$ with respect to the induced admissible ordering.
\end{proof}

It follows that, if $S$ is a sphere in an Apollonian packing $\scrP$ generated from the initial configuration $\scrV$ with its $F$-matrix $\bsF$, there exists a Platonic configuration $\scrV'$ with its $F$-matrix $\bsF' \in \Apollonian \bsF$ such that the inversive coordinate vector $\bsv(S)$ of the given sphere $S$ is captured by $\bsF'$, explicitly as one of the row vectors $\bsv'_k$ or implicitly as one of the complimentary vectors $2\bsv'_\mu-\bsv'_k$, $k=1,2,3,4$. An important observation here is that it suffices to consider the orbit of $\bsF$ under \emph{oriented} Apollonian group $\Apollonian^+$ to capture the vector $\bsv(S)$.

\begin{lem} \label{lem:A+onF}
If $\scrP$ is an orthoplicial Apollonian packing containing an orthoplicial Platonic configuration $\scrV$ with the $F$-matrix $\bsF$ and $S$ is a sphere in $\scrP$, then there exists an $F$-matrix $\bsF' \in \Apollonian^+ \bsF$ such that the inversive coordinate vector $\bsv(S)$ of the given spehre $S$ is captured by $\bsF'$, explicitly as one of the row vectors $\bsv'_k$ or implicitly as one of the complimentary vectors $2\bsv'_\mu-\bsv'_k$, $k=1,2,3,4$.
\end{lem}

\begin{proof}
Let $\scrV'' \subset \scrP$ be a Platonic configuration containing the given sphere $S \in \scrP$. If $\scrV''$ is in the $\Apollonian^+$-orbit of $\scrV$, we set $\scrV':=\scrV''$. If $\scrV''$ is \emph{not} in the $\Apollonian^+$-orbit of $\scrV$, take a configuration $\scrV' \ni S$ adjacent to $\scrV''$ along a pairwise tangent quadruple $\scrF$ containing $S$. Note that it takes an odd number of inversions to map $\scrV$ onto $\scrV''$, and hence post-composing these inversions with one more inversion along $\scrF$ maps $\scrV$ onto $\scrV' \ni S$ with even number of inversions; hence $\scrV'$ is in the $\Apollonian^+$-orbit of $\scrV$. In any case, it now follows that, for any given sphere $S$, we can find a Platonic configuration $\scrV' \ni S$ in the $\Apollonian^+$-orbit of $\scrV$. Writing $\bsF'$ for the $F$-matrix of $\scrV'$, we have $\bsF' \in \Apollonian^+ \bsF$ and the inversive coordinate vector $\bsv(S)$ of the given spehre $S \in \scrV'$ is captured by $\bsF'$, explicitly as one of the row vectors $\bsv'_k$ or implicitly as one of the complimentary vectors $2\bsv'_\mu-\bsv'_k$, $k=1,2,3,4$.
\end{proof}

\subsection{Integral Packings} \label{ssec:IntegralPackings}

An orthoplicial Apollonian packing $\scrP$ is said to be \emph{integral} if the bends of all constituent spheres are integers; it must be planar (slab type) or bounded (ball type) since the unoriented radius of any non-planar sphere $S \in \scrP$ is bounded above by 1. We write $\scrB(\scrP)$ for the set $\{b(S) \mid S \in \scrP\} \subset \bbZ$ of integral bends, and write $\scrB^+(\scrP):=\scrB(\scrP) \cap \bbN \subset \bbN$. An integral Apollonian packing $\scrP$ is said to be \emph{primitive} if $\gcd \scrB(\scrP)=1$. Rescaling any integral packing $\scrP$ by the factor $\gcd \scrB(\scrP)$ always yields a primitive packing; hence, questions on the bends in integral packings reduces to questions on the bends in primitive packings.

\begin{lem} \label{lem:IntegralApollonian}
Let $\scrP$ be an orthoplicial Apollonian packing containing an orthoplicial Platonic configuration $\scrV$ with its bend vector $\bsb=\bsb(\scrV)=(b_1, b_2, b_3, b_4, b_\mu)$. Then, $\scrP$ is integral if and only if $\bsb$ is integral; moreover, $\scrP$ is primitive if and only if $\bsb$ is primitive.
\end{lem}

\begin{proof}
Since the generators $\bsS_{ijk\ell}$ of the orthoplicial Apollonian group $\Apollonian$ are integer matrices, the entire group $\Apollonian$ consists only of integer matrices. It follows that, in an orthoplicial Apollonian packing $\scrP$, (i) the bend vectors of all Platonic configurations in $\scrP$ is integral if and only if the bend vector of one Platonic configuration in $\scrP$ is integral, and (ii) the bend vectors of all Platonic configurations in $\scrP$ is primitive if and only if the bend vector of one Platonic configuration in $\scrP$ is primitive. Hence, Proposition reduces to the following statements about Platonic configuraitons: (i$'$) the bend vector $\bsb=\bsb(\scrV)$ is integral if and only if $\scrV$ is integral, and (ii$'$) the bend vector $\bsb=\bsb(\scrV)$ is primitive if and only if $\scrV$ is primitive. These statements are already established as \hyperref[prop:IntegralPlatonic]{Proposition~\ref*{prop:IntegralPlatonic}}.
\end{proof}

\begin{figure}[h]
\begin{center}
\includegraphics[trim = 70px 91px 51px 62px, clip, height=55mm, width=55mm]{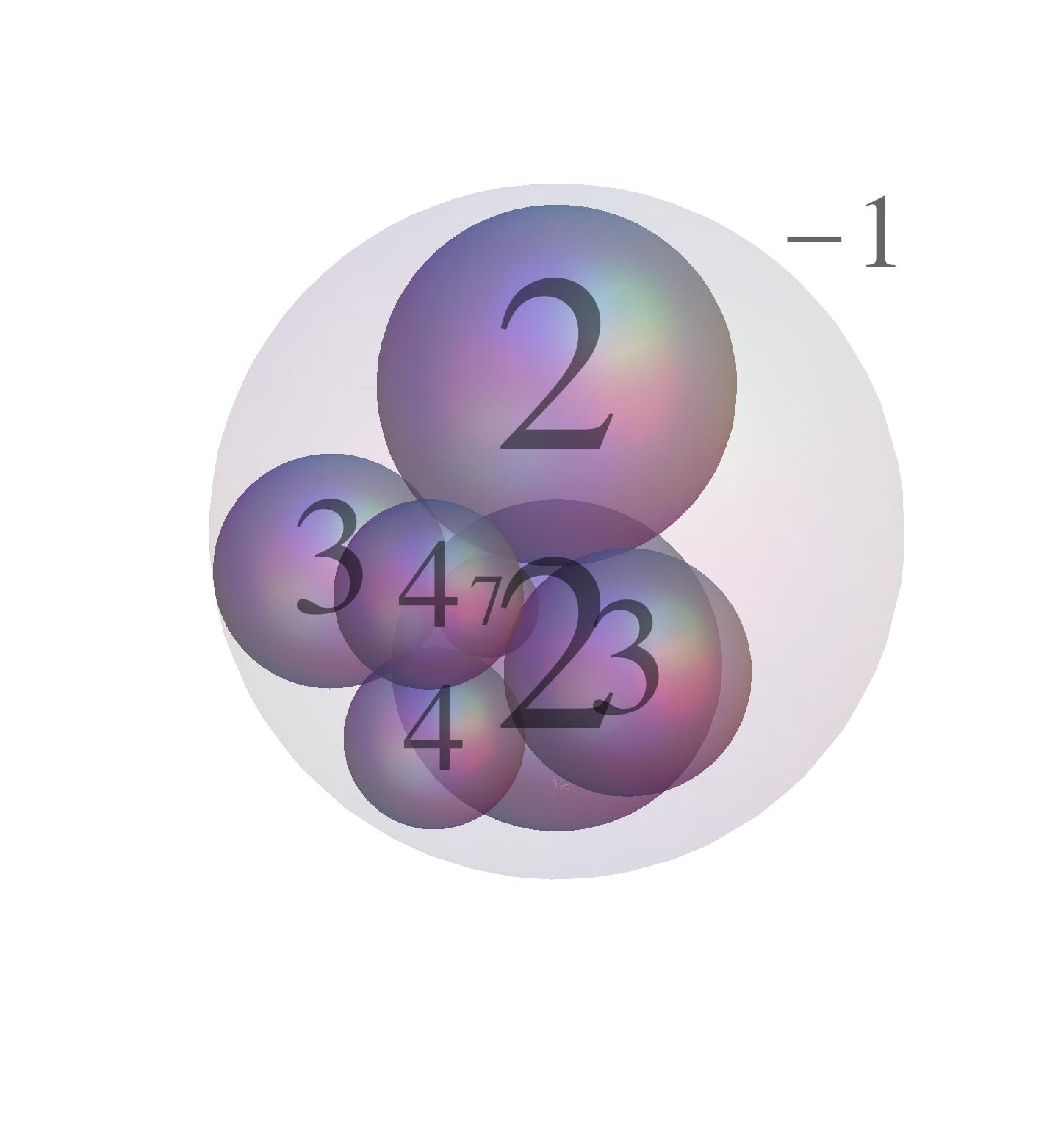}
\hspace{6mm}
\includegraphics[trim = 70px 91px 51px 62px, clip, height=55mm, width=55mm]{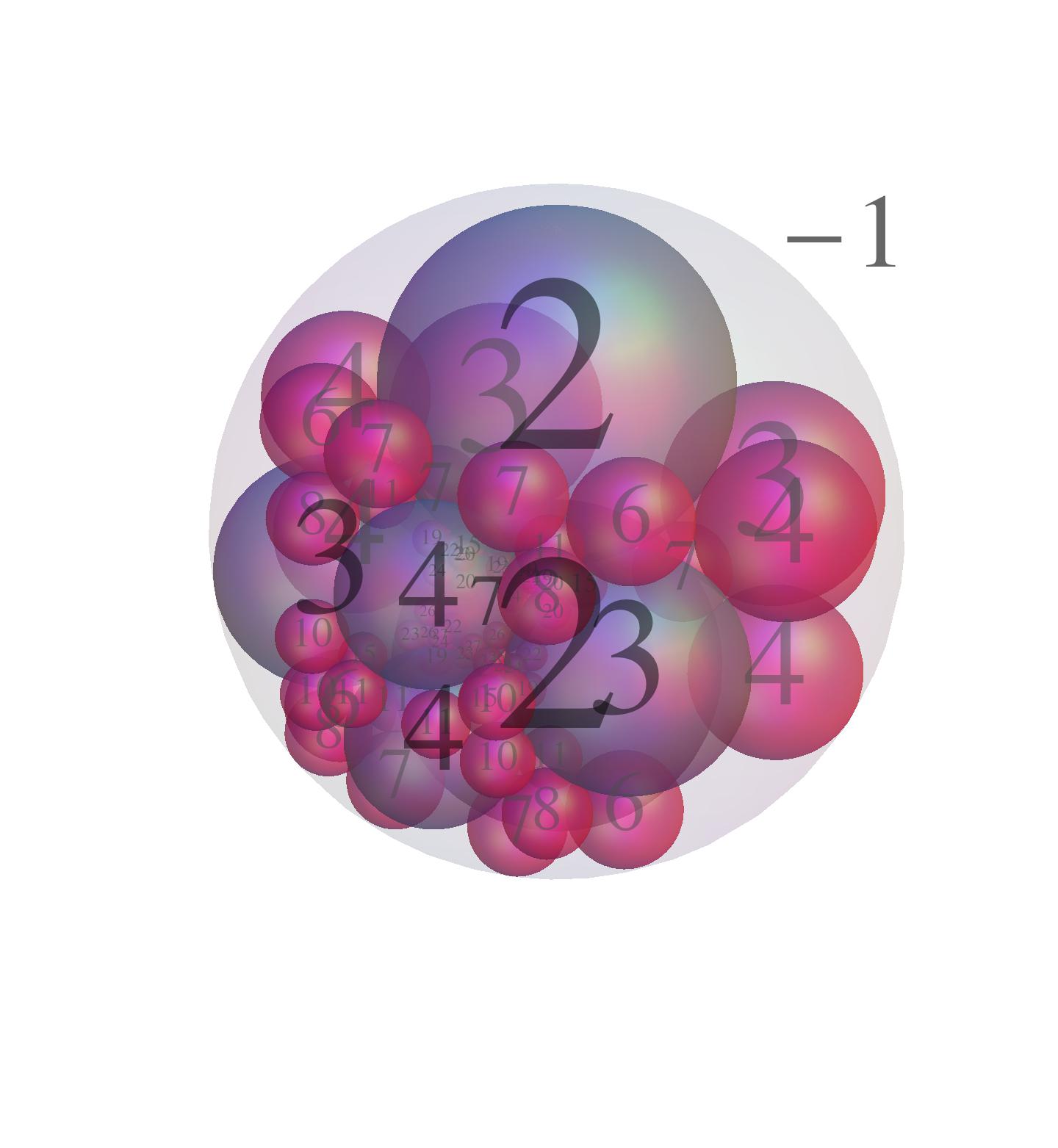}
\end{center}
\caption{The first step in the construction of the primitive Apollonian packing $\scrP_1$ from the Platonic configuration $\scrV_1$}
\label{fig:P1Step}
\end{figure}

We have already seen a few examples of primitive orthoplicial Platonic configurations, i.e.\;$\scrV_0$ in \hyperref[exa:V0]{Example~\ref*{exa:V0}}, $\scrV_1$ in \hyperref[exa:V1]{Example~\ref*{exa:V1}}, and $\scrV_{7\rmd}$ in \hyperref[exa:V7d]{Example~\ref*{exa:V7d}}. By \hyperref[lem:IntegralApollonian]{Lemma~\ref*{lem:IntegralApollonian}}, the Apollonian packings generated from these configurations are primitive. The primitive Apollonian packing $\scrP_1$ generated from the Platonic configuration $\scrV_1$ is shown in \hyperref[fig:P1]{Figure~\ref*{fig:P1}} and \hyperref[fig:P1Step]{Figure~\ref*{fig:P1Step}}. The primitive Apollonian packing $\scrP_{7\rmd}$ generated from the Platonic configuration $\scrV_{7\rmd}$ is shown in \hyperref[fig:IntegralSOASP]{Figure~\ref*{fig:IntegralSOASP}}.

%%%
\section{Sphere Stabilizer and Integral Bends} \label{sec:LG}

\subsection{Local Obstructions} \label{ssec:LO}

We are interested in understanding the set $\scrB(\scrP)$ of bends in primitive Apollonian packings $\scrP$. We start our investigation with computational observations based on a few primitive orthoplicial Apollonian packings. Let us continue to write $\scrP_0, \scrP_1, \scrP_{7\rmd}$ for primitive Apollonian packings generated from the configurations $\scrV_0,\scrV_1,\scrV_{7\rmd}$ respectively. Computation shows that the set of bends of spheres in $\scrP_0,\scrP_1,\scrP_{7\rmd}$ are as follows:
\begin{align*}
\scrB(\scrP_0)&=
\left\{\begin{matrix}
0, 1, 2, 4, 5, 6, 8, 9, 10, 12, 13, 14, 16, 17, 18, 20, 21, 22, 24,\\
25, 26, 28, 29, 30, 32, 33, 34, 36, 37, 38, 40, 41, 42, 44, 45, 46, \\
48, 49, 50, 52, 53, 54, 56, 57, 58, 60, 61, 62, 64, 65, 66, 68, \cdots
\end{matrix}\right\}\\
\scrB(\scrP_1)&=
\left\{\begin{matrix}
-1, 2, 3, 4, 6, 7, 8, 10, 11, 12, 14, 15, 16, 18, 19, 20, 22, 23, 24, \\
26, 27, 28, 30, 31, 32, 34, 35, 36, 38, 39, 40, 42, 43, 44, 46, 47, \\
48, 50, 51, 52, 54, 55, 56, 58, 59, 60, 62, 63, 64, 66, 67, 68, \cdots
\end{matrix}\right\}\\
\scrB(\scrP_{7\rmd})&=
\left\{\begin{matrix}
-7, 12, 17, 20, 22, 24, 25, 29, 30, 33, 34, 37, 38, 40, 41, 44, 46, \\
48, 49, 50, 52, 53, 54, 56, 58, 60, 61, 62, 64, 65, 66, 68, 69, \cdots, \\
200, 201, 202, 204, 205, 206, 208, 209, 210, 212, 213, 214, 216, \\
217, 218, 220, 221, 222, 224, 225, 226, 228, 229, 230, 232, 233, \\
234, 236, 237, 238, 240, 241, 242, 244, 245, 246, 248, 249, \cdots\;\;
\end{matrix}\right\}
\end{align*}
By inspecting these numbers, we observe that $\scrB(\scrP_0)$ seems to contain all positive integers $n \equiv 0,1,2 \pmod{4}$ but no integers $n \equiv 3 \pmod{4}$; similarly, we observe that $\scrB(\scrP_1)$ seems to contain all positive integers $n \equiv 0,2,3 \pmod{4}$ but no integers $n \equiv 1 \pmod{4}$. As for $\scrB(\scrP_{7\rmd})$, it appears that $\scrB(\scrP_{7\rmd})$ seems to contain all large enough integers $n \equiv 0,1,2 \pmod{4}$ but no integers $n \equiv 3 \pmod{4}$. For any other primitive orthoplicial Apollonian packing $\scrP$, the same phenomena can be observed computationally; these evidences suggest the following statements:
\begin{itemize}
\item[] \hspace{-9.3mm} For any primitive orthoplicial Apollonian packing $\scrP$,
\item[(a)] there is a local obstruction modulo 4,
\item[(b)] this obstruction modulo 4 is the only local obstruction, and
\item[(c)] every large enough integer avoiding the local obstruction appears in $\scrB(\scrP)$.
\end{itemize}
We verify the statement (a) in the following proposition; the remaining statements (b) and (c) will be treated in the subsequent subsections.

\begin{prop} \label{prop:LO}
For the set of bends $\scrB=\scrB(\scrP)$ of a primitive orthoplicial Apollonian packing $\scrP$, there is always a local obstruction modulo 4: there exists $\varepsilon=\varepsilon(\scrP) \in \{\pm 1\}$ such that, for every $b \in \scrB$,
\[
b \not \equiv -\varepsilon \mod{4}.
\]
Moreover, for any orthoplicial Platonic configuration $\scrV$ in $\scrP$, the ordered set of bends $b_k=b(S_k)$ of constituent spheres of $\scrV$ satisfy
\[
(b_1,b_2,b_3,b_4,b_5,b_6,b_7,b_8) \equiv (0,0,\varepsilon,\varepsilon,2,2,\varepsilon,\varepsilon) \mod{4}
\]
up to reordering by the permuting action of the orthoplicial Platonic group $\Platonic$.
\end{prop}

\begin{proof}
We first prove the presence of an analogous local obstruction for the bends of spheres in a primitive Platonic configuration $\scrV$. Consider the cone
\[
F(\bsb)= 2b_\mu^2-2b_\mu(b_1+b_2+b_3+b_4)+(b_1^2+b_2^2+b_3^2+b_4^2)=0
\]
defined by one of the equations \hyperref[eqn:DGM-diag]{(\ref*{eqn:DGM-diag})} in \hyperref[thm:DGM]{Theorem~\ref*{thm:DGM}}. This equation is degenerate over $\bbZfour$, and solutions over $\bbZfour$ include ones that are not reductions modulo 4 of solutions over $\bbZ$. Solving the equation over $\bbZeight$ instead, we find 3584 solutions for $\bsb=(b_1,b_2,b_3,b_4,b_\mu)$; together with the equation \hyperref[eqn:Antipodal]{(\ref*{eqn:Antipodal})} from \hyperref[coro:Antipodal]{Corollary~\ref*{coro:Antipodal}}, we only have 1794 solutions for $(b_1,b_2,b_3,b_4,b_5,b_6,b_7,b_8)$, since there are two choices of $b_\mu$ for each $2b_\mu$ in $\bbZeight$. Since we are only interested in primitive solution, we can then remove all \emph{even} solutions, including the origin; this leaves 1536 solutions.

These solutions are highly redundant, since we have not utilized the action of the Platonic group $\Platonic$; we need to choose one representative from each orbit of solutions under the action of $\Platonic$. Note that $\Platonic$ is the full symmetry group of 4-orthoplex, isomorphic to the signed-permutation group on $8=4 \times 2$ points. Consider the ordering $0 \prec 1 \prec 2 \prec \cdots \prec 6 \prec 7$ for elements of $\bbZeight$. Note that, for $k=1,2,3,4$, interchanging $b_k$ and $b_{k+4}$ by a suitable conjugate of $\bsR_4 \in \Platonic$ yields another solution; removing solutions with $b_k \succ b_{k+4}$ for some $k=1,2,3,4$, we are left with 240 solutions such that $b_k \preccurlyeq b_{k+4}$ for all $k=1,2,3,4$. Also, for $k=1,2,3$, interchanging $b_k$ and $b_{k+1}$ as well as $b_{k+4}$ and $b_{k+5}$ by $\bsR_k \in \Platonic$ yields another solution; removing solutions with $b_k \succ b_{k+1}$ for some $k=1,2,3$, we are now left only with the following 24 solutions such that $b_k \preccurlyeq b_{k+4}$ for all $k=1,2,3,4$ and $b_1 \preccurlyeq b_2 \preccurlyeq b_3 \preccurlyeq b_4$:
{\small
\[
\begin{matrix}
(0, 0, 1, 1, 2, 2, 1, 1), & (0, 0, 1, 1, 6, 6, 5, 5), & (0, 0, 1, 5, 2, 2, 1, 5), & (0, 0, 3, 3, 2, 2, 7, 7), \\
(0, 0, 3, 3, 6, 6, 3, 3), & (0, 0, 3, 7, 6, 6, 3, 7), & (0, 0, 5, 5, 2, 2, 5, 5), & (0, 0, 7, 7, 6, 6, 7, 7), \\
(0, 1, 1, 2, 6, 5, 5, 4), & (0, 1, 1, 4, 2, 1, 1, 6), & (0, 1, 4, 5, 2, 1, 6, 5), & (0, 2, 3, 3, 6, 4, 3, 3), \\
(0, 2, 3, 7, 6, 4, 3, 7), & (0, 2, 7, 7, 6, 4, 7, 7), & (0, 3, 3, 4, 2, 7, 7, 6), & (0, 4, 5, 5, 2, 6, 5, 5), \\
(1, 1, 2, 2, 5, 5, 4, 4), & (1, 1, 4, 4, 1, 1, 6, 6), & (1, 4, 4, 5, 1, 6, 6, 5), & (2, 2, 3, 3, 4, 4, 3, 3), \\
(2, 2, 3, 7, 4, 4, 3, 7), & (2, 2, 7, 7, 4, 4, 7, 7), & (3, 3, 4, 4, 7, 7, 6, 6), & (4, 4, 5, 5, 6, 6, 5, 5).
\end{matrix}
\]
}
Finally, reducing these solutions mod 4 and removing the redundancy once again by the action of the signed-permutation group $\Platonic$ using the ordering $0 \prec 1 \prec 2 \prec 3$ for elements of $\bbZfour$, as we have done so with $\bbZeight$, we obtain just two solutions
\[
(0,0,1,1,2,2,1,1), \qquad (0,0,3,3,2,2,3,3).
\]
In other words, for a Platonic configuration $\scrV$, there exists $\varepsilon=\varepsilon(\scrV) \in \{\pm1\}$ such that the bends $b_k=b(S_k)$ of the constituent spheres $S_k \in \scrV$ satisfy
\begin{align} \label{eqn:LO}
\begin{matrix}
b_k \not \equiv -\varepsilon \mod{4} \quad \text{for all $k=1,\cdots,8$, and}\\
(b_1,b_2,b_3,b_4,b_5,b_6,b_7,b_8) \equiv (0,0,\varepsilon,\varepsilon,2,2,\varepsilon,\varepsilon) \mod{4},
\end{matrix}
\end{align}
up to reordering by the permuting action of the orthoplicial Platonic group $\Platonic$.

To complete the proof, we only need to show that the local obstruction \hyperref[eqn:LO]{(\ref*{eqn:LO})} for the bends of spheres in a Platonic configuration persists under the action of the Apollonian group $\Apollonian$, so that $\varepsilon(\scrP):=\varepsilon(\scrV)$ is well-defined. We can verify this by direct computation for each generator $\bsS_{ijk\ell} \in \calS$ of $\Apollonian$, representing the inversion along any quadruple of pairwise tangent spheres; namely, if $\scrV$ and $\scrV'$ are adjacent Platonic configurations in $\scrP$, then $\varepsilon(\scrV)=\varepsilon(\scrV')$.
\end{proof}

\subsection{Sphere Stabilizer}

We now study the set $\scrB(\scrP)$ of bends in primitive orthoplicial Apollonian packings $\scrP$ by looking at the orbit of the initial bend vector $\bsb$ under the action of $\Apollonian$ and $\Apollonian^+$. Let us first restate \hyperref[lem:AonF]{Lemma~\ref*{lem:AonF}} and \hyperref[lem:A+onF]{Lemma~\ref*{lem:A+onF}} in terms of bend vectors in the following corollaries.

\begin{coro} \label{lem:Aonb}
If $\scrP$ is an orthoplicial Apollonian packing containing orthoplicial Platonic configurations $\scrV, \scrV'$, with their bend vectors $\bsb,\bsb'$, then, $\bsb' \in \Apollonian \bsb$.
\end{coro}

\begin{coro} \label{lem:A+onb}
If $\scrP$ is an orthoplicial Apollonian packing containing an orthoplicial Platonic configuration $\scrV$ with the bend vector $\bsb$ and $S$ is a sphere in $\scrP$, then there exists a bend vector $\bsb' \in \Apollonian^+ \bsb$ such that the bend $b(S)$ of the given sphere $S$ is captured by $\bsb'$, explicitly as one of the bend components $b'_k$ or implicitly as one of the complimentary bends $2b'_\mu-b'_k$, $k=1,2,3,4$.
\end{coro}

The main source of difficulty in studying the orbits of a vector $\bsb$ under the action of $\Apollonian$ and $\Apollonian^+$ is that they are \emph{thin} groups in $\SL_5(\bbZ)$; many of the powerful techniques in the theory of algebraic groups do not apply directly. We will circumvent this issue by focusing on the bends of spheres tangent to a chosen sphere in a primitive orthoplicial Apollonian sphere packing, following the approach taken for the classical \emph{tetrahedral} Apollonian circle packings in \cite{S:Letter}, \cite{BF:Positive}, \cite{BK:Strong}, \cite{B:Prime}, for the \emph{simplicial} Apollonian sphere packings in \cite{K:Soddy}, and for the \emph{octahedral} Apollonian circle packings in \cite{Z:Octahedral}.

\begin{defn}
The \emph{$S_1$-stabilizer} $\Apollonian_1<\Apollonian<\OF(\bbZ)$ is defined to be the subgroup generated by the 8 matrices in $\calS_1:=\{\bsS_{1jk\ell} \mid \bsS_{1jk\ell} \in \calS\}$. The \emph{oriented $S_1$-stabilizer} $\Apollonian^+_1<\Apollonian_1$ is the subgroup consisting of matrices with determinant 1, i.e.\;$\Apollonian^+_1:=\Apollonian_1 \cap \Apollonian^+<\SOF(\bbZ)$.
\end{defn}

The $S_1$-stabilizer $\Apollonian_1$ fixes the first row of $F$-matrices, representing the first sphere in the corresponding Platonic configurations. The oriented $S_1$-stabilizer $\Apollonian^+_1$ is the index 2 kernel of the determinant on the $S_1$-stabilizer $\Apollonian_1$, consisting of elements that can be, and can only be, written as even-length words in the generators $\calS_1=\{\bsS_{1jk\ell}\}$ of $\Apollonian_1$; hence, $\Apollonian^+_1$ is generated by $\{\bsS_{1jk\ell}\bsS_{1j'\!k'\!\ell'\!} \mid \bsS_{1jk\ell}, \bsS_{1j'\!k'\!\ell'\!} \in \calS_1\}$, which can easily be reduced to the set $\calS^+_1$ consisting of the 7 matrices
\begin{align*}
\begin{array}{lll}
\bsS_{238}:=\bsS_{1234}\bsS_{1238}, &
\bsS_{274}:=\bsS_{1234}\bsS_{1274}, &
\bsS_{634}:=\bsS_{1234}\bsS_{1634}, \\
\bsS_{278}:=\bsS_{1234}\bsS_{1278}, &
\bsS_{638}:=\bsS_{1234}\bsS_{1638}, &
\bsS_{674}:=\bsS_{1234}\bsS_{1674}, \\ &
\bsS_{678}:=\bsS_{1234}\bsS_{1678},
\end{array}
\end{align*}
which are given explicitly by
{\small
\begin{align*}
\begin{array}{lll} \vspace{2mm}
\bsS_{238}=\medpmatrix{
1&0&0&0&0\\
0&1&0&0&0\\
0&0&1&0&0\\
2&2&2&-1&0\\
2&2&2&0&-1
}, &
\bsS_{274}=\medpmatrix{
1&0&0&0&0\\
0&1&0&0&0\\
2&2&-1&2&0\\
0&0&0&1&0\\
2&2&0&2&-1
}, &
\bsS_{634}=\medpmatrix{
1&0&0&0&0\\
2&-1&2&2&0\\
0&0&1&0&0\\
0&0&0&1&0\\
2&0&2&2&-1
}, \\ \vspace{2mm}
\bsS_{278}=\medpmatrix{
1&0&0&0&0\\
0&1&0&0&0\\
2&2&-1&-2&4\\
2&2&-2&-1&4\\
4&4&-2&-2&5
}, &
\bsS_{638}=\medpmatrix{
1&0&0&0&0\\
2&-1&2&-2&4\\
0&0&1&0&0\\
2&-2&2&-1&4\\
4&-2&4&-2&5
}, &
\bsS_{674}=\medpmatrix{
1&0&0&0&0\\
2&-1&-2&2&4\\
2&-2&-1&2&4\\
0&0&0&1&0\\
4&-2&-2&4&5
}, \\&
\bsS_{678}=\medpmatrix{
1&0&0&0&0\\
2&-1&-2&-2&8\\
2&-2&-1&-2&8\\
2&-2&-2&-1&8\\
6&-4&-4&-4&19
}.
\end{array}
\end{align*}
}

Let $\scrP$ be a primitive orthoplicial Apollonian packing. Choose a sphere $S \in \scrP$, with the bend $b=b(S)$, and a configuration $\scrV$ containing $S$ as the first sphere, with the bend vector $\bsb=\bsb(\scrV)=(b, b_2, b_3, b_4, b_\mu)^\trans$. We shall study the set $\scrB(\scrP)$ by looking at the orbit $\Apollonian^+_1\bsb$ of this initial bend vector $\bsb$. Any bend vector $\bsb' \in \Apollonian^+_1\bsb$ is of the form $\bsb'=(b,b'_2,b'_3,b'_4,b'_\mu)$ and, by \hyperref[thm:DGM]{Theorem~\ref*{thm:DGM}}, lies on a section of the cone defined by the orthoplicial Descartes form $F$. Although the full Apollonian group $\Apollonian$ and the oriented Apollonian group $\Apollonian^+$ are intractable thin groups, the oriented $S_1$-stabilizer $\Apollonian^+_1$ admits an affine parametrization. Adapting the ideas of Sarnak \cite{S:Letter}, we establish this fact in two steps below. 

As the first step, we apply a suitable change of variables so that $\Apollonian^+_1<\SOF(\bbZ)$ is conjugated to a subgroup $\hat \Apollonian^+_1<\SODelta(\bbZ)$, preserving the discriminant
\begin{align} \label{eqn:HermDeltaDefn}
\varDelta(H)=\varDelta(A,B,C,D):=B^2+C^2-AD
\end{align}
of binary hermitian forms $H(\bsxi):=\bsxi^\herm \bsH \bsxi$, associated to hermitian matrices
\[
\bsH:=\medpmatrix{
A&B+iC\\
B-iC&D
},
\]
on $\bsxi=(\alpha,\beta)^\trans$ explicitly by
\begin{align} \label{eqn:Hermitian}
\begin{split}
H(\bsxi):&=A\bar \alpha\alpha+(B+iC)\bar \alpha\beta+(B-iC)\bar \beta\alpha+D\bar \beta\beta\\
&=A|\alpha|^2+2B\Re(\bar \alpha\beta)+2C\Im(\bar \beta\alpha)+D|\beta|^2
\end{split}
\end{align}

\begin{lem} \label{lem:ChangeVar}
The linear change of variables $\hat \bsb=\bsJ \bsb$ from $\bsb=(b,b_2,b_3,b_4,b_\mu)^\trans$ to $\hat \bsb=(b,A,B,C,D)^\trans$, given by  
\[
\bsJ:=
\medpmatrix{
1&0&0&0&0\\
1&1&0&0&0\\
-1/2&-1/2&-1/2&-1/2&1\\
-1/2&-1/2&-1/2&1/2&0\\
1&0&1&0&0
},
\]
conjugates $\Apollonian^+_1<\SOF(\bbZ)$ onto $\hat \Apollonian^+_1:=\bsJ \Apollonian^+_1 \bsJ^{-1}<\SODelta(\bbZ)$, preserving the discriminant \hyperref[eqn:HermDeltaDefn]{(\ref*{eqn:HermDeltaDefn})}, embedded in $\SL_5(\bbZ)$ as the lower $4\times4$ minors for our purpose.
\end{lem}

\begin{proof}
We will compute the generators $\bsJ \calS^+_1 \bsJ^{-1}$ directly, but let us first give an exposition on the role of $\bsJ$; in particular, for the time being, fix an integral packing $\scrP$ and a constituent sphere $S$ with the bend $b=b(S)$, and choose a configuration $\scrV$, containing $S$ as the first sphere, with the bend vector $\bsb=\bsb(\scrV)=(b, b_2, b_3, b_4, b_\mu)$. This bend vector $\bsb$ and any other bend vector in its $\Apollonian^+_1$-orbit is in the conic section
 \begin{align} \label{eqn:Fb}
F(b,b_2,b_3,b_4,b_\mu)=2b_\mu^2-2(b+b_2+b_3+b_4)b_\mu+(b^2+b_2^2+b_3^2+b_4^2)=0,
\end{align}
cut out by the fixed bend $b_1=b$ from the cone defined by $F$. To isolate $b$ in the equation $F(b,b_2,b_3,b_4,b_\mu)=0$, we change the variables by
\begin{align} \label{eqn:b-to-h}
h_2=b+b_2, \quad h_3=b+b_3, \quad h_4=b+b_4, \quad h_\mu=b+b_\mu.
\end{align}
Then, as intended, the equation \hyperref[eqn:Fb]{(\ref*{eqn:Fb})} can be rewritten as
\begin{align} \label{eqn:f}
f(h_2,h_3,h_4,h_\mu):=2h_\mu^2-2(h_2+h_3+h_4)h_\mu+(h_2^2+h_3^2+h_4^2)=-2b^2.
\end{align}
In other words, the action on vectors $(b_2,b_3,b_4,b_\mu)^\trans$, given by the lower right $4\times4$ minors of $\Apollonian^+_1$, is conjugated to the action on vectors $(h_2,h_3,h_4,h_\mu)^\trans$, independent of $b$, preserving quaternary quadratic form $f$. Next, in order to rewrite \hyperref[eqn:f]{(\ref*{eqn:f})} with a more familiar quaternary form, we further change the variables by
\begin{align} \label{eqn:h-to-hatb}
A=h_2, \quad B=\frac{-h_2-h_3-h_4+2h_\mu}{2}, \quad C=\frac{-h_2-h_3+h_4}{2}, \quad D=h_3.
\end{align}
Then, the equation \hyperref[eqn:Fb]{(\ref*{eqn:Fb})} can be rewritten as $2(B^2+C^2-AD)=-2b^2$, or equivalently
\begin{align*}
\varDelta(A,B,C,D)=B^2+C^2-AD=-b^2.
\end{align*}
In other words, the action of $\Apollonian^+_1$ on the vectors $(b_2,b_3,b_4,b_\mu)^\trans$ is now conjugated to the action on vectors $(A,B,C,D)^\trans$, independent of $b$, preserving the discriminant $\varDelta(H)=\varDelta(A,B,C,D)$ of binary hermitian forms \hyperref[eqn:Hermitian]{(\ref*{eqn:Hermitian})} as desired.

Let us now compute the above conjugation explicitly, working with the full $5\times5$ matrices rather than the $4\times4$ minors. The action on vectors $\bsb=(b,b_2,b_3,b_4,b_\mu)^\trans$ is conjugated to the action on vectors $\hat \bsb=(b,A,B,C,D)$, preserving
\begin{align*}
\hat F(\hat b):=\hat F(b,A,B,C,D)=2b^2+2(B^2+C^2-AD)=0.
\end{align*}
Combining the changes of variables \hyperref[eqn:b-to-h]{(\ref*{eqn:b-to-h})} and \hyperref[eqn:h-to-hatb]{(\ref*{eqn:h-to-hatb})}, with the first variable unchanged, we have a linear change of variables $\hat \bsb=\bsJ\bsb$ where
\[
\bsJ:=
\medpmatrix{
1&0&0&0&0\\
0&1&0&0&0\\
0&-1/2&-1/2&-1/2&1\\
0&-1/2&-1/2&1/2&0\\
0&0&1&0&0
}
\medpmatrix{
1&0&0&0&0\\
1&1&0&0&0\\
1&0&1&0&0\\
1&0&0&1&0\\
1&0&0&0&1
}
=\medpmatrix{
1&0&0&0&0\\
1&1&0&0&0\\
-1/2&-1/2&-1/2&-1/2&1\\
-1/2&-1/2&-1/2&1/2&0\\
1&0&1&0&0
}.
\]
Hence, $\bsJ$ conjugates the action of $\Apollonian^+_1<\SOF(\bbZ)$ on vectors $\bsb=(b,b_2,b_3,b_4,b_\mu)^\trans$ to the action of $\hat \Apollonian^+_1=\bsJ \Apollonian^+_1 \bsJ^{-1}<\SOhatF(\bbQ)$ on vectors $\hat \bsb=(b,A,B,C,D)^\trans$. Finally, we note that $\hat\Apollonian^+_1$ is generated by $\hat\calS^+_1:=\bsJ\calS^+_1\bsJ^{-1}$, consisting of matrices
\begin{align*}
\begin{array}{lll} &
\hat\bsS_{238}:=\bsJ\bsS_{238}\bsJ^{-1}, \\
\hat\bsS_{278}:=\bsJ\bsS_{278}\bsJ^{-1}, &
\hat\bsS_{274}:=\bsJ\bsS_{274}\bsJ^{-1}, &
\hat\bsS_{674}:=\bsJ\bsS_{674}\bsJ^{-1}, \\
\hat\bsS_{638}:=\bsJ\bsS_{638}\bsJ^{-1}, &
\hat\bsS_{634}:=\bsJ\bsS_{634}\bsJ^{-1}, &
\hat\bsS_{678}:=\bsJ\bsS_{678}\bsJ^{-1},
\end{array}
\end{align*}
which are given explicitly by
{\small
\begin{align} \label{eqn:hatS}
\begin{array}{lll} \vspace{2mm} &
\hat\bsS_{238}=\medpmatrix{
1&0&0&0&0\\
0&1&0&0&0\\
0&0&-1&0&0\\
0&0&0&-1&0\\
0&0&0&0&1
}, \\ \vspace{2mm} 
\hat\bsS_{278}=\medpmatrix{
1&0&0&0&0\\
0&1&0&0&0\\
0&2&1&0&0\\
0&0&0&1&0\\
0&4&4&0&1
}, &
\hat\bsS_{274}=\medpmatrix{
1&0&0&0&0\\
0&1&0&0&0\\
0&0&-1&0&0\\
0&-2&0&-1&0\\
0&4&0&4&1
}, &
\hat\bsS_{674}=\medpmatrix{
1&0&0&0&0\\
0&5&4&8&4\\
0&2&1&4&2\\
0&-4&-4&-7&-4\\
0&4&4&8&5
}, \\
\hat\bsS_{638}=\medpmatrix{
1&0&0&0&0\\
0&1&4&0&4\\
0&0&1&0&2\\
0&0&0&1&0\\
0&0&0&0&1
}, &
\hat\bsS_{634}=\medpmatrix{
1&0&0&0&0\\
0&1&0&4&4\\
0&0&-1&0&0\\
0&0&0&-1&-2\\
0&0&0&0&1
}, &
\hat\bsS_{678}=\medpmatrix{
1&0&0&0&0\\
0&5&8&4&4\\
0&4&7&4&4\\
0&-2&-4&-1&-2\\
0&4&8&4&5
}.
\end{array}
\end{align}
}
This shows, in particular, that we indeed have $\hat\Apollonian^+_1<\SOhatF(\bbZ)$ as claimed.
\end{proof}

Our next step is to verify that $\hat\Apollonian^+_1$ is indeed in the orthochronous subgroup $\SODelta^\dagger(\bbZ)<\SODelta(\bbZ)$, and identify its \emph{spin preimage} $\bar\Apollonian^+_1$ in $\PSL_2(\bbC)$ as a congruence subgroup of $\varGamma:=\PSL_2(\Gaussian)$. For our purpose, we employ the \emph{spin homomorphism} $\rho: \SL_2(\bbC) \rightarrow \SODelta(\bbR)$, with the target $\SODelta(\bbR)$ embedded in $\SL_5(\bbR)$ as the lower right $4\times4$ minors:
\begin{align*}
\rho: \hspace{3mm} \SL_2(\bbC) \hspace{2mm} &\longrightarrow \SODelta(\bbR) < \SL_5(\bbR)\\
\begin{pmatrix}
\hspace{1mm}\alpha&\beta\hspace{1mm}\\
\hspace{1mm}\gamma&\delta\hspace{1mm}
\end{pmatrix}
&\longmapsto
\medpmatrix{
1&0&0&0&0\\
0&\bar\alpha\alpha&2\Re(\bar\beta\alpha)
&2\Im(\bar\beta\alpha)&\bar\beta\beta\\
0&\Re(\bar\alpha\gamma)&\Re(\bar\beta\gamma+\bar\delta\alpha)
&\Im(\bar\delta\alpha+\bar\beta\gamma)&\Re(\bar\beta\delta)\\
0&\Im(\bar\alpha\gamma)&\Im(\bar\beta\gamma-\bar\delta\alpha)
&\Re(\bar\delta\alpha-\bar\beta\gamma)&\Im(\bar\beta\delta)\\
0&\bar\gamma\gamma&2\Re(\bar\delta\gamma)
&2\Im(\bar\delta\gamma)&\bar\delta\delta
}.
\end{align*}
The kernel of $\rho$ is $\pm \bsI$, and $\rho$ descends to an injection $\bar \rho:\PSL_2(\bbC) \rightarrow \SODelta(\bbR)$, which maps $\PSL_2(\bbC)$ onto the so-called \emph{orthochronous subgroup} $\SODelta^\dagger(\bbR)<\SODelta(\bbR)$. Restricting this map, we have the \emph{spin isomorphism}
\[
\bar \rho:\PSL_2(\bbC) \rightarrow \SODelta^\dagger(\bbR).
\]

\begin{rem}
There are various conventions on how to write down the spin homomorphism. Our choice reflects the underlying convention that $\SL_2(\bbC)$ acts on the left of column $\bbC^2$-vectors via the transpose, and on the left of hermitian forms via the spin homomorphism above.
\end{rem}

Let us also recall the notion of \emph{congruence subgroups}. Given any non-zero ideal $(q) \subset \Gaussian$, the \emph{principal congruence subgroup} $\varGamma(q)<\varGamma:=\PSL_2(\Gaussian)$ of \emph{level} $q$ is the kernel of the modulo $q$ reduction homomorphism, i.e.\;the subgroup consisting of matrices that are congruent to the identity modulo $q$. A subgroup $\varLambda<\varGamma$ is said to be a congruence subgroup if $\varGamma(q)<\varLambda$ for some non-zero ideal $(q)$, and it is said to be of level $q$ if $(q)$ is the maximal non-zero ideal such that $\varGamma(q)<\varLambda<\varGamma$.

\begin{lem} \label{lem:Spin}
The group $\hat\Apollonian^+_1$ is contained in $\SODelta^\dagger(\bbZ)<\SODelta(\bbZ)$, and its spin preimage $\bar\Apollonian^+_1:=\bar\rho^{-1}(\hat\Apollonian^+_1)$ is the folloiwing non-principal congruence subgroup of level 2:
\begin{align} \label{eqn:Spin}
\left\{ \,
\medbmatrix{\alpha&\beta\\\gamma&\delta} \in \PSL_2(\Gaussian)
\;\right|\left.
\medbmatrix{\alpha&\beta\\\gamma&\delta} \equiv
\medbmatrix{\pm1&0\\0&\pm1} \,\text{or}\,\,
\medbmatrix{\pm i&0\\0&\mp i}  \!\!\!\! \mod{2}
\,\right\}.
\end{align}
\end{lem}

\begin{proof}
We define $\bar\calS^+_1 \subset \varGamma$ to be the set of the following seven matrices:
\begin{align} \label{eqn:barS}
\begin{array}{lll} \vspace{2mm} &
\bar\bsS_{238}:=\medbmatrix{
\hspace{0.4mm}i&0\hspace{0.4mm}\\
\hspace{0.4mm}0&-i\hspace{0.4mm}
}, \\ \vspace{2mm} 
\bar\bsS_{278}:=\medbmatrix{
\hspace{0.4mm}1&0\hspace{0.4mm}\\
\hspace{0.4mm}2&1\hspace{0.4mm}
}, &
\bar\bsS_{274}:=\medbmatrix{
\hspace{0.4mm}i&0\hspace{0.4mm}\\
\hspace{0.4mm}2&-i\hspace{0.4mm}
}, &
\bar\bsS_{674}:=\medbmatrix{
\hspace{0.4mm}1+2i&2\hspace{0.4mm}\\
\hspace{0.4mm}2&1-2i\hspace{0.4mm}
}, \\
\bar\bsS_{638}:=\medbmatrix{
\hspace{0.4mm}1&2\hspace{0.4mm}\\
\hspace{0.4mm}0&1\hspace{0.4mm}
}, &
\bar\bsS_{634}:=\medbmatrix{
\hspace{0.4mm}i&2\hspace{0.4mm}\\
\hspace{0.4mm}0&-i\hspace{0.4mm}
}, &
\bar\bsS_{678}:=\medbmatrix{
\hspace{1.3mm}2+i\hspace{0.9mm}&\hspace{0.7mm}2\hspace{1.4mm}\\
\hspace{1.3mm}2\hspace{0.7mm}&\hspace{0.7mm}2-i\hspace{1.4mm}
}.
\end{array}
\end{align}
Direct computations verifies their spin images are the matrices of $\hat\calS^+_1$ shown in \hyperref[eqn:hatS]{(\ref*{eqn:hatS})}:
\begin{align*}
\begin{array}{lll} &
\hat\bsS_{238}=\bar\rho(\bar\bsS_{238}), \\
\hat\bsS_{278}=\bar\rho(\bar\bsS_{278}), &
\hat\bsS_{274}=\bar\rho(\bar\bsS_{274}), &
\hat\bsS_{674}=\bar\rho(\bar\bsS_{674}), \\
\hat\bsS_{638}=\bar\rho(\bar\bsS_{638}), &
\hat\bsS_{634}=\bar\rho(\bar\bsS_{634}), &
\hat\bsS_{678}=\bar\rho(\bar\bsS_{678}).
\end{array}
\end{align*}
Hence, it follows immediately that $\hat\Apollonian^+_1$ is indeed a subgroup of $\SODelta^\dagger(\bbZ)<\SODelta(\bbZ)$, and its spin preimage $\bar\Apollonian^+_1:=\bar\rho^{-1}(\hat\Apollonian^+_1)$ is a subgroup of $\varGamma=\PSL_2(\Gaussian)$, generated by the seven matrices of $\bar\calS^+_1$ defined in \hyperref[eqn:barS]{(\ref*{eqn:barS})} above.

It remains to check that $\bar\Apollonian^+_1$ is indeed the congruence subgroup \hyperref[eqn:Spin]{(\ref*{eqn:Spin})}, which we denote by $\varLambda$ for the time being. Each generator in \hyperref[eqn:barS]{(\ref*{eqn:barS})} satisfies the congruence relations \hyperref[eqn:Spin]{(\ref*{eqn:Spin})}, and thus $\Apollonian^+_1\leqslant\varLambda$. On the other hand, $\varLambda$ is generated by small matrices satisfying the congruence relation \hyperref[eqn:Spin]{(\ref*{eqn:Spin})}. We can compute and list them explicitly, and verify that each of these matrices are indeed in $\bar\Apollonian^+_1$, and thus $\varLambda\leqslant\bar\Apollonian^+_1$:
\begin{align*}
\begin{array}{c}
\begin{matrix} \vspace{2mm} 
\medbmatrix{
\hspace{0.6mm}i&0\hspace{0.6mm}\\
\hspace{0.6mm}0&-i\hspace{0.6mm}
}=\bar\bsS_{238},&
\medbmatrix{
\hspace{0.6mm}1&0\hspace{0.6mm}\\
\hspace{0.6mm}\pm2&1\hspace{0.6mm}
}=\bar\bsS_{278}^{\pm1},&
\medbmatrix{
\hspace{0.6mm}i&0\hspace{0.6mm}\\
\hspace{0.6mm}2&-i\hspace{0.6mm}
}=\bar\bsS_{274}, &
\medbmatrix{
\hspace{0.6mm}i&0\hspace{0.6mm}\\
\hspace{0.6mm}-2&-i\hspace{0.6mm}
}=\bar\bsS_{238}\bar\bsS_{274}\bar\bsS_{238}, \\
\vspace{2mm} &
\medbmatrix{
\hspace{0.6mm}1&\pm2\hspace{0.6mm}\\
\hspace{0.6mm}0&1\hspace{0.6mm}
}=\bar\bsS_{638}^{\pm1}, &
\medbmatrix{
\hspace{0.6mm}i&2\hspace{0.6mm}\\
\hspace{0.6mm}0&-i\hspace{0.6mm}
}=\bar\bsS_{634}, &
\medbmatrix{
\hspace{1.4mm}i\hspace{0.6mm}&-2\hspace{1.4mm}\\
\hspace{1.4mm}0\hspace{0.6mm}&-i\hspace{1.4mm}
}=\bar\bsS_{238}\bar\bsS_{634}\bar\bsS_{238}, \\
\end{matrix}
\\
\begin{matrix} \vspace{2mm}
\hspace{8.5mm}
\medbmatrix{
1&0\hspace{0.6mm}\\
\hspace{0.6mm}\pm2i&1\hspace{0.6mm}
}=(\bar\bsS_{238}\bar\bsS_{274})^{\pm1}, &
\medbmatrix{
\hspace{0.6mm}i&0\hspace{0.6mm}\\
\hspace{0.6mm}\pm2i&-i\hspace{0.6mm}
}=\bar\bsS_{238}\bar\bsS_{278}^{\mp1}, \\
\hspace{8.5mm}
\medbmatrix{
\hspace{0.6mm}1&\pm2i\hspace{0.6mm}\\
\hspace{0.6mm}0&1\hspace{0.6mm}
}=(\bar\bsS_{238}\bar\bsS_{634})^{\mp1}, &
\medbmatrix{
\hspace{1.4mm}i\hspace{0.6mm}&\pm2i\hspace{1.4mm}\\
\hspace{1.4mm}0\hspace{0.6mm}&-i\hspace{1.4mm}
}=\bar\bsS_{238}\bar\bsS_{638}^{\pm1}.
\end{matrix}
\end{array}
\end{align*}
Hence, we conclude $\varLambda=\bar\Apollonian^+_1$ as claimed.
\end{proof}

Since $\bar\Apollonian^+_1$ is a congruence subgroup of $\PSL_2(\Gaussian)$, we can now parametrize its elements by the congruence relations \hyperref[eqn:Spin]{(\ref*{eqn:Spin})}. In particular, we have the following lem about the bends of some spheres in an orthoplicial Apollonian packings.

\begin{lem} \label{lem:b2}
Let $\scrP$ be an orthoplicial Apollonian packing, and let $\scrV \subset \scrP$ be an orthoplicial Platonic configuration with the bend vector $\bsb=(b, b_2, b_3, b_4, b_\mu)^\trans$. Then, for any $\alpha,\beta \in \Gaussian$ satisfying $\alpha \equiv 1$ or $\; i$, $\beta \equiv 0 \pmod{2}$, the number
\begin{align} \label{eqn:b2}
\begin{split}
b'_2=&(|\alpha|^2-\Re(\bar\alpha \beta)-\Im(\bar\beta \alpha)+|\beta|^2-1)\ts b \\
&+(|\alpha|^2-\Re(\bar\alpha \beta)-\Im(\bar\beta \alpha))\ts b_2\\
&+(-\Re(\bar\alpha \beta)-\Im(\bar\beta \alpha)+|\beta|^2)\ts b_3\\
&+(-\Re(\bar\alpha \beta)+\Im(\bar\beta \alpha))\ts b_4\\
&+2\Re(\bar\alpha \beta)\ts b_\mu
\end{split}
\end{align}
appears as the bend of some sphere in the Apollonian packing $\scrP$.
\end{lem}

\begin{proof}
For any $\alpha, \beta \in \Gaussian$ satisfying $\alpha \equiv 1$ or $i$, $\beta \equiv 0 \pmod{2}$, there exists $\gamma, \delta \in \Gaussian$ such that
\[
\bar\bsA:=\medbmatrix{
\hspace{0.4mm}\alpha&\beta\hspace{0.4mm}\\
\hspace{0.4mm}\gamma&\delta\hspace{0.4mm}
} \in \bar\Apollonian^+_1<\PSL_2(\Gaussian).
\]
Then, via \hyperref[lem:ChangeVar]{Lemma~\ref*{lem:ChangeVar}} and \hyperref[lem:Spin]{Lemma~\ref*{lem:Spin}}, we set $\bsA:=\bsJ \bar\rho(\bar\bsA) \bsJ^{-1} \in \Apollonian^+_1<\SL_5(\bbZ)$. The direct computation shows that the second row of $\bsA$ is
\[
\medpmatrix{
|\alpha|^2-\Re(\bar\alpha \beta)-\Im(\bar\beta \alpha)+|\beta|^2-1 \\
|\alpha|^2-\Re(\bar\alpha \beta)-\Im(\bar\beta \alpha) \hspace{16.0mm} \\
\hspace{0.6mm} - \hspace{0.8mm} \Re(\bar\alpha \beta)-\Im(\bar\beta \alpha)+|\beta|^2 \\
- \hspace{0.8mm} \Re(\bar\alpha \beta)+\Im(\bar\beta \alpha) \hspace{9.4mm} \\
2\Re(\bar\alpha \beta) \hspace{23.6mm}
}^{\!\!\trans}.
\]
Hence, given an initial bend vector $\bsb=(b, b_2, b_3, b_4, b_\mu)^\trans$, its $\Apollonian^+_1$-orbit contains the bend vector $\bsA\bsb$, whose second component is precisely $b'_2$ given in \hyperref[eqn:b2]{(\ref*{eqn:b2})}; in other words, the bend of the second sphere in the configuration $\scrV \subset \scrP$ corresponding to $\bsA\bsb$ is precisely $b'_2$ given in \hyperref[eqn:b2]{(\ref*{eqn:b2})}.
\end{proof}

\subsection{The Local-Global Principle}

We now establish our asymptotic \emph{local-global principle} for integral orthoplicial Apollonian packings, addressing the statements (b) and (c) in \hyperref[ssec:LO]{\S\ref*{ssec:LO}} preceding \hyperref[prop:LO]{Proposition~\ref*{prop:LO}}. For this, we translate the question about the set $\scrB(\scrP)$ of integers appearing as bends of spheres in $\scrP$ to a well-studied question about the representability of integers by quadratic forms.

Recall that we have found the change of variables that conjugates $\Apollonian^+_1<\SOF(\bbZ)$, which preserves the orthoplicial Descartes form $F$, to $\hat\Apollonian^+_1<\SODelta(\bbZ)$, which preserves the discriminant $\varDelta$ of a binary hermitian form \hyperref[eqn:Hermitian]{(\ref*{eqn:Hermitian})}. Explicitly, for the initial bend vector $\bsb=(b, b_2, b_3, b_4, b_\mu)^\trans$, we choose $A,B,C,D$ by the change of bases $(b,A,B,C,D)^\trans=\bsJ(b, b_2, b_3, b_4, b_\mu)^\trans$ according to \hyperref[lem:ChangeVar]{Lemma~\ref*{lem:ChangeVar}}, or equivalently by
\begin{align} \label{eqn:ABCD}
\begin{array}{ll}
\displaystyle A:=b+b_2, &
\displaystyle B:=-\frac{b+b_2+b_3+b_4-2b_\mu}{2}, \\
\displaystyle C:=-\frac{b+b_2+b_3-b_4}{2}, &
\displaystyle D:=b+b_3
\end{array}
\end{align}
It should be noted here that $B$ and $C$ are integers since $b+b_2+b_3+b_4$ is always even by \hyperref[prop:LO]{Proposition~\ref*{prop:LO}}. With these $A,B,C,D$, we define the binary hermitian form $H_{\bsb}(\bsxi):=\bsxi^\herm \bsH_{\bsb} \bsxi$, associated to the matrix
\[
\bsH_{\bsb}:=\medpmatrix{
A&B+iC\\
B-iC&D
},
\]
on $\bsxi=(\alpha,\beta)^\trans$ explicitly by
\begin{align*}
\begin{split}
H_{\bsb}(\bsxi):
&=A|\alpha|^2+2B\Re(\bar \alpha\beta)+2C\Im(\bar \beta\alpha)+D\bar |\beta|^2\\
&=(b+b_2)|\alpha|^2-(b+b_2+b_3+b_4-b_\mu)\Re(\bar \alpha\beta)\\
&\hspace{15mm}-(b+b_2+b_3-b_4)\Im(\bar \beta\alpha)+(b+b_3)|\beta|^2.
\end{split}
\end{align*}
Now, writing $\alpha=\alpha_1+i\alpha_2$ with $\alpha_1:=\Re(\alpha)$ and $\alpha_2:=\Im(\alpha)$, $\beta=\beta_1+i\beta_2$ with $\beta_1:=\Re(\beta)$ and $\beta_2:=\Im(\beta)$, and regarding the complex vectors $\bsxi=(\alpha,\beta)^\trans$ as real vectors $\bseta=(\alpha_1,\alpha_2,\beta_1,\beta_2)^\trans$, we can define the corresponding quaternary quadratic form $Q_{\bsb}(\bseta):=H_{\bsb}(\bsxi)$. Namely, we define the quaternary quadratic form $Q_{\bsb}(\bseta):=\bseta^\trans \bsQ_{\bsb} \bseta$, associated to the matrix
\[
\bsQ_{\bsb}:=\medpmatrix{
A&0&B&-C\\
0&A&C&B\\
B&C&D&0\\
-C&B&0&D
}
\]
with $A,B,C,D$ from \hyperref[eqn:ABCD]{(\ref*{eqn:ABCD})}, on $\bseta=(\alpha_1,\alpha_2,\beta_1,\beta_2)^\trans$ explicitly by
\begin{align} \label{eqn:Qb}
\begin{split}
Q_{\bsb}(\bseta):&
=A(\alpha_1^2+\alpha_2^2)+2B(\alpha_1\beta_1+\alpha_2\beta_2)
+2C(\alpha_2\beta_1-\alpha_1\beta_2)+D(\beta_1^2+\beta_2^2)\\
&=(b+b_2)(\alpha_1^2 + \alpha_2^2)
-(b+b_2+b_3+b_4-2b_\mu) (\alpha_1 \beta_1 + \alpha_2 \beta_2)\\
&\hspace{15mm}-(b+b_2+b_3-b_4) (\alpha_2 \beta_1 - \alpha_1 \beta_2)
+(b+b_3) (\beta_1^2 + \beta_2^2)
\end{split}
\end{align}
With the binary hermitian form $H_{\bsb}$ and the quaternary quadratic form $Q_{\bsb}$ above, \hyperref[lem:b2]{Lemma~\ref{lem:b2}} can now be reinterpreted as follows.

\begin{coro} \label{coro:b2}
Let $\scrP$ be an orthoplicial Apollonian packing, and let $\scrV \subset \scrP$ be an orthoplicial Platonic configuration with the bend vector $\bsb=(b, b_2, b_3, b_4, b_\mu)^\trans$. Then, for any $\bsxi=(\alpha,\beta)^\trans$ and $\bseta=(\alpha_1,\alpha_2,\beta_1,\beta_2)^\trans$ with $\alpha,\beta \in \Gaussian$ satisfying $\alpha \equiv 1$ or $\; i$, $\beta \equiv 0 \pmod{2}$, $H_{\bsb}(\bsxi)-b=Q_{\bsb}(\bseta)-b$ coincides with $b'_2$ in  \hyperref[eqn:b2]{(\ref*{eqn:b2})}, and appears as the bend of some sphere in the Apollonian packing $\scrP$.
\end{coro}

\begin{proof}
Rearranging \hyperref[eqn:b2]{(\ref*{eqn:b2})}, we see that the expression of $b'_2$ in \hyperref[eqn:b2]{(\ref*{eqn:b2})} coincides $H_{\bsb}(\bsxi)-b$, and hence with $Q_{\bsb}(\bseta)-b$; the statement then follows from \hyperref[lem:b2]{Lemma~\ref*{lem:b2}}
\end{proof}

Hence, we can study the set $\scrB(\scrP)$ of bends in a primitive orthoplicial Apollonian packing $\scrP$ by investigating the integers represented by the binary hermitian form $H_\bsb(\bsxi)$ and the shifted form $H'_\bsb(\bsxi):=H_{\bsb}(\bsxi)-b$, or equivalently by the quaternary quadratic form $Q_\bsb(\bseta)$ and the shifted form $Q'_\bsb(\bseta):=Q_{\bsb}(\bseta)-b$.

To establish our main result, we utilize the well-known \emph{local-global principle} for quadratic forms. One of the central questions in the theory of quadratic forms asks if and when an integer $n$ can be represented by a given quadratic form $Q$ globally, i.e.\;over $\bbZ$, provided that $n$ is represented by $Q$ locally, i.e.\;over $\bbZp$ for all prime $p$; see \cite{Duke:Survey}, \cite{Hanke:Survey} for surveys of the subject. Kloosterman's work on the circle method yields the satisfactory answer for positive-definite quaternary forms: every (effectively bounded) sufficiently large locally represented integer $n$ with \emph{a priori} bounded divisibility at anisotropic primes can be represented globally. Here, a prime $p$ is said to be \emph{anisotropic} for a quadratic form $Q$ and $Q$ is said to be \emph{anisotropic} at $p$, if the equation $Q(\bseta)=0$ only has the zero solution over $\bbZp$; otherwise, $p$ is said to be \emph{isotropic} for $Q$ and $Q$ is said to be \emph{isotropic} at $p$. Any anisotropic prime always divides the discriminant of the form $Q$. The exposition on the circle method and the local-global principle for quadratic forms can be found in \cite[Thm.\,20.9]{IK}; see also \cite[Thm.\,6.3]{Hanke:Explicit} for an explicit bound.

\begin{lem} \label{lem:DeltaPD}
For any bend vector $\bsb=(b, b_2, b_3, b_4, b_\mu)^\trans$ of an orthoplicial Platonic configuration, the quadratic form $Q_{\bsb}$ is positive-semidefinite and has the discriminant $\varDelta(Q_{\bsb})=(2b)^4$; moreover, it is positive-definite if and only if $b \neq 0$.
\end{lem}

\begin{proof}
We verify the claims by direct computation. First, the discriminant of $Q_{\bsb}$ can be computed explicitly as
\begin{align} \label{eqn:DeltaQb}
\varDelta(Q_{\bsb}):=2^4 \det \bsQ_{\bsb}=2^4 (\det \bsH_{\bsb})^2
=2^4\left(\frac{1}{\,2\,}F(\bsb)-b^2\right)^{\!2}=(2b)^4,
\end{align}
where $F$ is the orthplicial Descartes form \hyperref[eqn:Fzeta]{(\ref*{eqn:Fzeta})} satisfying $F(\bsb)=0$ for any bend vector $\bsb$ by \hyperref[eqn:DGM-diag]{(\ref*{eqn:DGM-diag})}. Next, the characteristic polynomial $\chi_{\bsH_{\bsb}}(\lambda)$ of $\bsH_{\bsb}$ is
\[
\chi_{H_{\bsb}}(\lambda)
=\lambda^2-(2b+b_2+b_3)-\left(\frac{1}{\,2\,}F(\bsb)-b^2\right)
=\lambda^2-(2b+b_2+b_3)+b^2
\]
with eigenvalues
\begin{align} \label{eqn:lambda}
\lambda=\frac{1}{\,2\,} \left(2b+b_2+b_3 \pm \sqrt{(2 b+b_2+b_3)^2-(2b)^2}\right)
\end{align}
and the characteristic polynomial $\chi_{\bsQ_{\bsb}}(\lambda)$ of $\bsQ_{\bsb}$ is
\begin{align*}
\chi_{Q_{\bsb}}(\lambda)
=\big(\chi_{H_{\bsb}}(\lambda)\big)^2
=\big(\lambda^2-(2b+b_2+b_3)-b^2\big)^2
\end{align*}
with the same eigenvalues \hyperref[eqn:lambda]{(\ref*{eqn:lambda})} as $\bsH_{\bsb}$ but with double multiplicities. From the discriminant $(2 b+b_2+b_3)^2-(2b)^2$ of $\chi_{\bsH_\bsb}(\lambda)$, we see that the eigenvalues are real if and only if $b_2+b_3 \geqslant 0$; this is indeed the case for any bend vector, since any orthoplicial configuration $\scrV$ has at most one negative bend, which is necessarily the bend of the largest sphere enclosing all other spheres. Finally, observing that $2b+b_2+b_3 \geqslant \sqrt{(2b+b_2+b_3)^2-(2b)^2}$ is equivalent to $(2b)^2 \geqslant 0$, all eigenvalues \hyperref[eqn:lambda]{(\ref*{eqn:lambda})} are non-negative, and the smaller one vanishes if and only if $b \neq 0$.
\end{proof}

Given a primitive orthoplicial Apollonian packing $\scrP$, let $\varepsilon=\varepsilon(\scrP)$ from \hyperref[prop:LO]{Proposition~\ref*{prop:LO}} and write $\scrA(\scrP):=\{ n \in \bbZ \mid n \not\equiv -\varepsilon \}$. \hyperref[prop:LO]{Proposition~\ref*{prop:LO}} guarantees $\scrB(\scrP) \subset \scrA(\scrP)$. The asymptotic \emph{local-global principle} we are going to establish states that sufficiently large integer $n \in \scrA(\scrP)$ must be in $\scrB(\scrP)$. We first give the following preliminary version of the local-global principle; for \emph{simplicial} Apollonian packings, the analogous statement is given by Kontorovich in \cite[Prop.\,3.26]{K:Soddy}. 

\begin{prop} \label{prop:LG-prelim}
Let $\scrP$ be a primitive orthoplicial Apollonian packing and $b \in \scrB(\scrP)$ such that $b \neq 0$. If $n \in \scrA(\scrP)$ is sufficiently large integer satisfying $\gcd(n,b)=1$, then $n \in \scrB(\scrP)$.
\end{prop}

\begin{proof}
Let $S \in \scrP$ be a constituent sphere with the bend $b=b(S)$. We choose an orthoplicial configuration $\scrV$ in $\scrP$, containing $S$ as the first sphere, with the bend vector $\bsb=(b, b_2, b_3, b_4, b_\mu)^\trans$. Let $Q_{\bsb}$ be the quaternary quadratic form \hyperref[eqn:Qb]{(\ref*{eqn:Qb})}, i.e.\;defined on $\bseta=(\alpha_1,\alpha_2,\beta_1,\beta_2)^\trans$ by
\begin{align*}
Q_{\bsb}(\bseta):&=(b+b_2)(\alpha_1^2+\alpha_2^2)-(b+b_2+b_3+b_4-2b_\mu) (\alpha_1 \beta_1 + \alpha_2 \beta_2)\\
&\hspace{15mm}-(b+b_2+b_3-b_4) (\alpha_2 \beta_1 - \alpha_1 \beta_2)+(b+b_3)(\beta_1^2 + \beta_2^2).
\end{align*}
Note that, since $\bsb$ is a bend vector, it follows from \hyperref[prop:LO]{Proposition~\ref*{prop:LO}} that $B$ and $C$ are integers and $2B=b+b_2+b_3+b_4-2b_\mu$ and $2C=b+b_2+b_3-b_4$ are even.

We assume $\alpha=\alpha_1+i\alpha_2,\beta=\beta_1+i\beta_2 \in \Gaussian$ satisfying the congruence conditions $\alpha \equiv 1$ or $\; i$, $\beta \equiv 0 \pmod{2}$, i.e.\;$\alpha_1,\alpha_2$ must have opposite parities and $\beta_1,\beta_2$ are both even. It follows that we have $\alpha_1 \beta_1 + \alpha_2 \beta_2\equiv\alpha_2 \beta_1 - \alpha_1 \beta_2 \equiv 0\pmod{2}$, $\alpha_1^2+\alpha_2^2 \equiv 1\pmod{4}$, and $\beta_1^2 + \beta_2^2\equiv0\pmod{4}$. Reducing $Q_\bsb$ modulo 4, we obtain
\begin{align} \label{eqn:b2mod4}
Q_{\bsb}(\bseta)\equiv b+b_2 \pmod{4}
\end{align}
Note that this is the only local obstruction for $Q_\bsb(\bseta)$; for any odd prime $p$, we can choose $\alpha_1\equiv\alpha_2 \equiv 0 \pmod{p}$ and vary $\beta_1,\beta_2$ over the entire $\bbZp$, so that $Q_\bsb(\bseta) \equiv \beta_1^2+\beta_2^2$ ranges over the entire $\bbZp$.

Let $n \in \scrA(\scrP)$ be an integer satisfying $\gcd(b,n)=1$ and set $n':=b+n$. Rearranging the ordering on $\scrV$ with the action of $\Platonic$ if necessary, we may assume by \hyperref[prop:LO]{Proposition~\ref*{prop:LO}} that the bend vector $\bsb$ satisfies $b_2 \equiv n \pmod{4}$; then $n' \equiv b+n \pmod{4}$ is locally represented by $Q_\bsb$, cf.\;\hyperref[eqn:b2mod4]{(\ref*{eqn:b2mod4})}. We write $\varepsilon=\varepsilon(\scrP)$ from \hyperref[prop:LO]{Proposition~\ref*{prop:LO}}, so that $n \not\equiv -\varepsilon \pmod{4}$. Now, we observe two consequences of $\gcd(b,n)=1$. First, it immediately follows that $\gcd(b,n')=1$. Second, ruling out the even cases, we have the multi-set congruence $\{b,n\} \equiv \{0,\varepsilon\}, \{2,\varepsilon\}$, or $\{\varepsilon,\varepsilon\} \pmod{4}$. In the first two cases, we have $n' \equiv \pm \varepsilon \pmod{4}$, so $\gcd(2,n')=1$. In the last case, we have $n' \equiv 2 \pmod{4}$, which means that $n'$ is divisible by 2 exactly once. Hence, combining these observations, we see that $n'$ has bounded divisibility at prime divisors of the discriminant, $\varDelta(Q_\bsb)=(2b)^4$ given by \hyperref[lem:DeltaPD]{Lemma~\ref{lem:DeltaPD}}.

The form $Q_\bsb$ is positive-definite, also by \hyperref[lem:DeltaPD]{Lemma~\ref{lem:DeltaPD}}, and $n'$ is locally represented by $Q_\bsb$ with bounded divisibility at prime divisors of $\varDelta(Q_\bsb)$; hence, by the Kloosterman's work discussed in the paragraph preceding \hyperref[lem:DeltaPD]{Lemma~\ref{lem:DeltaPD}}, there exists $N=N(Q_\bsb)$ such that, if $n'>N$, then $n'$ is globally represented by $Q_\bsb$. Hence, if $n$ is sufficiently large so that $n'$ is sufficiently large, then $n'$ is represented by the form $Q_\bsb$ and $n$ is represented by the shifted form $Q'_\bsb$. Finally, it then follows that $n$ appears in $\scrB(\scrP)$ by \hyperref[coro:b2]{Corollary~\ref*{coro:b2}}.
\end{proof}

It is crucial to note that \hyperref[prop:LG-prelim]{Proposition~\ref*{prop:LG-prelim}} alone does not readily imply the local-global principle, even for primitive orthoplicial Apollonian packings. From \hyperref[prop:LG-prelim]{Proposition~\ref*{prop:LG-prelim}}, we can deduce that, given a primitive orthoplicial Apollonian packing $\scrP$ and a \emph{finite} collection $\scrS$ of constituent spheres in $\scrP$, there exists a number $N(\scrS)$ such that any integer $n \in \scrA(\scrP)$ satisfying $n>N$ and coprime to each bend in $\scrB(\scrS)$ is represented in $\scrB(\scrP)$. However, there are an infinite number of integers that are not coprime to any of the bends in $\scrB(\scrS)$, e.g. multiples of products of the bends of spheres in $\scrS$; the primitivity of $\scrP$ only means that the bend of a sphere in $\scrP$ is coprime to the bend of \emph{some} sphere in $\scrP$, but not necessarily a sphere in $\scrS$. Enlarging the collection $\scrS$ to cover more integers is futile, and we may end up enlarging the bound $N(\scrS)$ indefinitely. We remark that the same issue arise in simplicial Apollonian packings, but it seems to be overlooked in \cite{K:Soddy}.

We must use another subtle property of the quadratic form $Q_\bsb$ to strengthen \hyperref[prop:LG-prelim]{Proposition~\ref*{prop:LG-prelim}}, so that we can obtain a single bound $N(\scrP)$ up front. The next lemma serves this purpose by removing the need to impose the coprimitive condition on the integer $n$ all together.

\begin{lem} \label{lem:Isotropic}
For any bend vector $\bsb=(b,b_2,b_3,b_4,b_\mu)^\trans$ of any orthoplicial Platonic configuration $\scrV$, the quadratic form $Q_\bsb$ is isotropic at every prime.
\end{lem}

\begin{proof}
Let $p$ be a prime. If $p$ does not divide $b$, then $p$ does not divide the discriminant $\varDelta(Q_\bsb)=(2b)^4$, given by \hyperref[lem:DeltaPD]{Lemma~\ref*{lem:DeltaPD}}, and hence $Q_\bsb$ is isotropic at $p$. So, we may assume that $b \equiv 0 \pmod{p}$.

We need to find a non-zero vector $\bseta \in (\bbZp)^4$ satisfying $Q_\bsb(\bseta)\equiv0\pmod{p}$. Reducing modulo $p$, the quadratic form $Q_\bsb(\bseta)$ is associated to the matrix
\[
\bsQ_{\bsb}:\equiv\medpmatrix{
A&0&B&-C\\
0&A&C&B\\
B&C&D&0\\
-C&B&0&D
},
\]
where $A,B,C,D \pmod{p}$ are given by
\begin{align*}
\begin{array}{ll}
\displaystyle A:\equiv b_2, &
\displaystyle B:\equiv -\frac{b_2+b_3+b_4-2b_\mu}{2}, \\
\displaystyle C:\equiv -\frac{b_2+b_3-b_4}{2}, &
\displaystyle D:\equiv b_3.
\end{array}
\end{align*}
In order to find non-zero solutions $\bseta \in (\bbZp)^4$ for $Q_\bsb(\bseta)\equiv0\pmod{p}$, let us recall the degenerate case in \hyperref[lem:DeltaPD]{Lemma~\ref{lem:DeltaPD}}. There, when $b=0$, two eigenvalues of $\bsQ_\bsb$ degenerate to 0; noting that $b=0$ forces $B^2+C^2-AD=0$, we can quickly verify that the corresponding eigenvectors in $\bbZ^4$ are
\begin{align*}
\bseta_1=(C,-B,0,A)^\trans, \qquad \bseta_2=(-B,-C,A,0)^\trans.
\end{align*}
Reducing these vectors modulo $p$, we still have $Q_\bsb(\bseta_i)\equiv 0 \pmod{p}$; we remark that $\bseta_i \pmod{p}$ are both zero or both non-zero. If $\bseta_i \pmod{p}$ are non-zero, we are done. If $\bseta_i \pmod{p}$ happen to be zero, i.e. $A \equiv B \equiv C \equiv 0 \pmod{p}$, then the matrix $\bsQ_\bsb \pmod{p}$ above is highly degenerate, and we have plenty of non-zero vectors, e.g. $\bseta' \equiv(1,0,0,0) \pmod{p}$, satisfying $Q_\bsb(\bseta')\equiv 0 \pmod{p}$.
\end{proof}

\begin{rem}
We checked in the proof of \hyperref[prop:LG-prelim]{Proposition~\ref*{prop:LG-prelim}} that $n'$ has a bounded divisibility at 2 in order to show that the prime factor 2 of $n'$ does not obstruct the representability of $n'$. In the hindsight, we can also argue that 2 is an isotropic prime and hence it does not obstruct the representability of $n'$.
\end{rem}

\begin{thm} \label{thm:LG}
Every primitive orthoplicial Apollonian sphere packing $\scrP$ satisfy the asymptotic local-global principle: there is an effectively and explicitly computable bound $N=N(\scrP)$ so that, if $n>N$ and $n \in \scrA(\scrP)$, then $n \in \scrB(\scrP)$.
\end{thm}

The proof is almost identical to that of \hyperref[prop:LG-prelim]{Proposition~\ref*{prop:LG-prelim}}, with the coprimitivity condition $\gcd(b,n)=1$ removed; \hyperref[lem:Isotropic]{Lemma~\ref*{lem:Isotropic}} still allows us to deduce the global representability. We spell out the proof below for the completeness.

\begin{proof}
Fix an Apollonian packing $\scrP$, and choose a Platonic configuration $\scrV$ in $\scrP$. We write $\varepsilon=\varepsilon(\scrP)$ as in \hyperref[prop:LO]{Proposition~\ref*{prop:LO}} and $\scrA(\scrP)=\{ n \in \bbZ \mid n \not\equiv -\varepsilon \}$ as before. For any bend vector $\bsb=(b,b_2,b_3,b_4,b_\mu)^\trans$ of $\scrV$, we have the quaternary quadratic form $Q_{\bsb}$ form \hyperref[eqn:Qb]{(\ref*{eqn:Qb})}, i.e.\;defined on $\bseta=(\alpha_1,\alpha_2,\beta_1,\beta_2)^\trans$ by
\begin{align*}
Q_{\bsb}(\bseta):&=(b+b_2)(\alpha_1^2+\alpha_2^2)-(b+b_2+b_3+b_4-2b_\mu) (\alpha_1 \beta_1 + \alpha_2 \beta_2)\\
&\hspace{15mm}-(b+b_2+b_3-b_4) (\alpha_2 \beta_1 - \alpha_1 \beta_2)+(b+b_3)(\beta_1^2 + \beta_2^2).
\end{align*}
We assume $\alpha=\alpha_1+i\alpha_2,\beta=\beta_1+i\beta_2 \in \Gaussian$ satisfying the congruence conditions $\alpha \equiv 1$ or $\; i$, $\beta \equiv 0 \pmod{2}$. Then, as we have seen in the proof of \hyperref[prop:LG-prelim]{Proposition~\ref*{prop:LG-prelim}},
\begin{align} \label{eqn:b2mod4again}
Q_{\bsb}(\bseta)\equiv b+b_2 \pmod{4},
\end{align}
and this is the only local obstruction for $Q_\bsb(\bseta)$.

Let $n \in \scrA(\scrP)$. We now choose an admissible ordering on $\scrV$ and the corresponding bend vector $\bsb=(b,b_2,b_3,b_4,b_\mu)^\trans$ such that $n \equiv b_2\pmod{4}$; we can always choose such an ordering by \hyperref[prop:LO]{Proposition~\ref*{prop:LO}} and the primitivity, cf.\;\hyperref[lem:IntegralApollonian]{Lemma~\ref*{lem:IntegralApollonian}}. Then, for this bend vector $\bsb$, the form $Q_\bsb$ is positive-definite by \hyperref[lem:DeltaPD]{Lemma~\ref{lem:DeltaPD}} and isotropic at every prime by \hyperref[lem:Isotropic]{Lemma~\ref*{lem:Isotropic}}.

Set $n':=b+n$. Then, $n' \equiv b+n \equiv b+b_2 \pmod{4}$ is locally represented by $Q_\bsb$, cf.\;\hyperref[eqn:b2mod4again]{(\ref*{eqn:b2mod4again})}. Hence, by Kloosterman's work discussed in the paragraph preceding \hyperref[lem:DeltaPD]{Lemma~\ref{lem:DeltaPD}}, there exists $N=N(Q_\bsb)$ such that, if $n'>N$, then $n'$ is globally represented by $Q_\bsb$. Hence, if $n$ is sufficiently large so that $n'$ is sufficiently large, then $n'$ is represented by the form $Q_\bsb$ and $n$ is represented by the shifted form $Q'_\bsb$. Finally, it then follows that $n$ appears in $\scrB(\scrP)$ by \hyperref[coro:b2]{Corollary~\ref*{coro:b2}}.
\end{proof}

%%%%%
\appendix

%%%
\section*{Appendix. Orthoplicial Dual Apollonian Group}

In this article, we presented two examples of bounded primitive orthoplicial Apollonian packings, $\scrP_1$, $\scrP_{7\rmd}$, generated from Platonic configurations $\scrV_1$, $\scrV_{7\rmd}$ in \hyperref[exa:V7d]{Example~\ref*{exa:V7d}}. These configurations are found in the orbit of the standard configuration $\scrV_0$ under the action of the orthoplicial \emph{dual Apollonian group}.

\begin{defn}
The \emph{orthoplicial dual Apollonian group} $\dualApollonian$ is defined to be the $5 \times 5$ matrix group generated by $\calS_*:=\{\bsS_k\}$, consisting of the following 8 matrices:
{\small
\[
\begin{matrix} \vspace{2mm}\hspace{1mm}
\bsS_1=\medpmatrix{-1&0&0&0&0\\2&1&0&0&0\\2&0&1&0&0\\2&0&0&1&0\\2&0&0&0&1},\hspace{2.3mm}
\bsS_2 = \medpmatrix{1&2&0&0&0\\0&-1&0&0&0\\0&2&1&0&0\\0&2&0&1&0\\0&2&0&0&1},\hspace{2.3mm}
\bsS_3 = \medpmatrix{1&0&2&0&0\\0&1&2&0&0\\0&0&-1&0&0\\0&0&2&1&0\\0&0&2&0&1},\hspace{2.3mm}
\bsS_4 = \medpmatrix{1&0&0&2&0\\0&1&0&2&0\\0&0&1&2&0\\0&0&0&-1&0\\0&0&0&2&1},\hspace{2.3mm}\\
\bsS_5 = \medpmatrix{-5&0&0&0&12\\-2&1&0&0&4\\-2&0&1&0&4\\-2&0&0&1&4\\-2&0&0&0&5},\;
\bsS_6 = \medpmatrix{1&-2&0&0&4\\0&-5&0&0&12\\0&-2&1&0&4\\0&-2&0&1&4\\0&-2&0&0&5},\;
\bsS_7 = \medpmatrix{1&0&-2&0&4\\0&1&-2&0&4\\0&0&-5&0&12\\0&0&-2&1&4\\0&0&-2&0&5},\;
\bsS_8 = \medpmatrix{1&0&0&-2&4\\0&1&0&-2&4\\0&0&1&-2&4\\0&0&0&-5&12\\0&0&0&-2&5}.
\end{matrix}
\]}
\end{defn}

One can check that each matrix in the orbit $\dualApollonian\bsF_0$ of the standard configuration $\scrV_0$ is an $F$-matrix for some Platonic configuration. $\bsF_1$ and $\bsF_{7\rmd}$ of the configurations $\scrV_1,\scrV_{7\rmd}$ are found in the orbit $\dualApollonian\bsF_0$. Indeed, the $F$-matrices in the orbit $\dualApollonian\bsF_0$ give rise to infinite number of inequivalent primitive orthoplicial Apollonian packings. A full exposition on this fact will be given elsewhere \cite{N:BC4GR}.

%%%%%
% Back Matters

\bibliography{ASPBC4LG}
\bibliographystyle{amsalpha}

%%%%%
\end{document}